\RequirePackage{snapshot}

\newif\ifsubsections

\documentclass[]{pcmi}

\usepackage[sc]{mathpazo}          
\usepackage{eulervm}               
\usepackage[scaled=0.86]{berasans} 
\usepackage[scaled=1]{inconsolata} 
\usepackage[T1]{fontenc}

\usepackage[%
	protrusion=true,
	expansion=false,
	auto=false
	]{microtype}

\usepackage{xcolor}
\usepackage{graphicx}
\graphicspath{{./figures/}}

\ifdraft
	\definecolor{linkred}{rgb}{0.7,0.2,0.2}
	\definecolor{linkblue}{rgb}{0,0.2,0.6}
\else
	\definecolor{linkred}{rgb}{0.0,0.0,0.0}
	\definecolor{linkblue}{rgb}{0,0.0,0.0}
\fi

\usepackage{amssymb,amsfonts,amsthm}
\usepackage[makeroom]{cancel}
\usepackage{mathtools}
\usepackage{pifont}
\usepackage{mathrsfs}
\usepackage{slashed,mathabx} 
\usepackage[bbgreekl]{mathbbol}
\usepackage{tikz-cd}
\usepackage{anyfontsize}

\usepackage{enumitem}
\usepackage[all]{xy}
\usepackage{comment}

\PassOptionsToPackage{hyphens}{url} 
\usepackage[
    setpagesize=false,
    pagebackref,
	pdfpagelabels=false,
    pdfstartview={FitH 1000},
    bookmarksnumbered=false,
    linktoc=all,
    colorlinks=true,
    anchorcolor=black,
    menucolor=black,
    runcolor=black,
    filecolor=black,
    linkcolor=linkblue,
	citecolor=linkblue,
	urlcolor=linkred,
]{hyperref}
\usepackage[backrefs,msc-links,nobysame]{amsrefs}

\customizeamsrefs 

\theoremstyle{plain}
\newtheorem{Proposition}[equation]{Proposition}
\newtheorem{Lemma}[equation]{Lemma}
\newtheorem{Corollary}[equation]{Corollary}
\newtheorem{Theorem}[equation]{Theorem}
\newtheorem{Conjecture}[equation]{Conjecture}

\theoremstyle{definition}
\newtheorem{Definition}[equation]{Definition}

\newtheorem{Example}[equation]{Example}

\theoremstyle{remark}
\newtheorem{Remark}[equation]{Remark}

\usepackage{macros}
\begin{document}

%
%
%
%
%
%

\title[Cutoff for biased random transpositions]{Cutoff for the Biased Random Transposition Shuffle} 

%
%
\author{Evita Nestoridi, Alan Yan}
\date{} 
\address{}
\email{}

\begin{abstract}
    In this paper, we study the biased random transposition shuffle, a natural generalization of the classical random transposition shuffle studied by Diaconis and Shahshahani. We diagonalize the transition matrix of the shuffle and use these eigenvalues to prove that the shuffle exhibits total variation cutoff at time $t_N = \frac{1}{2b} N \log N$ with window $N$. We also prove that the limiting distribution of the number of fixed cards near the cutoff time is Poisson. 
\end{abstract}

%
%
\maketitle

\tableofcontents
%
%


\section{Introduction} 
Studying card shuffling with the help of representation theory was initially introduced by Diaconis and Shahshahani for the study of the random transpositions shuffle. Consider a deck of $N$ distinct cards. Pick two cards uniformly at random with repetition and swap them. Diaconis and Shahshahani proved that it takes $\frac{1}{2} N \log N$ repetitions until the deck is shuffled sufficiently well. This model is now famously known as the \emph{random transpositions shuffle}, and to this day, it is considered the best approximation of the Markov chain that genes in DNA sequences follow~\cite{BeDu}. Following Diaconis and Shahshahani, there has been a variety of other card shuffles that have been studied with the use of representation theory (see~\cites{Flatto1985, Hough2016, Dieker2018, BernsteinNestoridi2019, BateConnorMatheau-Raven2021,NestoridiPeng2021, Bidigare-Hanlon-Rockmore-1999, BrownDiaconis1998}).

In this paper, we consider a generalization of random transpositions, where we partition the cards into two sets and pick cards with probabilities depending on which half of the deck they belong. More precisely, consider a deck of $N=2n$ distinct cards and denote $[N]:= \{1, \ldots, N\}$. We partition the cards into two sets $[N] = A \sqcup B$ where $|A| = |B| = n$. Fix numbers $0 < b \leq a$ satisfying $a + b = 2$. Let $\mu_{a, b}$ be the probability measure on $[N]$ given by 
\[
    \mu_{a, b}(x) = \begin{cases}
        \frac{a}{N} & \text{ if } x \in A \\
        \frac{b}{N} & \text{ if } x \in B. 
    \end{cases}
\]
We pick two cards uniformly at random with repetition according to $\mu_{a, b}$ and swap them. By identifying all of the configurations of our deck with elements in $\fS_N$, we can think of the shuffle as a random walk on $\fS_N$. 
We refer to it as the \textbf{biased random transpositions shuffle}, or the $\bRT$ shuffle for short, to connect it back to the classical case of random transpositions considered in~\cite{Diaconis1981}. Indeed, in the case $a=b=1$, we recover the original random transposition shuffle.  

For $x \in \fS_N$, let $P^t(x, \bullet)$ be the distribution of our deck configuration after $t$ shuffles. We prove that for the case $|A|=|B|$, this Markov chain exhibits total variation distance cutoff at $\frac{1}{2b} N \log N$ by diagonalizing the transition matrix $P=(P(x, y))_{x,y \in \mathfrak{S}_N}$ and using this spectral information to produce lower and upper bounds for the mixing time. We are ready to state our first result.

\begin{Theorem} \label{thm:main-theorem-1}
    Let $|A| = |B| = n$. Let $c > 0$ be a positive real number. For sufficiently large $N$ and some universal constant $C > 0$, we have
     \begin{align*}
        d_{\TV} \left ( P^{*\frac{N}{2b}(\log N + c)}(\id, \bullet), U \right ) & \leq C \cdot e^{-c} 
        \textup{ and } \\
        d_{\TV} \left ( P^{*\frac{N}{2b}(\log N - c)}(\id, \bullet), U \right ) & \geq 1 - e^{- \frac{1}{2} \left ( \sqrt{1 + \frac{1}{2} e^c} - 1 \right )^2} + o(1),
    \end{align*}
    where $U$ is the uniform measure on $\mathfrak{S}_N$.
\end{Theorem}

Theorem \ref{thm:main-theorem-1} describes the occurrence of the cutoff phenomenon.
A family of Markov chains on $\mathfrak{S}_N$ is said to have cutoff at $t_N$ with window $w_N=o(t_N)$ if
$$\lim_{c \rightarrow \infty} \lim_{N \rightarrow \infty} d(t_N-cw_N)= 1 \quad\mbox{ and }\quad \lim_{c \rightarrow \infty} \lim_{N \rightarrow \infty} d(t_N+cw_N)= 0.$$ 
The occurrence of cutoff is a central question in Markov chain mixing, see more in \cite{Salez}. The first instances of cutoff were in card shuffling in works by Aldous, Diaconis and Shahshahani \cites{Aldous, Diaconis1981}, and since then cutoff has been studied in terms of many other card shuffles \cites{Bayer, Bese, BernsteinNestoridi2019, Hough2016, Lacoin}. 

Biased transpositions was also studied in \cite{Bernstein2017}, where the authors focused on separation distance. The proof relied on analyzing the behavior of a stopping time, that unfortunately turned out not to be a strong stationary time as explained in Section 5.3 of \cite{Graham}. In this paper, we follow a spectral approach to prove the behavior that was predicted in \cite{Bernstein2017}.

The upper bound of Theorem \ref{thm:main-theorem-1} is proven by an $\ell_2$ bound, which makes use of the eigenvalues of the transition matrix $P$. Our second theorem discusses the spectrum of $P$. The statement relies on a few statistics on partitions.

For a positive integer $N$, a \textbf{partition} $\lambda=(\lambda_1, \lambda_2, \ldots)$ of size $N$ is a \emph{weakly decreasing} sequence of non-negative integers which sum to $N$. We write $\lambda \vdash N$ for a partition of $N$. We can represent partitions diagrammatically with their Young diagrams. Let $\lambda^*$ be the partition corresponding to the transpose diagram of $\lambda$.
Let the \textbf{diagonal index} of a partition $\lambda$ be
\[
    \Diag(\lambda) \eqdef \sum_{i \geq 0} \binom{\lambda_i}{2} - \sum_{i \geq 0} \binom{\lambda_i^*}{2}. 
\]
Then, the eigenvalue spectrum of $P$ is given by the following Theorem. All definitions on partitions are given in Section \ref{sec:prelim}. 
\begin{Theorem}\label{thm:main-theorem-3}
    The transition matrix $P$ has eigenvalues 
    $$\frac{a^2 |A| + b^2 |B|}{2N} + \frac{2(a^2 - ab)}{N^2} \Diag(\mu) + \frac{2(b^2 -ab)}{N^2} \Diag(\nu) + \frac{2ab}{N^2} \Diag(\lambda), $$ 
    with multiplicities $f_\lambda f_\mu f_\nu c^\lambda_{\mu, \nu}$ for all $\lambda \vdash N$, $\mu \vdash |A|$, and $\nu \vdash |B|$ partitions, where $f_\lambda $ is the number of Standard Young Tableaux of shape $\lambda$ and  $c^\lambda_{\mu, \nu}$ is the Littlewood Richardson coefficient.
\end{Theorem}

The lower bound of Theorem \ref{thm:main-theorem-1} is proven by studying the number of fixed cards at time $t$. Let $\Fix_c$ be the number of fixed points after shuffling our deck of cards using the biased random transposition shuffle $\frac{1}{2b} N (\log N - c)$ times. Let $\Fix$ be the number of fixed points of a uniformly picked permutation. The random variable $\Fix_c$ has the following limiting distribution. 

\begin{Theorem}\label{thm:main-theorem-4'}
    Suppose that $b < 1$. Then, 
    \begin{equation}\label{eqn:fix-converges-poisson-yes}
        \Fix_c \xrightarrow{dist} \Poiss \left (1 + \frac{1}{2}e^c \right ), \quad \text{ as } N \to \infty. 
    \end{equation}
    Also,
    \begin{equation}\label{eqn:tv-converges-yes}
        d_{\TV} \left ( \Fix_c , \Fix \right ) \longrightarrow d_{\TV} \left ( \Poiss(1), \Poiss \left (1 + \frac{1}{2}e^c \right ) \right ), \quad \text{ as } N \to \infty. 
    \end{equation}
\end{Theorem}
Note that Equation~\ref{eqn:tv-converges-yes} follows immediately from Equation~\ref{eqn:fix-converges-poisson-yes} by definition of convergence in distribution. Based on Theorem \ref{thm:main-theorem-4'}, we formulate the following conjecture concerning the limit profile of the biased transpositions card shuffle. 

\begin{Conjecture}[Limit profile for biased random transpositions] \label{conjecture:limit-profile}
    Let $c \in \RR$ and $b < 1$. Then we have 
    \[
        d_{\TV} \left ( P^{\frac{1}{2b}N(\log N - c)}, U \right ) \to d_{\TV} \left ( \Poiss \left (1 + \frac{e^c}{2} \right ), \Poiss(1) \right ), \quad \text{as } N \to \infty.
    \]
\end{Conjecture}

The limit profile is a more precise statement than cutoff, and we know the limit profile for only a few famous Markov chains (see~\cites{LP, Bayer}). In~\cite{Teyssier2019}, Teyssier proved that the limit profile for random transpositions (the case $a=b=1$) is $d_{\TV} \left ( \Poiss \left (1 + e^c \right ), \Poiss(1) \right )$. The same result holds for $k$--cycles with $k=o(n)$~\cite{NT} and star transpositions~\cite{star}. Note that Conjecture~\ref{conjecture:limit-profile} describes a different limiting behavior for $a \neq b$, which we discuss in more detail in Section~\ref{sec:lower-bound}

\subsection{Structure of paper}

In Section~\ref{sec:prelim}, we introduce the relevant mathematical background. This includes the combinatorial terminology related to partitions and tableaux. We also discuss the necessary representation theory for finite groups and the symmetric group. We end the preliminary section discussing the hive model, which we use to analyze Littlewood-Richardson coefficients, and define some distances on probability measure (e.g. total variation distance and Hellinger distance). In Section~\ref{sec:proof-of-theorem-2}, we prove Theorem~\ref{thm:main-theorem-3}. This section depends heavily on the representation theoretic preliminaries in Section~\ref{sec:prelim}. In Section~\ref{sec:proof-of-upper-bound} we prove the upper bound of Theorem~\ref{thm:main-theorem-1}. The proof splits into three different zones, with different types of analysis in each zone. At times, the analysis is technical. We move some of the technicalities and calculations to the appendix. In Section~\ref{sec:lower-bound} we prove the lower bound of Theorem~\ref{thm:main-theorem-1} by proving Theorem~\ref{thm:main-theorem-4'}. 

\section{Preliminaries}\label{sec:prelim}

\subsection{Combinatorial Terminology}

In this section, we review and set up the relevant combinatorial definitions and results for the proof of Theorem~\ref{thm:main-theorem-1}. The main combinatorial object will be (integer) partitions. We will define partitions, the dominance order on partitions, various types of tableaux, and some statistics on partitions. Most of the material written in this section can be found in any book on the combinatorics of the symmetric group (see~\cites{Sagan2001, Macdonald2015, Fulton1996}).

\subsubsection{Partitions}

For a positive integer $N$, a \textbf{partition} $\mu$ of size $N$ is a \emph{weakly decreasing} sequence of non-negative integers which sum to $N$. A partition corresponding to the sequence $\lambda_1, \lambda_2, \ldots$ is notated either as $(\lambda_1, \lambda_2, \ldots)$ or $1^{\mu_1} 2^{\mu_2} 3^{\mu_3} \ldots$ where $\mu_k$ is the number of $k$'s in the sequence. The notation $\lambda \vdash N$ means that $\lambda$ is a partition of $N$. 

\begin{Remark}
    Since any partition of finite size necessarily contains an infinite tail of zeroes, we often notate partitions with a \emph{finite} weakly decreasing sequence of non-negative integers. We do this with the understanding that the sequence continues with an infinite tail of zeroes. For example, the partitions $(3, 2, 2, 1)$ and $(3, 2, 2, 1, 0)$ both represent the same partition $(3, 2, 2, 1, 0, 0, \ldots)$.
\end{Remark}

\subsubsection{Diagrams}

We can represent partitions diagrammatically with their Young diagrams. The \textbf{Young diagram} of a partition $\lambda = (\lambda_1, \lambda_2, \ldots)$ is a left-justified array of boxes such that the first row has $\lambda_1$ boxes, the second row has $\lambda_2$ boxes, and so on. For example, see the left picture of Figure~\ref{fig:conjugative-partitions} for the Young diagram of $\lambda = (3, 1)$.

Let $\lambda = (\lambda_1, \lambda_2, \ldots)$ be a partition. The \textbf{conjugate partition}, denoted $\lambda^*$, is the partition $\lambda^* = (\lambda_1^*, \lambda_2^*, \ldots)$ where $\lambda_k^* = \max_{i \geq 0} \{ \lambda_i \geq k \}$. Diagrammatically, the Young diagram of $\lambda^*$ is the transpose of the Young diagram of $\lambda$. See the right picture of Figure~\ref{fig:conjugative-partitions} for the Young diagram of $\lambda^*$ when $\lambda = (3, 1)$.

\begin{figure}
\centering
\begin{tikzpicture}[scale = 0.8]
    \draw (0, 0) rectangle + (1, -1);
    \draw (1, 0) rectangle + (1, -1);
    \draw (2, 0) rectangle + (1, -1);
    \draw (0, -1) rectangle + (1, -1);

    \draw (7, 0) rectangle + (1, -1);
    \draw (8, 0) rectangle + (1, -1);
    \draw (7, -1) rectangle + (1, -1);
    \draw (7, -2) rectangle + (1, -1);
\end{tikzpicture}

\caption{On the left, we have the Young diagram for $\lambda = (3, 1)$. On the right, we have the Young diagram for $\lambda^* = (2,1, 1)$.}\label{fig:conjugative-partitions}
\end{figure}
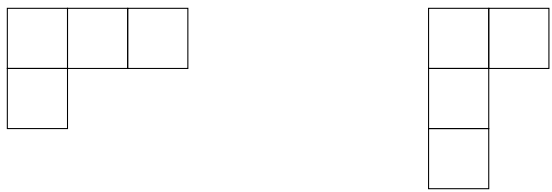

When the Young diagram of a partition $\lambda$ contains the Young diagram of another partition $\mu$, we write $\lambda \supset \mu$. In this case, the \textbf{skew Young diagram} of shape $\lambda / \mu$ is the Young diagram of $\lambda$ with the Young diagram of $\mu$ removed. The right picture in Figure~\ref{fig:3221-diagram} without the filling is an example of a skew Young diagram with $\lambda = (4, 3, 2)$ and $\mu = (2, 1)$.

For $N \geq 0$, we let $\YY_N$ denote the set of partitions of size $N$. We let $\YY = \bigcup_{N \geq 0} \YY_N$ be the set of all partitions.

\subsubsection{Tableaux}

We define various types of tableaux on Young diagrams and skew Young diagrams. Let $\lambda, \mu \vdash N$ be partitions of $N$. A \textbf{filling} of $\lambda$ with \textbf{content} $\mu$ is a way of labeling the boxes in the Young diagram of $\lambda$ (exactly one number in each box) with $\mu_1$ $1$'s, $\mu_2$ $2$'s, and so on. Figure~\ref{fig:3221-diagram} offers two examples of fillings of different shapes and different contents. In a similar way, we define fillings on Young diagrams of skew shape. 

Given a partitions $\lambda, \mu \vdash N$, a \textbf{semistandard Young tableaux} (SSYT) of shape $\lambda$ and content $\mu$ is a filling of $\lambda$ with content $\mu$ such that rows are \emph{weakly} increasing from left to right and columns are \emph{strictly} increasing from top to bottom. In a similar way, we define semistandard Young tableaux on skew shapes. 

Given a partition $\lambda \vdash N$, a \textbf{standard Young tableaux} (SYT) of shape $\lambda$ is a semistandard Young tableaux of $\lambda$ with content $(1, 1, \ldots, 1)$. In other words, it is a way to fill in the Young diagram of $\lambda$ with the numbers $[N]$ (each number used exactly once) such that the rows and columns are strictly increasing. Figure~\ref{fig:3221-diagram} shows a standard Young tableaux of shape $(3, 2, 2, 1)$.

\begin{figure}
\centering
\begin{tikzpicture}[scale=0.8]
    \foreach \y [count=\i] in {3, 2, 2, 1} {
        \foreach \x in {1,...,\y} {
            \draw (\x,-\i) rectangle ++(1,-1);
        }
    }
    \node at (1.5,-1.5) {1};
    \node at (2.5,-1.5) {2};
    \node at (3.5,-1.5) {3};
    \node at (1.5,-2.5) {4};
    \node at (2.5,-2.5) {5};
    \node at (1.5,-3.5) {6};
    \node at (2.5,-3.5) {7};
    \node at (1.5,-4.5) {8};

    \draw (6, -4) rectangle + (1, 1);
    \draw (7, -4) rectangle + (1, 1);
    \draw (7, -3) rectangle + (1, 1);
    \draw (8, -3) rectangle + (1, 1);
    \draw (8, -2) rectangle + (1, 1);
    \draw (9, -2) rectangle + (1, 1);

    \node at (6.5, -3.5) {2};
    \node at (7.5, -3.5) {3}; 
    \node at (7.5, -2.5) {1};
    \node at (8.5, -2.5) {2};
    \node at (8.5, -1.5) {1};
    \node at (9.5, -1.5) {1};

\end{tikzpicture} 
\caption{On the left, we have a \emph{standard Young tableaux} on a \emph{Young diagram} of shape $(3, 2, 2, 1)$. On the right, we have a \emph{Littlewood-Richardson tableaux} of \emph{content} $(3, 2, 1)$ on a skew-Young diagram of shape $(4, 3, 2)/(2, 1)$.}\label{fig:3221-diagram}
\end{figure}

Let $\lambda / \mu$ be a Young diagram of skew shape of size $\ell$. We say that a semistandard tableaux of shape $\lambda / \mu$ is a \textbf{Littlewood Richardson tableaux} if the concatenation of the \emph{reversed rows} from top to bottom satisfies the property that in every prefix the number of instances of $i$ is at least the number of instances of $i+1$ for all $i$. In other words, the content of any prefix of the word is a partition. 

For an example, the concatenation of the reversed rows of the right diagram in Figure~\ref{fig:3221-diagram} is $112132$. This satisfies the aforementioned property. Thus, the right picture in Figure~\ref{fig:3221-diagram} is a Littlewood-Richardson tableaux. 

\subsubsection{Dominance order}

Let $\lambda, \mu \in \YY$ be any partitions (not necessarily of the same size). We write $\lambda \lhd \mu$ if and only if  $\sum_{i = 1}^k \lambda_i \leq \sum_{i = 1}^k \mu_i$ for all $k \geq 1$. When $\lambda \lhd \mu$ we say that $\mu$ \textbf{dominates} $\lambda$. The partial order $\lhd$ is called the \textbf{dominance} order. 

\subsubsection{Statistics}~\label{sec:statistics}

We now define some statistics on partitions and discuss some of their properties. The statistics will enumerate partitions and certain number of fillings on (skew) Young tableaux. 

\subsubsection{Partition function}

Let $p(n)$ be the \textbf{partition function}. This counts the number of partitions of $n$. In~\cite{Hardy2000}, Hardy and Ramanujan proved the following asymptotic estimate for the partition function. 

\begin{Theorem}[Hardy-Ramanujan]\label{thm:partition-function-bound}
    \phantom{h}
    \[
        p(n) \sim \frac{1}{4n\sqrt{3}} \exp \left\{ \pi \sqrt{\frac{2n}{3}} \right\}.
    \]
\end{Theorem}

\subsubsection{Number of SYT and Kostka Numbers}

For any partition $\lambda$, we define $f_\lambda$ to be the number of standard Young tableaux of shape $\lambda$. To compute $f_\lambda$, one can use the celebrated \emph{Hook-length formula}. We will not need this formula and we direct the interested reader to~\cites{Frame1954, Meliot2017}. 

For any partitions $\lambda, \mu \vdash N$, we define the \textbf{Kostka number} $K_{\lambda, \mu}$ to be the number of semistandard Young tableaux of shape $\lambda$ and content $\mu$. We will need the following enumerative result only in Section~\ref{sec:lower-bound} when we prove the lower bound of Theorem~\ref{thm:main-theorem-1}. 

\begin{Lemma}\label{lem:kostka number example computation}
    Let $j \geq 0$ be a fixed non-negative integer, and $T \vdash j$ a fixed partition. For sufficiently large $N$, let $\lambda = (N-j, T)$ be a partition of $N$. For sufficiently large $N$, we have that 
    \begin{equation}\label{eqn:kostka-in-this-specific-case}
        K_{\lambda, (N-t, 1^t)} = \binom{t}{j} f_T. 
    \end{equation}
    In particular, when $t < j$ we have $K_{\lambda, (N-t, 1^t)} = 0$. 
\end{Lemma}

\begin{proof}
    Suppose that $t \geq j$, the left hand side counts the number of semistandard Young tableaux on $\lambda$ with content $(N-t, 1^t)$. In each such SSYT, we are forced to place all $N-t$ ones in the first $N-t$ entries of the first row of $\lambda$. Our remaining labels consist of $t$ distinct numbers. The remaining boxes are $t-j$ boxes in the first row and $j$ boxes below the first row in the shape of $T$. For sufficiently large $N$, the $t-j$ boxes in the first row is strictly to the right of the $j$ boxes below. In this case, the first row and the rows below the first row are independent of each other with respect to the strictly increasing column condition. Thus, we can pick any $t-j$ numbers from are $t$ distinct numbers to be placed in increasing order in the first row. For the rest of the boxes, we can arrange them as any SYT of shape $T$. Thus, Equation~\ref{eqn:kostka-in-this-specific-case} holds when $t \geq j$. 

    When $t < j$, we want to prove that $K_{\lambda, (N-t, 1^t)} = 0$. But this is true because any SSYT on $\lambda$ will be forced to have $N-t$ ones in the first row. But the first row of $\lambda$ is of size $N-j < N-t$. Thus, there are no such SSYT. We have proved in that Equation~\ref{eqn:kostka-in-this-specific-case} holds in all cases. 
\end{proof}

\subsubsection{Littlewood-Richardson Coefficients}

For any triple of partitions $\lambda, \mu, \nu \in \YY$, let $c^\lambda_{\mu, \nu}$ be the number of Littlewood-Richardson tableaux of shape $\lambda / \mu$ and content $\nu$. The numbers $c^\lambda_{\mu, \nu}$ are called \textbf{Littlewood-Richardson coefficients}, or LR coefficients for short. 

We will be particularly interested in the combinatorial conditions when $c^\lambda_{\mu, \nu} > 0$ as this determines when the terms which show up in our upper bound for Theorem~\ref{thm:main-theorem-1}. When this happens, we say $(\lambda, \mu, \nu)$ is a \textbf{Littlewood-Richardson triple}. We discuss this more when we discuss hive models in Section~\ref{sec:hive-model}. 

\begin{Definition}
    Let $\lambda \vdash N$ be a partition and let $[N] = A \sqcup B$ be the two sets in the biased random transposition shuffle. We define $\LR^\lambda$ to be the set of tuples of partitions $(\mu, \nu)$ with $|\mu| \vdash |A|$ and $\nu \vdash |B|$ which form a Littlewood-Richardson triple with $\lambda$. For $\mu, \nu \in \YY$, we define $\LR_{\mu, \nu}$ to be the set of partitions $\lambda$ which form a Littlewood-Richardson triple with $(\mu, \nu)$. Explicitly, these sets are given by
    \begin{align*}
        \LR^\lambda & \eqdef \{ (\mu, \nu) \in \YY_{|A|}\times \YY_{|B|} : c^\lambda_{\mu, \nu} > 0 \}, \\
        \LR_{\mu, \nu} & \eqdef \{\lambda \in \YY : c^\lambda_{\mu, \nu} > 0 \}.
    \end{align*}
\end{Definition}

\subsubsection{Relations between statistics}

The interplay between $f_\lambda$ and $c^\lambda_{\mu, \nu}$ comes from the representation theory of the symmetric group and general linear group. We present several identities which come exactly from the representation theory of these groups. 

\begin{Proposition}\label{prop:sym-rep-theory-identities} \phantom{h}
    \begin{enumerate}
        \item Let $N \geq 1$ be a positive integer. Then 
        \[
            \sum_{\lambda \vdash N} f_\lambda^2 = N!
        \]
        \item Let $\lambda \vdash N$ and let $m \leq N$. Then 
        \[
            f_\lambda = \sum_{\substack{\mu \vdash m, \\ \nu \vdash N-m}} c^\lambda_{\mu, \nu} f_\mu f_\nu. 
        \]
        \item Let $\mu, \nu \in \YY$ be partitions. Then 
        \[
            \binom{|\mu| + |\nu|}{|\mu|}f_\mu f_\nu = \sum_{\lambda \in \YY} c^\lambda_{\mu, \nu} f_\lambda. 
        \]        
        \item Let $\lambda, \mu, \nu$ be partitions. Then $c^\lambda_{\mu, \nu} = c^{\lambda^*}_{\mu^*, \nu^*}$. 
    \end{enumerate}
\end{Proposition}

\begin{proof}
    Part (1) comes from taking the dimensions of Proposition~\ref{prop:rep-theory-facts-sym}(4). Part (2) comes from taking dimensions of Theorem~\ref{thm:LR-rule}(1). Part (3) comes from taking dimensions of Theorem~\ref{thm:LR-rule}(2). Part (4) comes from Corollary 2, Section 5.2 of~\cite{Fulton1996}.
\end{proof}

We will frequently use the following bound on $f_\lambda$ in our analysis of the cutoff time. 

\begin{Lemma}\label{lem:f-lambda-bound}
    \phantom{h}
    \begin{enumerate}
        \item Let $\lambda \vdash 2n$ be a partition satisfying $\lambda_1 \geq dn$ and $\lambda_1^* \geq d^* n$ for two positive constants $d, d^* \in [0, 2)$. Then we have the bound 
        \[
            (f_\lambda)^2 \leq \exp \left ((2 - d - d^*) n \log n + O(n) \right ) 
        \]
        where the $O(n)$ error is independent of $d, d^*$. 
        \item Let $\lambda \vdash 2n$ have first row of length $\lambda_1 = 2n - j$. Then 
        \[
            (f_\lambda)^2 \leq \frac{1}{j!} \exp (2j \log (2n)). 
        \]
    \end{enumerate}
\end{Lemma}

\begin{proof}
    We will bound $f_\lambda$ by over-counting the number of standard Young tableaux. The first row and first column of a Young diagram of shape $\lambda$ have a total of $\lambda_1 + \lambda_1^* - 1$ squares. We first pick $\lambda_1 + \lambda_1^* - 1$ numbers from $[2n]$ to occupy these squares. To arrange them in these squares, it is enough to pick the numbers which go in the first row. After doing so, the order of the numbers in the first row, the content of the numbers in the first column, and the order of the numbers in the first column will all be determined. There is at most $\exp(O(n))$ ways to do this. 
    
    The number of ways to fill the remaining squares with our remaining numbers is bounded above by the number of SYT of shape $\mu$ where $\mu \vdash 2n - \lambda_1 - \lambda_1^* + 1$ is the Young diagram obtained after deleting the first row and first column from the Young diagram of $\lambda$. From Proposition~\ref{prop:sym-rep-theory-identities}(1) this number is at most $\sqrt{(2n - \lambda_1 - \lambda_1^* + 1)!}$. 

    Putting this all together, we get the bound
    \begin{align*}
        (f_\lambda)^2 & \leq (2n - \lambda_1 - \lambda_1^*)! \cdot \exp (O(n)) \\
        & \leq ((2-d-d^*)n)! \cdot \exp(O(n)) \\
        & \leq n^{(2-d-d^*)n} \cdot \exp(O(n)) \\
        & = \exp ((2-d-d^*)n \log n + O(n) ). 
    \end{align*}
    This completes the proof of part (1). For part (2), see~\cite{Diaconis1981}*{Corollary 2}.
\end{proof}

\subsubsection{Diagonal index} We define the notion of the \emph{diagonal index} of a partition. This terminology was introduced in~\cite{Dieker2018} to study the spectrum of the random-to-random shuffle. 

\begin{Definition}
    Let $\lambda$ be a partition of any size. We define the \textbf{diagonal index} to be
    \[
        \Diag(\lambda) \eqdef \sum_{i \geq 0} \binom{\lambda_i}{2} - \sum_{i \geq 0} \binom{\lambda_i^*}{2} = \sum_{(i, j) \in \lambda} (j-i). 
    \]
\end{Definition} 

\begin{Remark}
    In representation theory, the diagonal index is referred to as the \emph{content} of the partition. When we refer to the content of a cell in a Young diagram, we are referring to the summand $j-i$ where $(i, j)$ are the coordinates of that cell. By definition, the diagonal index of $\lambda$ is the sum of the contents of all of the cells in the Young diagram of shape $\lambda$. 
\end{Remark}

\begin{Definition}
    Let $\lambda = (\lambda_1, \lambda_2, \ldots)$ and $\mu = (\mu_1, \mu_2, \ldots)$ be partitions of any size. We define the \textbf{inner product} of $\lambda$ and $\mu$, denoted $\langle \lambda, \mu \rangle$, to be their inner product as elements of $\ell_2(\NN)$. Explicitly, we have 
    \[
        \langle \lambda , \mu \rangle \eqdef \sum_{i \geq 1} \lambda_i \mu_i. 
    \]
\end{Definition}

We now prove a few useful lemmas about the diagonal index. We will apply these results frequently in our analysis of the cutoff time. 

\begin{Lemma}\label{lem:oliver}
    Let $\lambda, \mu, \nu$ be partitions of $N$ such that $\lambda \rhd \mu$. Then we have:
    \begin{enumerate}
        \item $\langle \lambda, \nu \rangle \geq \langle \mu, \nu \rangle$. 
        \item $\Diag(\lambda) \geq \Diag(\mu)$. 
        \item $\Diag(\lambda^*) = - \Diag(\lambda)$.
        \item $\Diag(\lambda) \leq (\lambda_1 - 1) \frac{N}{2}$. 
    \end{enumerate}
\end{Lemma}

\begin{proof}
    For parts (2), (3), and (4), see Lemma 2.3.4 in \cite{matheauraven2020random}. For (1), it suffices to prove the inequality when $\lambda \rhd \mu$ is a covering relation. In the covering relation case, $\lambda$ and $\mu$ agree on all rows except for two, say with indices $r_1 < r_2$, and we have
    \begin{align*}
        \lambda_{r_1} & = \mu_{r_1} + 1 \\
        \lambda_{r_2} & = \mu_{r_2} - 1. 
    \end{align*}
    Then, if we look at the difference of $\langle \lambda, \nu \rangle$ and $\langle \mu, \nu \rangle$, we have 
    \[
        \langle \lambda, \nu \rangle - \langle \mu, \nu \rangle = (\lambda_{r_1} - \mu_{r_1}) \nu_{r_1} + (\lambda_{r_2} - \mu_{r_2}) \nu_{r_2} = v_{r_1} - v_{r_2} \geq 0. 
    \]
\end{proof}

\begin{Lemma}\label{lem:diag-of-addition-can-be-expanded}
    Let $\lambda, \mu$ be partitions. Then we have 
    \[
        \Diag(\lambda + \mu) = \Diag(\lambda) + \Diag(\mu) + \langle \lambda, \mu \rangle.
    \]
\end{Lemma}
\begin{proof}
    To prove this, we will use the formula for the diagonal sum given in terms of the \emph{content} of each box: 
    \[
        \Diag(\lambda) = \sum_{(i, j) \in \lambda} (j-i). 
    \]
    Since the Young diagram of $\lambda + \mu$ is the Young diagram of $\lambda$ with the rows of $\mu$ added to the rows of $\lambda$, the Young diagram of $\lambda + \mu$ still contains the boxes of the Young diagram of $\lambda$. Moreover, these boxes still give the same contribution in $\Diag(\lambda + \mu)$ as in $\Diag(\lambda)$. In other words, we have 
    \[
        \Diag(\lambda + \mu) - \Diag(\lambda) = \sum_{(i, j) \in (\lambda + \mu)/\lambda} (j-i). 
    \]
    The rest of the boxes are in $\mu$ but shifted by some amount horizontally depending on the row of the box. Explicitly, if a box of $\mu$ is in the $k$th row, then its location in $\lambda + \mu$ is $\lambda_k$ boxes to the right. The row index stays the same but the column index increases by $\lambda_k$. Thus, we have 
    \begin{align*}
        \sum_{(i, j) \in (\lambda + \mu) / \lambda} (j-i) & = \sum_{k \geq 1} \sum_{(i, j) \in \mu_{(k)}} (j-i) + \lambda_k \\
        & = \Diag(\mu) + \langle \lambda, \mu \rangle. 
    \end{align*}
    This completes the proof.
\end{proof}

\begin{Lemma}\label{lem:diagonal-index-asymptotics}
    Let $j \geq 0$ be a positive integer. Let $\lambda \eqdef \lambda^{(N)} \vdash N$ be a sequence of partitions such that $\lambda_1 = N - j$. Then, 
    \[
        \Diag (\lambda) = \frac{1}{2}N^2 - \frac{(2j+1)}{2} N + O_j(1).
    \]
\end{Lemma}
\begin{proof}
    The only contribution of non-constant order to the diagonal index comes from the first row. Hence, we have 
    \[
        \Diag(\lambda) = \binom{N-j}{2} + O(1) = \frac{1}{2}N^2 - \frac{(2j+1)}{2} N + O_j(1). 
    \]
\end{proof}

\begin{Lemma}\label{lem:generic-diagonal-sum-bound}
    Let $\lambda \vdash 2n$ be a partition. Then there is some universal constant $C > 0$ such that the following hold:
    \begin{enumerate}
        \item We have 
        \[
            \Diag(\lambda) \leq n \lambda_1 - \frac{\lambda_1 \lambda_1^*}{2} - \frac{(\lambda_1^*)^2}{2} + C n
        \] 

        \item We have  
        \[
            \Diag(\lambda) \geq -n \lambda_1^* + \frac{\lambda_1 \lambda_1^*}{2} + \frac{\lambda_1^2}{2} - Cn. 
        \] 
    \end{enumerate}
\end{Lemma}

\begin{proof}
    In order to give an upper bound to the diagonal index $\Diag(\lambda)$, we find a partition which dominates any partition with the same first row and first column as $\lambda$. Then, we can use Lemma~\ref{lem:oliver}(2) to get an upper bound.

    To construct this Young diagram, it is easiest to consider its Young diagram. The diagram will be the one with first row $\lambda_1$, first column $\lambda_1^*$, and with the maximum possible number of rows of length $\lambda_1$. To make this construction more formal, we first start with a diagram with one column of length $\lambda_1^*$. We will then start filling the rows from top to bottom with $(\lambda_1 - 1)$ more boxes in each row. In particular, each row that we add $\lambda_1 - 1$ boxes to will have length $\lambda_1$. Suppose that we can only complete at most the first $k$ rows to rows of length $\lambda_1$. The number $k$ is characterized by the maximum number satisfying $k(\lambda_1 - 1) \leq 2n - \lambda_1^*$. In other words, when we write 
    \[
        2n - \lambda_1^* = (\lambda_1 - 1) k + r, \quad 0 \leq r < \lambda_1 - 1,
    \]
    the numbers $k$ and $r$ are exactly the quotient and remainder when dividing $2n-\lambda_1^*$ by $\lambda_1-1$. We then fill in the next row with the remaining $r$ boxes. Call this Young diagram $\mathbf{W}_\lambda$.\footnote{Note that even though the subscript $\lambda$ is a partition, the definition of $\mathbf{W}_\lambda$ only depends on $\lambda_1$ and $\lambda_1^*$.}
    
    Given any partition $\mu$ with first row $\lambda_1$ and first column $\lambda_1^*$, it is not difficult to prove that we can always transform it into $\mathbf{W}_\lambda$ by moving outside corners (not lying in the first column) up to complete rows to length $\lambda_1$ from top to bottom. This proves that $\mathbf{W}_\lambda$ dominates $\lambda$.
    
    We now bound the diagonal index $\Diag(\mathbf{W}_\lambda)$. From the construction, the Young diagram of $\mathbf{W}_\lambda$ consists of $k$ rows of length $\lambda_1$, possibly an extra row of length $r+1$, and exactly $\lambda_1^*$ boxes in the first column. There are no other possible rows. Thus, we can bound the diagonal index by 
    \begin{align*}
        \Diag(\mathbf{W}_\lambda) & \leq k\binom{\lambda_1}{2} + \binom{r+1}{2} - \binom{\lambda_1^*}{2} \\
        & \leq \frac{\lambda_1}{2} (2n-\lambda_1^* - r) + \frac{r(r+1)}{2} - \frac{\left ( \lambda_1^* \right )^2}{2} + O(n) \\
        & \leq n \lambda_1 - \frac{\lambda_1 \lambda_1^*}{2} - \frac{\left (\lambda_1^* \right )^2}{2} + Cn 
    \end{align*}
    for some $C > 0$. Since $\Diag(\lambda) \leq \Diag(\mathbf{W}_\lambda)$, this completes the proof of part (1). Part (2) follows immediately from part (1). Indeed, we have that 
    \[
        \Diag(\lambda) = - \Diag(\lambda^*) \geq - \left \{ n \lambda_1^* - \frac{\lambda_1 \lambda_1^*}{2} - \frac{\lambda_1^2}{2} \right \} - Cn. 
    \]
    This suffices for the proof.
\end{proof}

\subsection{Spectrum of \texorpdfstring{$\bRT$}{bRT}}

One of the main ingredients of our proof of Theorem~\ref{thm:main-theorem-1} is an explicit diagonalization of the transition matrix of the $\bRT$ shuffle. We first define the eigenvalues which appear in the spectrum. In Theorem~\ref{thm:main-theorem-2}, we also give the multiplicities for each eigenvalue. 

\begin{Definition}
    Let $\lambda, \mu, \nu$ be partitions. Let $\Eig^\lambda_{\mu, \nu}$ be equal to 
    \[
        \frac{a^2 |A| + b^2 |B|}{2N} + \frac{2(a^2 - ab)}{N^2} \Diag(\mu) + \frac{2(b^2 -ab)}{N^2} \Diag(\nu) + \frac{2ab}{N^2} \Diag(\lambda). 
    \]
\end{Definition}

\begin{Theorem}\label{thm:main-theorem-2}
    The transition matrix of $\bRT$ is diagonalizable with eigenvalues $\Eig^\lambda_{\mu, \nu}$ for all $\lambda \vdash N$, $\mu \vdash |A|$, $\nu \vdash |B|$ with multiplicities $c^\lambda_{\mu, \nu} f_\lambda f_\mu f_\nu$.
\end{Theorem}

The proof of Theorem~\ref{thm:main-theorem-2} relies on the representation theory of the symmetric group. We can recast $\bRT$ in the language of representation theory. In this perspective, the shuffle becomes left multiplication by some element in the group algebra. Applying \emph{Schur's lemma} in a clever way will allow us to diagonalize $\bRT$. In the next section, we briefly review the necessary background from the representation theory of the symmetric group. 

\subsection{Representation theory of \texorpdfstring{$\fS_N$}{SN}}

In this section, we briefly review the representation theory of the symmetric group. For a more comprehensive treatment, we direct the reader to~\cites{Fulton1996, Fulton2004, Serre1977}. Recall that the symmetric group $\fS_N$ is the group of permutations on $N$ letters.  Although we are primarly interested in the group $\fS_N$, many of the following concepts are defined for a general finite group. 

\subsubsection{The group algebra} Let $G$ be a finite group. The group algebra $\CC[G]$ is the complex vector space of complex-valued functions on $G$ equipped with the \textbf{convolution product} $*$: for $f_1, f_2 : G \to \CC$, we define
\[
    (f_1 * f_2) (x) = \sum_{\substack{g_1, g_2 \in G \\ g_1g_2 = x}} f_1(g_1) f_2(g_2).  
\]
Equivalently, we can view elements of $\CC[G]$ as formal linear combinations of elements in $G$.  By viewing $g \in G$ as the indicator function $1_g(\bullet) : G \to \CC$, we can write any function $f : G \to \CC$ as a linear combination of elements in $G$ in the following way: 
\[
    f = \sum_{g \in G} f(g) \cdot 1_g = \sum_{g \in G} f(g) \cdot g. 
\]
Under the identification of functions as formal linear combinations of group elements, the convolution operation on formal linear combinations of group elements translates to the natural way one would think to multiply these formal sums:
\[
    \left ( \sum_{g \in G} c_g \cdot g \right ) \cdot \left ( \sum_{h \in G} d_h \cdot h \right ) = \sum_{g, h \in G} c_g d_h \cdot (gh) = \sum_{g \in G} \left ( \sum_{xy = g} c_x d_y \right ) \cdot g.
\]
To summarize, we can view the group algebra in the following equivalent ways:
\[
    \CC[G] \eqdef \{ f : G \to \CC \} = \left \{ \sum_{g \in G} c_g \cdot g \, \mid \, c_g \in \CC \text{ for all } g \in G \right \},
\]
equipped with the convolution product. 

\subsubsection{Representations}

A \textbf{representation} of $G$ is a $\CC[G]$-module $V$. Equivalently, a representation is an algebra homomorphism $\rho : \CC[G] \to \End_\CC(V)$ or a group homomorphism $\rho : G \to \GL(V)$. We say a vector subspace $W \subset V$ of a representation is a \textbf{subrepresentation} or \textbf{$G$-invariant} or \textbf{$\CC[G]$-invariant} if $\CC[G]W \subseteq W$. When $W$ is a $G$-invariant subspace, we can view $W$ as a representation of $G$. We say that a representation $V$ is \textbf{irreducible} if and only if the only subrepresentations are $0$ and $V$. 

\begin{Example}\label{example:regular-representation}
    \phantom{h}
    \begin{enumerate}
        \item The group algebra $\CC[G]$ is a $\CC[G]$-module where $\CC[G]$ acts on $\CC[G]$ by left multiplication. This is called the \textbf{regular representation} of $G$. This is the representation-theoretic model for our card shuffle. 
        
        \item The symmetric group $\fS_N$ acts on $\CC^N$ by permuting the standard basis. This gives $\CC^N$ the structure of a $\fS_N$-representation called the \textbf{permutation module} $M_N$. The permutation module will play a vital role in Section~\ref{sec:lower-bound}. 

        \item The permutation module is not irreducible. The \textbf{standard representation} is a non-trivial subrepresentation of the permutation module given by the subspace
        \[
            \{(x_1, \ldots, x_n) \in \CC^n :  x_1 + \ldots + x_n = 0 \}. 
        \]
    \end{enumerate}
\end{Example}

\subsubsection{Direct sums, tensor products, and direct products} Let $\rho_1 : G \to \GL(V_1)$ and $\rho_2 : G \to \GL(V_2)$ be group representations. Then the direct sum $V_1 \oplus V_2$ has a natural group representation structure via $(\rho_1 \oplus \rho_2) (g) \eqdef \rho_1(g) \oplus \rho_2(g)$. The representation $V_1 \oplus V_2$ is called the \textbf{direct sum} of $V_1$ and $V_2$. 

Let $G$ and $H$ be two groups and let $\rho : G \to \GL(V)$ and $\psi : H \to \GL(W)$ be two group representations. Then the tensor product $V \otimes_\CC W$ has a natural $G \times H$ representation structure given by $(\rho \boxtimes \psi) (g \otimes h) \eqdef \rho(g) \otimes \rho(h)$. We denote the resulting representation as $V \boxtimes W$ and call it the \textbf{direct product} of the two representations.

Let $V, W$ be two $\CC[G]$-modules. We define the \textbf{tensor product} of $V$ and $W$, denoted $V \otimes W$, to be the representation $V \otimes_\CC W$ with the module structure given by the diagonal action $g(v \otimes w) \eqdef gv \otimes gw$. 

\subsubsection{Restriction and induction}

We introduce restricted representations and induced representations. The reader does not need to know the definition of induced representations to read the rest of this paper. The only fact we use related to induced representations is Equation~\ref{eqn:dimensions-induct-restrict}. 

Let $G$ be a finite group and $H \subseteq G$ be a subgroup. Let $V$ be a $G$-module. The \textbf{restriction} of $V$ to $H$, denoted $\Res^G_H V$, is $V$ viewed as a $H$-module. Let $W$ be a $H$-module. The \textbf{induced representation} of $W$ to $G$, denoted $\Ind_H^G W$ is the left $\CC[G]$-module $\CC[G] \otimes_{\CC[H]} W$ where we view the left tensor factor $\CC[G]$ as a right $\CC[H]$-module. It is not hard to see that
\begin{equation}\label{eqn:dimensions-induct-restrict}
    \dim \Res_H^G V = \dim V, \quad \dim \Ind_H^G W = [G : H] \dim W. 
\end{equation}

\subsubsection{Specht Modules}

We now specialize to $G = \fS_N$. In this setting, the irreducible representations are in a natural 1:1 correspondence with partitions of $N$. For every $\lambda \vdash N$, there is an irreducible representation $S^\lambda$ called the \textbf{Specht module} of weight $\lambda$. We will not give an explicit description of these representations, and instead only list the properties that we need.

\begin{Proposition}[Properties of Specht modules]\label{prop:rep-theory-facts-sym}
    \phantom{h}
    \begin{enumerate}
        \item The representations $\{S^\lambda\}_{\lambda \vdash N}$ are pairwise non-isomorphic and exhaust all isomorphism classes of irreducible $\fS_N$-representations. 
        
        \item For all $\lambda \vdash N$, $\dim S^\lambda = f_\lambda$. 
        
        \item For distinct $i, j \in [N]$, let $t_{ij}$ be the transposition swapping $i$ and $j$. Let $t = \sum_{1 \leq i < j \leq N} t_{ij}$. Then left multiplication by $t$ on $S^\lambda$ is given by scalar multiplication by $\Diag(\lambda)$. 
        
        \item Let $\CC[\fS_N]$ be the regular representation. Then
        \[
            \CC[\fS_N] \simeq \bigoplus_{\lambda \vdash N} (S^\lambda)^{\oplus f_\lambda}.
        \]
    \end{enumerate}
\end{Proposition}

\begin{proof}
    For (1) and (2), we refer the reader to~\cite{Fulton1996}. For (3), see~\cite{Etingof2024}*{Exercise 27.9}. Finally (4) follows from (1), (2), and~\cite{Fulton2004}*{Proposition 3.29}.
\end{proof}

\subsubsection{Littlewood-Richardson Rule}

The \emph{Littlewood-Richardson rule} says that the LR coefficients defined in Section~\ref{sec:statistics} are the structure constants for products of Schur functions. This can be reinterpreted in the following equivalent form. See~\cite{Zelevinsky1981}*{Proposition 1}. 

\begin{Theorem}[The Littlewood-Richardson Rule]\label{thm:LR-rule}
    \phantom{1}
    \begin{enumerate}
        \item Let $\lambda \vdash N$ be a partition and let $r_1, r_2 \geq 0$ be positive integers satisfying $r_1 + r_2 = N$. Then, we have 
        \[
            \Res^{\fS_N}_{\fS_{r_1}\times \fS_{r_2}} S^\lambda = \bigoplus_{\substack{\mu \vdash r_1 \\ \nu \vdash r_2}} \left ( S^\mu \boxtimes S^\nu \right )^{c^\lambda_{\mu, \nu}}.
        \]

        \item Let $\mu \vdash r_1$ and $\nu \vdash r_2$ be partitions. Then 
        \[
            \Ind^{\fS_{r_1 + r_2}}_{\fS_{r_1} \times \fS_{r_2}} \left ( S^\mu \boxtimes S^\nu \right ) = \bigoplus_{\lambda \vdash r_1 + r_2} \left ( S^\lambda \right )^{\oplus c^\lambda_{\mu, \nu}}.
        \]
    \end{enumerate}
\end{Theorem}

\subsection{Algebraic interpretation of \texorpdfstring{$P^{t}(x, \bullet)$}{the transition matrix}}

Recall that $P$ is the transition matrix of the biased random transposition shuffle where $P(x, y)$ is the probability of going from $x$ to $y$ after one shuffle. The entries of the transition matrix $P(x, y)$ are given by

\begin{equation}\label{eqn:explicit-trans-matrix}
    P(x, y) = P(e, yx^{-1}) = 
    \begin{cases}
        \frac{2a^2}{N^2} & \text{ if } yx^{-1} \in \Flip(A) \\
        \frac{2b^2}{N^2} & \text{ if } yx^{-1} \in \Flip(B) \\
        \frac{2ab}{N^2} & \text{ if } yx^{-1} \in \Flip(A, B) \\
        \frac{a^2 |A| + b^2 |B|}{N^2} & \text{ if } yx^{-1} = \id.
    \end{cases}
\end{equation}

For any $t \geq 0$ and $x \in \fS_N$, the measure $P^{t} (x, \bullet)$ can be viewed as a function $P^{t} (x, \bullet) : \fS_N \to [0, 1]$ given $y \mapsto P^t (x, y)$. In other words, $P^t(x, \bullet)$ can be viewed as an element of the group algebra $\CC[G]$ with $\CC[G]$-module structure given by the regular representation (see Example~\ref{example:regular-representation}). In the following, we represent $P^t(x, \cdot)$ as a formal sum of group elements. 

\subsubsection{}

For any subset $S \subset [N]$, let $\Flip(S)$ be the set of transpositions involving elements of $S$. For disjoint sets $S_1, S_2 \subseteq [N]$, let $\Flip(S_1, S_2)$ denote the set of transpositions which swaps a card from $S_1$ with a card from $S_2$. For any disjoint subsets $S, T\subset [N]$, we define $\cT_S, \cT_{S, T} \in \CC[\fS_N]$ to be the elements defined by
\[
    \cT_S \eqdef \sum_{g \in \Flip(S)} g, \quad \cT_{S, T} \eqdef \sum_{\tau \in \Flip(S, T)} \tau. 
\]
For disjoint $S, T \subseteq [n]$, we have $\cT_{S, T} = \cT_{S \cup T} - \cT_S - \cT_T$. In this notation, we define the group algebra element $\cA \in \CC[\fS_N]$ given by 
\begin{align*}
    \mathscr{A} & \eqdef \left ( \frac{a^2 |A| + b^2 |B|}{N^2} \right ) \cdot \id + \frac{2a^2}{N^2} \cT_A + \frac{2b^2}{N^2} \cT_B + \frac{2ab}{N^2} \cdot \cT_{A, B} \\
    & = \left ( \frac{a^2 |A| + b^2 |B|}{N^2} \right ) \cdot \id + \frac{2(a^2 - ab)}{N^2} \cT_A + \frac{2(b^2 - ab)}{N^2} \cT_B + \frac{2ab}{N^2} \cT_{A \cup B}. 
\end{align*}
is the formal sum representation of $P(e, \cdot) : \fS_N \to \CC$.

\subsubsection{}

In the next section, we define the hive model. This is a combinatorial model for Littlewood-Richardson coefficients. The hive model will provide ways to bound $\Eig^\lambda_{\mu, \nu}$ using the fact that $c^\lambda_{\mu, \nu} > 0$. More specifically, it will give us some inequalities associated with Littlewood-Richardson triples.

In this paper, we only use hives in the proof of Theorem~\ref{thm:main-theorem-1}. The reader interested only in the proofs of Theorem~\ref{thm:main-theorem-2} and Theorem~\ref{thm:main-theorem-3} may immediately skip to Section~\ref{sec:proof-of-theorem-2} without any difficulty.

\subsection{The Hive Model}\label{sec:hive-model}

In~\cite{Knutson1999}, Knutson and Tao introduce honeycombs and hives as combinatorial, convex-geometric models for LR coefficients, and use them to prove \emph{the saturation conjecture}. In their paper, honeycombs take center-stage. In~\cite{Buch1998}, Buch uses hive models as the main object to give a streamlined version of their proof. Our treatment of hives comes from~\cite{Buch1998}. 

For $n \geq 0$, let $H_n$ be the \textbf{Hive triangle} graph. This is the planar graph given by a triangular array of side length $n$. In particular, each side of the triangle has $n+1$ vertices and there are $n^2$ small triangles.  For example, see Figure~\ref{fig:hive-triangle} for $H_6$. 

\begin{figure}
    \begin{tikzpicture}[scale=0.3]
        \draw (0,0) -- (2,0) -- (1,{sqrt(3)}) -- cycle;
        \draw (2,0) -- (4,0) -- (3,{sqrt(3)}) -- cycle;
        \draw (4,0) -- (6,0) -- (5,{sqrt(3)}) -- cycle;
        \draw (6,0) -- (8,0) -- (7,{sqrt(3)}) -- cycle;
        \draw (8,0) -- (10,0) -- (9, {sqrt(3)}) -- cycle;
        \draw (10,0) -- (12,0) -- (11, {sqrt(3)}) -- cycle;

        \draw (1, {sqrt(3)}) -- (3, {sqrt(3)}) -- (2, {2*sqrt(3)}) -- cycle;
        \draw (3, {sqrt(3)}) -- (5, {sqrt(3)}) -- (4, {2*sqrt(3)}) -- cycle;
        \draw(5, {sqrt(3)}) -- (7, {sqrt(3)}) -- (6, {2*sqrt(3)}) -- cycle;
        \draw(7, {sqrt(3)}) -- (9, {sqrt(3)}) -- (8, {2*sqrt(3)}) -- cycle; 
        \draw (9, {sqrt(3)}) -- (11, {sqrt(3)}) -- (10, {2 * sqrt(3)}) -- cycle;

        \draw (2, {2*sqrt(3)}) -- (4, {2*sqrt(3)}) -- (3, {3*sqrt(3)}) -- cycle;
        \draw (4, {2*sqrt(3)}) -- (6, {2*sqrt(3)}) -- (5, {3*sqrt(3)}) -- cycle;
        \draw (6, {2*sqrt(3)}) -- (8, {2*sqrt(3)}) -- (7, {3*sqrt(3)}) -- cycle;
        \draw (8, {2*sqrt(3)}) -- (10, {2*sqrt(3)}) -- (9, {3*sqrt(3)}) -- cycle;

        \draw (3, {3*sqrt(3)}) -- (5,{3*sqrt(3)}) -- (4,{4*sqrt(3)}) -- cycle;
        \draw (5, {3*sqrt(3)}) -- (7,{3*sqrt(3)}) -- (6,{4*sqrt(3)}) -- cycle;
        \draw (7, {3*sqrt(3)}) -- (9,{3*sqrt(3)}) -- (8,{4*sqrt(3)}) -- cycle;

        \draw (4, {4*sqrt(3)}) -- (6,{4*sqrt(3)}) -- (5,{5*sqrt(3)}) -- cycle;
        \draw (6, {4*sqrt(3)}) -- (8,{4*sqrt(3)}) -- (7,{5*sqrt(3)}) -- cycle;

        \draw (5,{5*sqrt(3)}) -- (7, {5*sqrt(3)}) -- (6, {6*sqrt(3)}) -- cycle;

        \foreach \x in {0,2,...,12}
        \fill (\x,0) circle (6pt);
        \foreach \x in {1,3,...,11}
        \fill (\x,{sqrt(3)}) circle (6pt);
        \foreach \x in {2,4,...,10}
        \fill (\x,{2*sqrt(3)}) circle (6pt);
        \foreach \x in {3,5,...,9}
        \fill (\x,{3*sqrt(3)}) circle (6pt);
        \foreach \x in {4,6,8}
        \fill (\x,{4*sqrt(3)}) circle (6pt);
        \foreach \x in {5,7}
        \fill (\x,{5*sqrt(3)}) circle (6pt);
        \fill (6,{6*sqrt(3)}) circle (6pt);
    \end{tikzpicture}
    \caption{The Hive Triangle $H_6$}\label{fig:hive-triangle}
\end{figure}
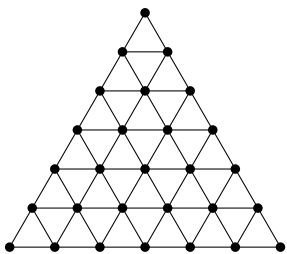

A \emph{hive} is a way of labeling the vertices of the Hive triangle with real numbers such that the numbers satisfy some collection of \emph{rhombus inequalities}. A \textbf{rhombus} of the hive triangle is just a pair of small triangles which share an edge. Note that there are three types of rhombi depending on their orientations. The three orientations are shown in Figure~\ref{fig:three-differenti-rhombi}. A \textbf{hive parallelogram} of $H_n$ is a parallelogram with size lengths parallel to the edges of $H_n$ which is made up of triangles. For example, see Figure~\ref{fig:hive-parallelogram}. 

\begin{Lemma}\label{lem:hive-parallelogram-can-be-split}
    Any hive parallelogram can be partitioned into rhombi. 
\end{Lemma}

\begin{proof}
    The sides of a hive parallelogram will be parallel to two sides of $H_n$. When we divide the parallelogram into pieces cutting along the lines in $H_n$ parallel to these two sides, this gives a partition into rhombi.  
\end{proof}

We say that a labeling of the vertices of the Hive triangle satisfies the \textbf{rhombus inequality} associated to a rhombus $R$ if we have $x_1 + x_2 \geq y_1 + y_2$ where $x_1, x_2$ are the labels of the obtuse vertices of $R$ and $y_1, y_2$ are the labels of the acute vertices of $R$. 

\begin{Definition}
    A \textbf{hive} of size $n$ is a labeling of the vertices in $H_n$ with real numbers which satisfies all rhombus inequalities. A hive is \textbf{integral} if the labeling consists of integers. 
\end{Definition}

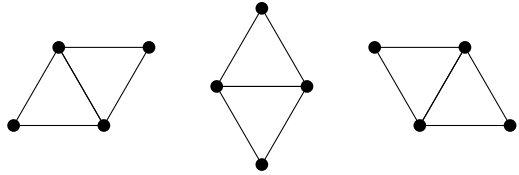
\begin{figure}
\begin{tikzpicture}[scale = 0.6]
    \draw (0, 0) -- (2, 0) -- (1, {sqrt(3)}) -- cycle;
    \draw (1, {sqrt(3)}) -- (2, 0) -- (3, {sqrt(3)}) -- cycle;
    \fill (0, 0) circle (4 pt);
    \fill (2, 0) circle (4 pt);
    \fill (3, {sqrt(3)}) circle (4 pt);
    \fill (1, {sqrt(3)}) circle (4 pt);

    \draw (9, 0) -- (11, 0) -- (10, {sqrt(3)}) -- cycle;
    \draw (9, 0) -- (10, {sqrt(3)}) -- (8, {sqrt(3)}) -- cycle;
    \fill (9, 0) circle (4 pt);
    \fill (11,0) circle (4 pt);
    \fill (10, {sqrt(3)}) circle (4 pt);
    \fill (8, {sqrt(3)}) circle (4 pt);

    \draw (4.5, {sqrt(3)/2}) -- (6.5, {sqrt(3)/2}) -- (5.5, {3*sqrt(3)/2}) -- cycle;
    \draw (4.5, {sqrt(3)/2}) -- (6.5, {sqrt(3)/2}) -- (5.5, {-sqrt(3)/2}) -- cycle;

    \fill (4.5, {sqrt(3)/2}) circle (4 pt);
    \fill (6.5, {sqrt(3)/2}) circle (4 pt);
    \fill (5.5, {3*sqrt(3)/2}) circle (4 pt);
    \fill (5.5, {-sqrt(3)/2}) circle (4 pt);
\end{tikzpicture}
\caption{The three different types of rhombi}\label{fig:three-differenti-rhombi}
\end{figure}

We can generalize the notion of rhombus inequalities to hive parallelograms. Given a hive parallelogram, we have a \textbf{parallelogram inequality} defined in the same way: let $x_1, x_2$ be the numbers on the obtuse angles and let $y_1, y_2$ be the numbers on the acute angles. The parallelogram inequality is $x_1 + x_2 \geq y_1 + y_2$. 

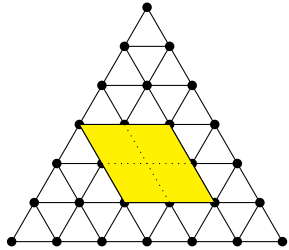
\begin{figure}
    \begin{tikzpicture}[scale=0.3]
        \draw (0,0) -- (2,0) -- (1,{sqrt(3)}) -- cycle;
        \draw (2,0) -- (4,0) -- (3,{sqrt(3)}) -- cycle;
        \draw (4,0) -- (6,0) -- (5,{sqrt(3)}) -- cycle;
        \draw (6,0) -- (8,0) -- (7,{sqrt(3)}) -- cycle;
        \draw (8,0) -- (10,0) -- (9, {sqrt(3)}) -- cycle;
        \draw (10,0) -- (12,0) -- (11, {sqrt(3)}) -- cycle;

        \draw (1, {sqrt(3)}) -- (3, {sqrt(3)}) -- (2, {2*sqrt(3)}) -- cycle;
        \draw (3, {sqrt(3)}) -- (5, {sqrt(3)}) -- (4, {2*sqrt(3)}) -- cycle;
        \draw(5, {sqrt(3)}) -- (7, {sqrt(3)}) -- (6, {2*sqrt(3)}) -- cycle;
        \draw(7, {sqrt(3)}) -- (9, {sqrt(3)}) -- (8, {2*sqrt(3)}) -- cycle; 
        \draw (9, {sqrt(3)}) -- (11, {sqrt(3)}) -- (10, {2 * sqrt(3)}) -- cycle;

        \draw (2, {2*sqrt(3)}) -- (4, {2*sqrt(3)}) -- (3, {3*sqrt(3)}) -- cycle;
        \draw (4, {2*sqrt(3)}) -- (6, {2*sqrt(3)}) -- (5, {3*sqrt(3)}) -- cycle;
        \draw (6, {2*sqrt(3)}) -- (8, {2*sqrt(3)}) -- (7, {3*sqrt(3)}) -- cycle;
        \draw (8, {2*sqrt(3)}) -- (10, {2*sqrt(3)}) -- (9, {3*sqrt(3)}) -- cycle;

        \draw (3, {3*sqrt(3)}) -- (5,{3*sqrt(3)}) -- (4,{4*sqrt(3)}) -- cycle;
        \draw (5, {3*sqrt(3)}) -- (7,{3*sqrt(3)}) -- (6,{4*sqrt(3)}) -- cycle;
        \draw (7, {3*sqrt(3)}) -- (9,{3*sqrt(3)}) -- (8,{4*sqrt(3)}) -- cycle;

        \draw (4, {4*sqrt(3)}) -- (6,{4*sqrt(3)}) -- (5,{5*sqrt(3)}) -- cycle;
        \draw (6, {4*sqrt(3)}) -- (8,{4*sqrt(3)}) -- (7,{5*sqrt(3)}) -- cycle;

        \draw (5,{5*sqrt(3)}) -- (7, {5*sqrt(3)}) -- (6, {6*sqrt(3)}) -- cycle;

        \foreach \x in {0,2,...,12}
        \fill (\x,0) circle (6pt);
        \foreach \x in {1,3,...,11}
        \fill (\x,{sqrt(3)}) circle (6pt);
        \foreach \x in {2,4,...,10}
        \fill (\x,{2*sqrt(3)}) circle (6pt);
        \foreach \x in {3,5,...,9}
        \fill (\x,{3*sqrt(3)}) circle (6pt);
        \foreach \x in {4,6,8}
        \fill (\x,{4*sqrt(3)}) circle (6pt);
        \foreach \x in {5,7}
        \fill (\x,{5*sqrt(3)}) circle (6pt);
        \fill (6,{6*sqrt(3)}) circle (6pt);

        \draw[fill = yellow] (5, {sqrt(3)}) -- (9, {sqrt(3)}) -- (7, {3*sqrt(3)}) -- (3, {3*sqrt(3)}) -- cycle;

        \draw[style = dotted] (4, {2*sqrt(3)}) -- (8, {2*sqrt(3)});
        \draw[style = dotted] (5, {3*sqrt(3)}) -- (7, {sqrt(3)});
    \end{tikzpicture}
    \caption{The yellow parallelogram is a hive parallelogram on $H_5$. The dotted lines indicate a splitting into hive rhombi.}\label{fig:hive-parallelogram}
\end{figure}

\begin{Lemma}\label{lem:hives-parallelogram-inequalities}
    Hives satisfy all parallelogram inequalities. 
\end{Lemma}

\begin{proof}
    From Lemma~\ref{lem:hive-parallelogram-can-be-split}, we can split any hive parallelogram into hive rhombi. Rewriting the rhombus inequalities to 
    \[
        x_1 + x_2 - y_1 - y_2 \geq 0
    \]
    and adding them, we see that after all the cancellations we end up with the desired parallelogram inequality. 
\end{proof}

Hives provide a combinatorial way to interpret the Littlewood-Richardson coefficients. Indeed, the following theorem shows that the Littlewood-Richardson coefficients enumerate integral hives with some boundary condition. For a proof, see \cite{Buch1998}*{Theorem 1}. 

\begin{Theorem}\label{thm:hives}
    Let $\lambda, \mu, \nu$ be partitions with $|\lambda| = |\mu| + |\nu|$ and let $n$ be large enough so that $\lambda$ has at most $n$ positive parts. We view $\lambda, \mu, \nu$ as weakly decreasing non-negative integer vectors in $\RR^n$. Then $c^\lambda_{\mu, \nu}$ is the number of integral hives of size $n$ with the following boundary condition:
    \begin{itemize}
        \item The vertices along the right side of the triangle from the top vertex to the rightmost vertex is given labels:
        \[
            0, \mu_1, \mu_1 + \mu_2, \ldots, \mu_1 + \ldots + \mu_n
        \]
        in this order. 

        \item The vertices along the left side of the triangle from the top vertex to the leftmost vertex is given labels:
        \[
            0, \lambda_1, \lambda_1 + \lambda_2, \ldots, \lambda_1 + \ldots + \lambda_n 
        \]
        in this order. 

        \item The vertices along the bottom side of the triangle from the rightmost vertex to the leftmost vertex is given labels:
        \[
            |\mu|, |\mu| + \nu_1, |\mu| + \nu_1 + \nu_2, \ldots, |\mu| + \nu_1 + \ldots + \nu_n
        \]
        in this order. 
    \end{itemize}
\end{Theorem}

\begin{Example}
    For an example of Theorem~\ref{thm:hives}, consider the partitions $\lambda = (4, 3, 2)$, $\mu = (3, 2, 1)$, and $\nu = (2, 1)$. Then $c^\lambda_{\mu, \nu} = 2$ corresponding to two Littlewood-Richardson tableaux of shape $\lambda / \mu$ and content $\nu$. There are two integral hives of size $3$ satisfying the conditions of Theorem~\ref{thm:hives}. These are shown in Figure~\ref{fig:LR-hive-example}. In the figure, the black labels are the ones dictated by the boundary condition. The red labels are mutable. 
\end{Example}

\begin{figure}
    \begin{tikzpicture}[scale=0.6]
        \draw (-2,0) -- (-8,0) -- (-5,{3*sqrt(3)}) -- cycle;
        \draw (-7, {sqrt(3)}) -- (-3, {sqrt(3)});
        \draw (-6, {2*sqrt(3)}) -- (-4, {2*sqrt(3)});
        \draw (-6, 0) -- (-4, {2*sqrt(3)});
        \draw (-4, 0) -- (-6, {2*sqrt(3)});
        \draw (-6, 0) -- (-7, {sqrt(3)});
        \draw (-4, 0) -- (-3, {sqrt(3)});
        
        \draw (2, 0) -- (8, 0) -- (5, {3 * sqrt(3)}) -- cycle;
        \draw (7, {sqrt(3)}) -- (3, {sqrt(3)});
        \draw (6, {2*sqrt(3)}) -- (4, {2*sqrt(3)});
        \draw (6, 0) -- (4, {2*sqrt(3)});
        \draw (4, 0) -- (6, {2*sqrt(3)});
        \draw (6, 0) -- (7, {sqrt(3)});
        \draw (4, 0) -- (3, {sqrt(3)});

        \fill (-5,{3*sqrt(3)}) circle (3pt);
        \node[yshift=6] at (-5, {3*sqrt(3)}) {0};
        \fill (5,{3*sqrt(3)}) circle (3pt);
        \node[yshift=6] at (5, {3*sqrt(3)}) {0};

        \fill (-6, {2*sqrt(3)}) circle (3pt);
        \node[xshift=-5, yshift=5] at (-6, {2*sqrt(3)}) {4};
        \fill (6, {2*sqrt(3)}) circle (3pt);
        \node[xshift=5, yshift=5] at (6, {2*sqrt(3)}) {3};

        \fill (-4,{2*sqrt(3)}) circle (3pt);
        \node[xshift=5,yshift=5] at (-4,{2*sqrt(3)}) {3};
        \fill (4,{2*sqrt(3)}) circle (3pt);
        \node[xshift=-5,yshift=5] at (4,{2*sqrt(3)}) {4};

        \fill (-7,{sqrt(3)}) circle (3pt);
        \node[xshift=-5,yshift=5] at (-7,{sqrt(3)}) {7};
        \fill (7,{sqrt(3)}) circle (3pt);
        \node[xshift=5,yshift=5] at (7,{sqrt(3)}) {5};

        \fill (-3,{sqrt(3)}) circle (3pt);
        \node[xshift=5,yshift=5] at (-3,{sqrt(3)}) {5};
        \fill (3,{sqrt(3)}) circle (3pt);
        \node[xshift=-5,yshift=5] at (3,{sqrt(3)}) {7};

        \fill (-8,0) circle (3pt);
        \node[xshift=-5,yshift=-5] at (-8,0) {9};
        \fill (8,0) circle (3pt);
        \node[xshift=5,yshift=-5] at (8,0) {6};

        \fill (-6,0) circle (3pt);
        \node[yshift=-6] at (-6,0) {9};
        \fill (6,0) circle (3pt);
        \node[yshift=-6] at (6,0) {8};

        \fill (-4,0) circle (3pt);
        \node[yshift=-6] at (-4,0) {8};
        \fill (4,0) circle (3pt);
        \node[yshift=-6] at (4,0) {9};

        \fill (-2,0) circle (3pt);
        \node[xshift=5,yshift=-5] at (-2,0) {6};
        \fill (2,0) circle (3pt);
        \node[xshift=-5,yshift=-5] at (2,0) {9};

        \fill (-5,{sqrt(3)}) circle (3pt);
        \node[yshift=7, color=red] at (-5,{sqrt(3)}) {6};
        \fill (5,{sqrt(3)}) circle (3pt);
        \node[yshift=7, color=red] at (5,{sqrt(3)}) {7};
        
    \end{tikzpicture}
    \caption{$(\lambda, \mu, \nu)$-hives for $\lambda = (4, 3, 2)$, $\mu = (3, 2, 1)$, and $\nu = (2, 1)$.}\label{fig:LR-hive-example}
\end{figure}
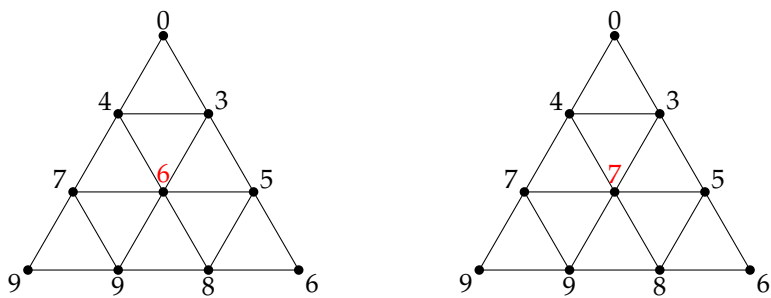

If an integral hive satisfies the conditions in Theorem~\ref{thm:hives} we call it a $(\lambda, \mu, \nu)$-hive. The theorem then states that $c^\lambda_{\mu, \nu}$ is the number of $(\lambda, \mu, \nu)$-hives. We use hives to prove the following result about dominance order. 

\begin{Proposition}\label{prop:LR-lattice-dominance-global-maximum}
    If $c^\lambda_{\mu, \nu} > 0$ then $\lambda \lhd \mu + \nu$. 
\end{Proposition}

\begin{proof}
    Since $c^\lambda_{\mu, \nu}$, there is at least one $(\lambda, \mu, \nu)$-hive. From the parallelogram inequality corresponding to the hive parallelogram with vertices labeled 
    \[
        |\mu|, |\mu| + \nu_1 + \ldots + \nu_k, \lambda_1 + \ldots + \lambda_k, \mu_1 + \ldots + \mu_k. 
    \]
    The obtuse vertices correspond to $|\mu| + \nu_1 + \ldots + \nu_k$ and $\mu_1 + \ldots + \mu_k$. From Lemma~\ref{lem:hives-parallelogram-inequalities}, we get 
    \[
        \sum_{j = 1}^k (\mu_j + \nu_j) \geq \sum_{j} \lambda_j. 
    \]
    This suffices for the proof. 
\end{proof}

In other words, the poset $\LR_{\mu, \nu}$ with partial order given by the dominance order has a global maximum given by $\mu + \nu$ (one can check $c^{\mu + \nu}_{\mu, \nu} = 1$). We will apply Proposition~\ref{prop:LR-lattice-dominance-global-maximum} to prove an asymptotic upper-bound for LR-coefficients in the case where $\lambda \vdash 2n$, $\mu \vdash n$, and $\nu \vdash n$ and $\lambda_1 = 2n - j$ for some constant $j$. This result will also be used in Section~\ref{sec:lower-bound} to prove the lower bound of Theorem~\ref{thm:main-theorem-1}. 

\begin{Lemma}\label{lem:asymptotic-bound-on-LR-coefficient}
    Let $j$ be a non-negative integer. Let $\lambda \vdash N$, $\mu \vdash n$, and $\nu \vdash n$ be partitions with $N = 2n$, and $\lambda_1 = N-j$.
    \begin{enumerate}
        \item If $c^\lambda_{\mu, \nu} > 0$, then for sufficiently large $N$ there exist non-negative integers $i_1, i_2$ such that $\mu_1 = n-i_1$, $\nu_1 = n-i_2$, and $i_1 + i_2 \leq j$. 
        \item There is some constant $C_j$ depending on $j$ such that $c^\lambda_{\mu, \nu} \leq C_j$ for sufficiently large $N$. 
    \end{enumerate}
\end{Lemma}

\begin{proof}
    From Proposition~\ref{prop:LR-lattice-dominance-global-maximum}, we know that $\mu + \nu$ dominates $\lambda$. In particular, $\mu_1 + \nu_1 \geq \lambda_1$. Letting $i_1 = n - \mu_1$ and $i_2 = n - \nu_1$, we get $i_1 + i_2 \leq j$. This completes the proof of part (1).

    For part (2), we can suppose that $c^\lambda_{\mu, \nu} > 0$. For sufficiently large $N$, we have $\lambda = (N-j, T)$, $\mu = (n-i_1, T_1)$, and $\nu = (n-i_2, T_2)$ where $T \vdash j$, $T_1 \vdash i_1$, and $T_2 \vdash i_2$ are partitions. The number $c^\lambda_{\mu, \nu}$ counts the number of \emph{Littlewood-Richardson tableaux} of shape $\lambda / \mu$ and content $\nu$. This is bounded above by the number of \emph{semistandard Young tableaux} of shape $\lambda / \mu$ and content $\nu$. In any such tableaux, the first row of $\lambda / \mu$ must be filled with $1$'s. Thus, $c^\lambda_{\mu, \nu}$ is bounded above by the number of semistandard Young Tableaux of shape $T / T_1$ and content $(j - (i_1 + i_2), T_2)$. But this is clearly bounded above by a constant depending only on $j$. 
\end{proof}

\subsection{Distances between probability measures}\label{sec:distances-p-measures}

For our purposes, we will need two notions of distance between probability measures. All of our probability measures will be on the symmetric group, but the distances can be defined between any pair of discrete probability measures. We first define the \emph{total variation distance} which is used in the statement of Theorem~\ref{thm:main-theorem-1}. 

\begin{Definition}
    Let $X$ be a finite set and let $\mu, \nu$ be two probability measures on $X$. We define the \textbf{total variation distance} between $\mu$ and $\nu$ to be the quantity 
    \[
        d_{\TV}(\mu, \nu) \eqdef \frac{1}{2} \sum_{x \in X} |\mu(x) - \nu(x)| = \sup_{A \subseteq X} |\mu(A) - \nu(A)|. 
    \]
\end{Definition}

In the proof of the lower bound of Theorem~\ref{thm:main-theorem-1} in Section~\ref{sec:lower-bound}, we will the use the notion of the \emph{Hellinger distance}.

\begin{Definition}
    Let $X$ be a finite set and let $\mu, \nu$ be two probability measures on $X$. We define the (squared) \textbf{Hellinger distance} between $\mu$ and $\nu$ to be the quantity
    \[
        H^2(\mu, \nu) \eqdef \frac{1}{2} \sum_{x \in \ZZ} \left ( \sqrt{\mu(x)} - \sqrt{\nu(x)} \right )^2 = 1 - \sum_{x \in \ZZ} \sqrt{\mu(x)\nu(x)}.
    \]
\end{Definition}

From~\cite{LeCam2000}*{page 44}, the total variation distance is always bounded below by the (squared) Hellinger distance. That is, for all distributions $p, q$ we have  
\begin{equation}\label{eqn:HELL>=TV}
    H^2(p, q) \leq d_{\TV}(p, q).
\end{equation}
Since $H^2(\Poiss(1), \Poiss(1+x)) = 1 - e^{- \frac{1}{2}(\sqrt{1+x} - 1)^2}$, Equation~\ref{eqn:HELL>=TV} gives the inequality 
\begin{equation}\label{eqn:poissTV>=LOWERBOUND}
    d_{\TV} (\Poiss(1), \Poiss(1+x)) \geq 1 - e^{- \frac{1}{2}(\sqrt{1+x} - 1)^2}.
\end{equation}

\section{Proof of Theorem~\ref{thm:main-theorem-2}}\label{sec:proof-of-theorem-2}

In this section, we complete the proof of Theorem~\ref{thm:main-theorem-2}. We will accomplish this in three steps. First, we diagonalize linear operator on $\CC[\fS_N]$ given by left multiplication by $\cA$. The spectrum of this operator is exactly the same as the spectrum in Theorem~\ref{thm:main-theorem-2}. Second, we prove that the matrix (with respect to the standard basis of $\CC[\fS_N]$) of $\cA$ is the transpose of $P$. Third, we will complete the proof of Theorem~\ref{thm:main-theorem-2} by using the fact that $P$ is symmetric. 

\begin{Remark}
    Since we work with group representations, we often identify elements of the group algebra with their corresponding endomorphism on underlying vector space of the representation. In our setting, we work with the regular representation $\CC[\fS_N]$, and we can view every $x \in \CC[\fS_N]$ as an element of $\End_\CC(\CC[\fS_N])$. Thus, when we refer to the eigenvalues, trace, etc. of $x \in \CC[fS_N]$, we mean the eigenvalues, trace, etc. of the corresponding endomorphism. In particular, we use this language starting in the following theorem.  
\end{Remark}

\begin{Theorem}
    The group algebra element $\cA$ has eigenvalue spectrum $\Eig^\lambda_{\mu, \nu}$ with multiplicities $f_\lambda f_\mu f_\nu c^\lambda_{\mu, \nu}$ for all $\lambda \vdash N$, $\mu \vdash |A|$, and $\nu \vdash |B|$. 
\end{Theorem}

\begin{proof}
    From Proposition~\ref{prop:rep-theory-facts-sym}(4), we can decompose $\CC[\fS_N]$ into irreducible subrepresentations
    \[
        \CC[\fS_N] = \bigoplus_{\lambda \vdash N} (S^\lambda)^{\oplus f_\lambda}. 
    \]
    Since subrepresentations are invariant under the $\CC[\fS_N]$ action, if we want to find the specturm of $\mathscr{A}$ on $\CC[\fS_N]$, it suffices to find the spectrum on each individual Specht module $S^\lambda$. 

    From Proposition~\ref{prop:rep-theory-facts-sym}, $\cT_{A \cup B}$ acts on $S^\lambda$ by as scalar multiplication by $\Diag(\lambda)$. It remains to diagonalize the action of  
    \[
        \mathscr{U} \eqdef \frac{2(a^2-ab)}{N^2} \cT_A + \frac{2(b^2-ab)}{N^2} \cT_B
    \]
    on $S^\lambda$. Let $\fS_A \times \fS_B \subset \fS_N$ denote the subgroup of permutations which permutes the cards of $A$ among each other, and the cards in $B$ among each other. Then, we can view $\cU$ as a group algebra element in $\CC[\fS_A \times \fS_B]$. Theorem~\ref{thm:LR-rule} implies that the Specht module $S^\lambda$, as a $\fS_A \times \fS_B$ representation, can be decomposed into irreducible $\fS_A \times \fS_B$ subrepresentations as 
    \[
        \Res^{\fS_N}_{\fS_A \times \fS_B} S^\lambda = \bigoplus_{\substack{\mu \vdash |A| \\ \nu \vdash |B|}} \left (S^\mu \boxtimes S^\nu \right )^{\oplus c^\lambda_{\mu, \nu}}.
    \]
    Each subrepresentation is invariant under action of $\CC[\fS_A \times \fS_B]$, in particular, the action of $\cU$. It remains to diagonalize the action of $\cU$ on each $S^\mu \boxtimes S^\nu$. But this is easy to do since $\cU$ is a central element in $\CC[\fS_A \times \fS_B]$ and thus acts as a scalar multiple on each $S^\mu \boxtimes S^\nu$ by Schur's lemma. To find the explicit scalar multiple, let $u \boxtimes v \in S^\mu \boxtimes S^\nu$ be an arbitrary simple tensor. Then, 
    \begin{align*}
        \cU (u \boxtimes v) & = \frac{2(a^2 - ab)}{N^2} ( \cT_A u \boxtimes v) + \frac{2(b^2 - ab)}{N^2} (u \boxtimes \cT_B v) \\
        & = \left ( \frac{2(a^2-ab)}{N^2} \Diag(\mu) + \frac{2(b^2-ab)}{N^2} \Diag(\nu) \right ) (u \boxtimes v).
    \end{align*}
    So, $\cA$ acts by scalar multiplication on each factor of $S^\mu \boxtimes S^\nu$ with scalar multiple $\Eig^\lambda_{\mu, \nu}$. There are $f^\lambda c^\lambda_{\mu, \nu}$ copies of $S^\mu \boxtimes S^\nu$. Since $\dim (S^\mu \boxtimes S^\nu) = f_\mu f_\nu$, this implies that the mutliplicity of $\Eig^\lambda_{\mu, \nu}$ is exactly $f_\lambda f_\mu f_\nu c^\lambda_{\mu, \nu}$.
\end{proof}

\begin{Lemma}\label{lem:transpose-is-leftmult}
    The matrix with respect to the standard basis $\{g : g \in \fS_N \} \subset \CC[\fS_N]$ of $\cA$  is the transpose of $P$.  
\end{Lemma}

\begin{proof}
    For $g \in \fS_N$, we have 
    \begin{align*}
        \mathscr{A} g & = \sum_{x \in \fS_N} P(e, x) \cdot xg = \sum_{x \in \fS_N} P(e, xg^{-1}) x = \sum_{x \in \fS_N} P(g, x) x.
    \end{align*}
    This means that the column indexed by $g$ of the matrix of $\cA$ is the row indexed by $g$ of $P$. This suffices for the proof.
\end{proof}

\begin{proof}[Proof of Theorem~\ref{thm:main-theorem-2}]
    From Equation~\ref{eqn:explicit-trans-matrix}, the matrix $P$ is symmetric. Thus, it has the same eigenvalues as its transpose. Lemma~\ref{lem:transpose-is-leftmult} and Theorem~\ref{thm:main-theorem-3} completes the proof. 
\end{proof}

\section{Upper Bound}\label{sec:proof-of-upper-bound}

In this section we prove Theorem~\ref{thm:main-theorem-1}. The proof initially follows the same structure of the proof in~\cite{Diaconis1981} in the case of random transpositions, but immediately complications arise and new ideas are needed.

In~\cite{Diaconis1981}, Diaconis and Shahshahani upper bound the total variation distance by a sum indexed by non-trivial partitions. To bound this sum, they split up the space of partitions into zones and then bound the sum over each zone using different arguments. In our case, we similarly upper bound the total variation distance by a sum indexed by partitions (see Section~\ref{sec:l2-bound}) and then we divide this space into several zones to bound each zone separately. However, we will have several more zones and need to be more careful in our analysis of each zone. The partitioning of our space into zones also has an additional \emph{dynamical} component where some of our zones are determined by the dynamical behavior of certain functions that we define later. The main terms in both $\mathsf{RT}$ and $\mathsf{bRT}$ remain largely the same: the main terms correspond to the partitions with long first row or long first column.

In Section~\ref{sec:l2-bound}, we provide an $\ell_2$-bound on the total variation in the form of a sum over partitions. In Section~\ref{sec:useful-estimate}, we provide a useful estimate for bounding this sum. In Section~\ref{sec:zones}, we describe the zones. Approximately, we have three main zones: \emph{red, blue, and yellow}. In Section~\ref{sec:red-zone}, we bound the sum over the red zone. In Section~\ref{sec:blue-zone}, we bound the sum over the blue zone. Finally, in Section~\ref{sec:yellow-zone}, we bound the sum over the yellow zone. The contribution from the red and blue zone will turn out to be very small (on the order of $\exp(-n \log n)$). The main terms will lie in the yellow zone. 

Recall, in the setting of Theorem~\ref{thm:main-theorem-1} we specialize the biased random transposition shuffle to the case $|A| = |B| = n$ and $N = 2n$. In this specialization, our eigenvalue becomes
\begin{equation}\label{eqn:eigenvalue-when-AandB-same}
    \Eig^\lambda_{\mu, \nu} \eqdef \frac{a^2 + b^2}{4n} + \frac{a^2 - ab}{2n^2} \Diag(\mu) + \frac{b^2 -ab}{2n^2} \Diag(\nu) + \frac{ab}{2n^2} \Diag(\lambda). 
\end{equation}

For the rest of this paper, we will define $t_N \eqdef \frac{1}{2b} N \log N$.

\subsection{\texorpdfstring{$\ell_2$}{l2} bound on total variation}\label{sec:l2-bound}

The $\bRT$ shuffle is reversible and transitive. Since we know the eigenvalues of the transition matrix, we can bound the total variation distance using a general $\ell_2$-bound for reversible, transitive Markov chains.  

\begin{Lemma}\label{lem:l2-bound}
    Let $|A| = |B| = n$ and $N = 2n$. Let $U$ be the uniform distribution on $\fS_N$. Then, any time $t \geq 0$, we have the bound 
    \begin{equation}\label{eqn:l2-bound}
        d_{\TV} (P^t(x, \bullet), U(\bullet))^2 \leq \frac{1}{4} \sum_{\substack{\lambda \in \YY_{2n} \\ \lambda \neq (2n)}} \sum_{(\mu, \nu) \in \LR^\lambda} c^\lambda_{\mu, \nu} f_\lambda f_\mu f_\nu \cdot |\Eig^\lambda_{\mu, \nu}|^{2t}.
    \end{equation}
\end{Lemma}
\begin{proof}
    This is an application of~\cite{LPW-MCMT}*{Lemma 12.18(ii)} along with Theorem~\ref{thm:main-theorem-2}.
\end{proof}

\subsection{Estimate for \texorpdfstring{$\Omega_\lambda(t)$}{Omega}}\label{sec:useful-estimate}

Note that the sum of the right hand side of Equation~\ref{eqn:l2-bound} have summands indexed by $\lambda$. We consolidate these terms to create an \emph{error term} associated with the partition $\lambda$. 

\begin{Definition}
    For $\lambda \vdash 2n$ and $\lambda \neq (2n)$. For any $t$, we define the \textbf{error term} associated with $\lambda$ to be the quantity
    \begin{equation}\label{eqn:main-bounding-term}
        \Omega_\lambda (t) \eqdef \sum_{(\mu, \nu) \in \LR^\lambda} c^\lambda_{\mu, \nu} f_\lambda f_\mu f_\nu \cdot |\Eig^\lambda_{\mu, \nu} |^{2t}. 
    \end{equation}
\end{Definition}

Note that the right hand side of Equation~\ref{eqn:l2-bound} is a sum of $\Omega_\lambda(t)$ indexed by $\lambda \in \mathbb{Y}_{2n}$. We will want to split the sum into different subsets $Z \subseteq \mathbb{Y}_{2n}$. If we want to bound the sum over a specific subset $Z$, then the following lemma gives a way to circumvent working with Littlewood-Richardson coefficients. 

\begin{Lemma}\label{lem:essential-tool}
    Let $Z \subset \YY_{2n}$ be some subset of partitions of size $2n$. Then for all $t \geq 0$ we have the bound 
    \begin{align*}
        \sum_{\lambda \in Z} \Omega_\lambda(t) & \leq \sum_{\lambda \in Z} (f_\lambda)^2 \max_{(\mu, \nu) \in \LR^\lambda} |\Eig^\lambda_{\mu, \nu}|^{2t} \\
        & \leq \max_{\lambda \in Z} (f_\lambda)^2 \max_{\substack{\lambda \in Z \\ (\mu, \nu) \in \LR^\lambda}} |\Eig^\lambda_{\mu, \nu}|^{2t} \cdot \exp(O(\sqrt{n})).
    \end{align*}
\end{Lemma}

\begin{proof}
    Recall that 
    \begin{align*}
        \sum_{\lambda \in Z} \Omega_\lambda (t) & = \sum_{\lambda \in Z}  \sum_{(\mu, \nu) \in \LR^\lambda} c^\lambda_{\mu, \nu} f_\lambda f_\mu f_\nu \cdot |\Eig^\lambda_{\mu, \nu} |^{2t}
    \end{align*}

    After replacing $|\Eig^\lambda_{\mu, \nu}|$ by its maximum over $(\mu, \nu) \in \LR^\lambda$, the first inequality follows from Proposition~\ref{prop:sym-rep-theory-identities}(2). The second inequality follows from the first by replacing summand by its maximum and using Theorem~\ref{thm:partition-function-bound} since there are at most $p(2n)$ terms. 
\end{proof}

\subsection{Zones}\label{sec:zones}

In this section, we describe the zones. Pictorially, these zones can be represented on the positive quadrant where the $x$-axis represents the value of $\lambda_1$ and the $y$-axis represents the value of $\lambda_1^*$. For a picture of all of the zones, see Figure~\ref{fig:zones-skeleton}. 

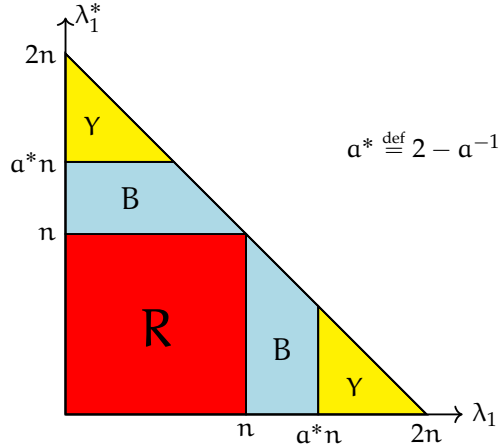
\begin{figure}
    \begin{tikzpicture}[scale = 1.2]
        \definecolor{lightblue}{rgb}{0.68, 0.85, 0.90}

        \draw[->, line width=0.25mm] (0,0) -- (4.4,0) node[right] {$\lambda_1$};
        \draw[->, line width=0.25mm] (0,0) -- (0,4.4) node[right] {$\lambda_1^*$};

        \draw[fill = red, opacity=0.5] (0, 0) -- (2, 0) -- (2, 2) -- (0, 2) -- cycle;
        \draw[line width=0.25mm] (0, 0) -- (2, 0) -- (2, 2) -- (0, 2) -- cycle;

        \draw[fill = lightblue] (0, 2) -- (2, 2) -- (1.2, 2.8) -- (0, 2.8) -- cycle;
        \draw[line width=0.25mm] (0, 2) -- (2, 2) -- (1.2, 2.8) -- (0, 2.8) -- cycle;

        \draw[fill = lightblue] (2, 0) -- (2, 2) -- (2.8,1.2) -- (2.8,0) -- cycle;
        \draw[line width=0.25mm] (2, 0) -- (2, 2) -- (2.8,1.2) -- (2.8,0) -- cycle;

        \draw[fill = yellow, opacity=0.5] (2.8, 0) -- (2.8, 1.2) -- (4, 0) -- cycle;
        \draw[line width=0.25mm] (2.8, 0) -- (2.8, 1.2) -- (4, 0) -- cycle;

        \draw[fill = yellow, opacity=0.5] (0, 2.8) -- (1.2, 2.8) -- (0, 4) -- cycle;
        \draw[line width=0.25mm] (0, 2.8) -- (1.2, 2.8) -- (0, 4) -- cycle;

        \draw[line width=0.25mm] (0,0) -- (0,4) -- (4,0) -- cycle;

        \draw (0,0) -- (2,0) node[below] {$n$} -- (2.8,0) node[below] {$a^*n$} -- (4,0) node[below] {$2n$};
        \draw (0,0) -- (0,2) node[left] {$n$} -- (0,2.8) node[left] {$a^*n$} -- (0,4) node[left] {$2n$};

        \draw (3, 3) node[right] {$a^* \eqdef 2-a^{-1}$};

        \draw (1, 0.65) node[above] {\Huge $R$};
        \draw (2.4, 0.5) node[above] {\Large $B$};
        \draw (0.5, 2.4) node[right] {\Large $B$};
        \draw (3.2, 0.1) node[above] {$Y$};
        \draw (0.1, 3.2) node[right] {$Y$};
    \end{tikzpicture}
    \caption{Cartesian plane with $x$-axis $\lambda_1$ and $y$-axis $\lambda_1^*$}\label{fig:zones-skeleton}
\end{figure}

Each point $(x,y)$ in the plane represents the sum of $\Omega_\lambda (t)$ where the sum ranges over all $\lambda \vdash 2n$ with more than one row with first row $x$ and first column $y$. Looking at Figure~\ref{fig:zones-skeleton}, we have drawn approximately three major zones. These are denoted by $\Zone_{R}$, $\Zone_{B}$, $\Zone_{Y}$ ("R" stands for \textbf{R}ed, "B" stands for \textbf{B}lue, and "Y" stands for \textbf{Y}ellow). We define the zones as follows: 
\begin{Definition}
    Let $a^* \eqdef 2-a^{-1}$. We define the \textbf{Red}, \textbf{Blue}, and \textbf{Yellow} zones explicitly by 
    \begin{align*}
        \Zone_R & \eqdef \{\lambda : \lambda_1, \lambda_1^* \leq n \}, \\
        \Zone_B & \eqdef \{ \lambda : n \leq \lambda_1 \leq a^*n, \lambda_1^* \leq n\} \cup \{\lambda : n \leq \lambda_1^* \leq a^*n, \lambda_1 \leq n \}, \\
        \Zone_Y & \eqdef \{\lambda : 2n-1 \geq \lambda_1 \geq a^* n \} \cup \{ \lambda : \lambda_1^* \geq a^*n \}. 
    \end{align*}
\end{Definition}

\begin{Remark}
    These \emph{are not exactly} the zones we will use, but they are a good approximation. We will use the Red zone as it is defined. For the Blue zone, we will extend a bit to the right to create the \emph{modified Blue zone}. For the Yellow zone, we will retract the left side a bit to the right to create the \emph{modified Yellow zone}. 
\end{Remark}

\begin{Remark}[The Window]
    We want to prove that $t_N$ is the cutoff time with window $N$. To prove the upper bound, we want to show that 
    \[
        \sum_{\lambda \in \Zone_R} \Omega_\lambda (t_N + cN) + \sum_{\lambda \in \Zone_B} \Omega_\lambda (t_N + cN) + \sum_{\lambda \in \Zone_Y} \Omega_\lambda (t_N + cN)
    \]
    is small. However, since our main terms lie in the yellow zones, the sums over the Red zone and Blue zones are small even before the window. So, in our proof, we will prove that the larger sum 
    \[
        \sum_{\lambda \in \Zone_R} \Omega_\lambda (t_N) + \sum_{\lambda \in \Zone_B} \Omega_\lambda (t_N) + \sum_{\lambda \in \Zone_Y} \Omega_\lambda (t_N + cN)
    \]
    is small where the sums of the Red zone and Blue zones end up being on the order of $\exp (-C N \log N)$. 
\end{Remark}

\subsection{The Red Zone}\label{sec:red-zone}

We want to prove that $\sum_{\lambda \in \Zone_R} \Omega_\lambda \left (t_N \right )$ is small as $N \to \infty$. To bound this sum, we will further partition the zone $\Zone_R$ into smaller subzones. We will then bound these subzones separately. 

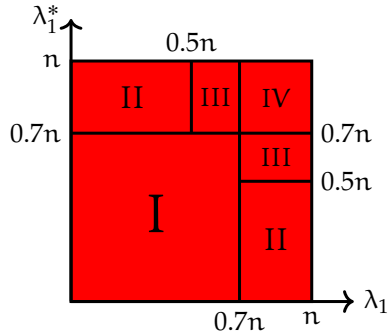
\begin{figure}
    \begin{tikzpicture}[scale = 0.8]
        \draw[fill = red, opacity = 0.4] (0, 0) -- (4, 0) -- (4, 4) -- (0, 4) -- cycle;
        \draw[line width=0.4mm] (0,0)--(4,0)--(4,4)--(0,4)--cycle;

        \draw[line width=0.4mm] (2.8, 0) -- (2.8, 4);
        \draw[line width=0.4mm] (0, 2.8) -- (4, 2.8);
        \draw[line width=0.4mm] (2.8, 2) -- (4, 2);
        \draw[line width=0.4mm] (2, 2.8) -- (2, 4);

        \draw[->, line width=0.4mm] (0,0) -- (2.8,0) node[below] {$0.7n$} -- (4,0) node[below] {$n$} -- (4.7,0) node[right] {$\lambda_1$};
        \draw[->, line width=0.4mm] (0,0) -- (0,2.8) node[left] {$0.7n$} -- (0,4) node[left] {$n$} -- (0,4.7) node[left] {$\lambda_1^*$};

        \draw (2,4) node[above] {$0.5n$};
        \draw (1.4, 1) node[above] {\Huge $I$};

        \draw (3.4, 0.7) node[above] {\Large $II$};
        \draw (0.65, 3.4) node[right] {\Large $II$};

        \draw (3.4, 2.1) node[above] {$III$};
        \draw (1.98, 3.4) node[right] {$III$};

        \draw (3.8, 3.39) node[left] {$IV$};

        \draw (4,0) -- (4,2) node[right] {$0.5n$} -- (4,2.8) node[right] {$0.7n$};
    \end{tikzpicture}
    \caption{Subzones of $\Zone_R$}\label{fig:subzones-of-redzone}
\end{figure}

The subzones are drawn in Figure~\ref{fig:subzones-of-redzone} labeled with roman numerals $I$, $II$, $III$, and $IV$. Explicitly, these subzones are 
\begin{align*}
    \Zone_R^{I} & \eqdef \{\lambda_1, \lambda_1^* \leq 0.7n \}, \\
    \Zone_R^{II} & \eqdef \{0.7n \leq \lambda_1 \leq n, \lambda_1^* \leq 0.5n\} \cup \{0.7n \leq \lambda_1^* \leq n, \lambda_1 \leq 0.5n \}, \\
    \Zone_R^{III} & \eqdef \{0.7n \leq \lambda_1 \leq n, 0.5n \leq \lambda_1^* \leq 0.7n\} \cup \{0.7 \leq \lambda_1^* \leq n, 0.5n \leq \lambda_1 \leq 0.7n \}, \\
    \Zone_R^{IV} & \eqdef \{0.7n \leq \lambda_1, \lambda_1^* \leq n \}. 
\end{align*}

For the analysis of $\Zone_R^{III}$ and $\Zone_R^{IV}$ we will the estimate for $\Eig^\lambda_{\mu, \nu}$ given in Proposition~\ref{prop:BOUND-on-EIG-between-0.5-and-1}. We define our first auxiliary function. 

\begin{Definition}\label{defn:q-function}
    We define a function $\mathbf{Q}_R : \RR^2 \to \RR$ defined by 
    \[
        \mathbf{Q}_R(x, y) \eqdef \frac{a^2 - b^2}{4} + \frac{a^2 - ab}{2} (x^2 - x) + \frac{ab - b^2}{2} (y^2 - y) + \frac{ab}{2} \left( x - \frac{xy}{2} - \frac{y^2}{2} \right).
    \]
\end{Definition}

\begin{Proposition}\label{prop:BOUND-on-EIG-between-0.5-and-1}
    Let $\lambda \vdash 2n$ with $0.5n \leq \lambda_1, \lambda_1^* \leq n$. Then, we have 
    \[
        \max_{(\mu, \nu) \in \LR^\lambda} |\Eig^\lambda_{\mu, \nu}| \leq \max \left\{ \mathbf{Q}_R \left( \frac{\lambda_1}{n}, \frac{\lambda_1^*}{n} \right) , \mathbf{Q}_R \left( \frac{\lambda_1^*}{n}, \frac{\lambda_1}{n} \right) \right\} + O \left( \frac{1}{n} \right)
    \]
    where the $O(1/n)$ error depends only on $a, b, n$. 
\end{Proposition}

\begin{proof}
    Let $\lambda \vdash 2n$ be a partition with $0.5 n \leq \lambda_1, \lambda_1^* \leq n$. Let $(\mu, \nu) \in \LR^\lambda$. Then Lemma~\ref{lem:generic-diagonal-sum-bound} and Lemma~\ref{lem:oliver} give us the bounds 
    \begin{align*}
        \Diag(\lambda) & \leq n \lambda_1 - \frac{\lambda_1 \lambda_1^*}{2} - \frac{(\lambda_1^*)^2}{2} + O(n) \\ 
        \Diag (\mu) & \leq \Diag(\lambda_1, n- \lambda_1) \leq \lambda_1^2 - n \lambda_1 + \frac{n^2}{2} \\ 
        \Diag(\nu) & \geq - \Diag(\lambda_1^*, n - \lambda_1^*) \geq - \left\{ (\lambda_1^*)^2 - n \lambda_1^* + \frac{n^2}{2} \right\}.
    \end{align*}
    This gives the upper bound 
    \[
        \Eig^\lambda_{\mu, \nu} \leq \mathbf{Q}_R \left( \frac{\lambda_1}{n}, \frac{\lambda_1^*}{n} \right) + O\left( \frac{1}{n} \right).
    \]
    We can repeat the argument for the lower bound. Indeed, Lemma~\ref{lem:essential-tool} and Lemma~\ref{lem:oliver} give us the bounds 
    \begin{align*}
        \Diag(\lambda) & \geq - \left\{  n \lambda_1^* - \frac{\lambda_1 \lambda_1^*}{2} - \frac{\lambda_1^2}{2} \right\} + O(n) \\
        \Diag(\mu) & \geq - \Diag(\lambda_1^*, n - \lambda_1^*) \geq - \left\{ (\lambda_1^*)^2 - n \lambda_1^* + \frac{n^2}{2} \right\} \\
        \Diag(\nu) & \leq \Diag(\lambda_1, n- \lambda_1) \leq \lambda_1^2 - n \lambda_1 + \frac{n^2}{2}. 
    \end{align*}
    This gives the lower bound 
    \[
        \Eig^\lambda_{\mu, \nu} \geq -\mathbf{Q}_R \left( \frac{\lambda_1^*}{n}, \frac{\lambda_1}{n} \right) + O \left ( \frac{1}{n} \right ). 
    \]
    This completes the proof. 
\end{proof}

\subsubsection{The Red Zone I}

In this section, we consider the sum of $\Omega_\lambda(t)$ over $\lambda \in \Zone_R^I$. For this sub-zone, we only need to use the bounds already found in~\cite{Diaconis1981}. In this paper, Diaconis and Shahshahani bound the following quantity: 

\begin{Definition}
    For $\mu \vdash n$, we define 
    \[
        D_\mu \eqdef \frac{1}{n} + \frac{2}{n^2} \Diag(\mu). 
    \]
\end{Definition}

The eigenvalues for random transpositions are exactly $D_\mu$ for $\mu \vdash n$ with multiplicity $f_\mu^2$. Recall from Proposition~\ref{prop:LR-lattice-dominance-global-maximum} we proved that when $c^\lambda_{\mu, \nu} > 0$ we have $\lambda \lhd \mu + \nu$. We can use this to bound $\Eig^\lambda_{\mu, \nu}$ based on $\mu$ and $\nu$. We write the bound in terms of the $D_\mu$. 

\begin{Proposition}
    Let $c^\lambda_{\mu, \nu} > 0$. Then we have 
    \[
        \frac{a^2}{4} D_\mu + \frac{b^2}{4} D_\nu - \frac{ab}{2} \frac{\langle \mu^*, \nu^* \rangle }{n^2} \leq \Eig^\lambda_{\mu, \nu} \leq \frac{a^2}{4} D_\mu + \frac{b^2}{4} D_\nu + \frac{ab}{2} \frac{\langle \mu, \nu \rangle}{n^2}. 
    \]
\end{Proposition}

\begin{proof}
    From Proposition~\ref{prop:sym-rep-theory-identities}(4), we have $c^{\lambda^*}_{\mu^*, \nu^*} = c^\lambda_{\mu, \nu} > 0$. From this fact and Proposition~\ref{prop:LR-lattice-dominance-global-maximum}, we have $\lambda \lhd \mu + \nu$ and $\lambda^* \lhd \mu^* + \nu^*$. From Lemma~\ref{lem:oliver}, we get the inequalities
    \[
        -\Diag(\mu^* + \nu^*) \leq \Diag(\lambda) \leq \Diag(\mu + \nu).
    \]
    By applying these bounds to Equation~\ref{eqn:eigenvalue-when-AandB-same} and using the expansion in Lemma~\ref{lem:diag-of-addition-can-be-expanded}, we get the desired bound. 
\end{proof}

As a corollary, we obtain the following bound. 

\begin{Corollary}\label{cor:Eig-bound-in-terms-of-RT}
    Let $c^\lambda_{\mu, \nu} > 0$. Then we have 
    \[
        |\Eig^\lambda_{\mu, \nu}| \leq \frac{a^2}{4} |D_\mu| + \frac{b^2}{4} |D_\nu| + \frac{ab}{2} \max \left \{ \frac{\langle \mu, \nu \rangle}{n^2}, \frac{\langle \mu^*, \nu^* \rangle}{n^2} \right \}
    \]
\end{Corollary} 

In our current subzone, Diaconis and Shahshahani give bounds for the random transposition eigenvalues $D_\mu$ and $D_\nu$ in the proof of~\cite{Diaconis1981}*{Theorem 1}. We summarize their results in the following lemma. 

\begin{Lemma}\label{lem:diaconis-inequalities}
    For $\mu \vdash n$, the following bounds hold:
    \begin{enumerate}
        \item If $\mu_1, \mu_1^* \leq \frac{n}{3}$ then 
        \[
            |D_\mu| \leq \frac{1}{3}, \quad (f_\mu)^2 \leq \exp (n \log n + O(n)).
        \]

        \item If $\frac{n}{3} < \mu_1 \leq \frac{n}{2}$, $\mu_1^* \leq \frac{n}{2}$ or $\mu_1 \leq \frac{n}{2}$, $\frac{n}{3} < \mu_1^* \leq \frac{n}{2}$ then 
        \[
            |D_\mu| \leq \frac{1}{2}, \quad (f_\mu)^2 \leq \exp \left ( \frac{2}{3}n \log n + O(n) \right )
        \]

        \item If $0.5n < \mu_1 \leq 0.7n$ or $0.5n < \mu_1^* \leq 0.7n$ then 
        \[
            |D_\mu| \leq 0.58 + O \left ( \frac{1}{n} \right ), \quad (f_\mu)^2 \leq \exp \left ( \frac{1}{2} n\log n + O(n) \right ). 
        \]
    \end{enumerate}
\end{Lemma}

For this current subzone, our goal is to bound the sum 
\begin{equation}\label{eqn:DS-subzone-red}
    \sum_{\lambda_1, \lambda_1^* \leq 0.7n} \sum_{(\mu, \nu) \in \LR^\lambda} c^\lambda_{\mu, \nu} f_\lambda f_\mu f_\nu \cdot | \Eig^\lambda_{\mu, \nu} |^{2t}. 
\end{equation}
Since we are summing over $(\mu, \nu) \in \LR^\lambda$, we must have $\mu \subseteq \lambda$ and $\nu \subseteq \lambda$. This implies that $\mu_1, \mu_1^*, \nu_1, \nu_1^*$ are \emph{all} at most $0.7n$. Swapping the order of summation, this implies that our sum is bounded above by 
\begin{align*}
   (\ref{eqn:DS-subzone-red}) & \leq \sum_{\substack{\mu_1, \mu_1^* \leq 0.7n \\ \nu_1, \nu_1^* \leq 0.7n}} f_\mu f_\nu \sum_{\lambda \in \LR_{\mu, \nu}}  f_\lambda c^\lambda_{\mu, \nu} |\Eig^\lambda_{\mu, \nu}|^{2t_n} \\ 
    & \leq \sum_{\substack{\mu_1, \mu_1^* \leq 0.7n \\ \nu_1, \nu_1^* \leq 0.7n}} f_\mu f_\nu \max_{\lambda \in \LR_{\mu, \nu}} |\Eig^\lambda_{\mu, \nu}|^{2t_n} \sum_{\lambda \in \LR_{\mu, \nu}} f_\lambda c^\lambda_{\mu, \nu}\\ 
    & = \sum_{\substack{\mu_1, \mu_1^* \leq 0.7n \\ \nu_1, \nu_1^* \leq 0.7n}} \binom{2n}{n} (f_\mu f_\nu)^2 \max_{\lambda \in \LR_{\mu, \nu}} |\Eig^\lambda_{\mu, \nu}|^{2t_n} \\
    & = \exp ( O(n) ) \sum_{\substack{\mu_1, \mu_1^* \leq 0.7n \\ \nu_1, \nu_1^* \leq 0.7n}} (f_\mu f_\nu)^2 \max_{\lambda \in \LR_{\mu, \nu}} |\Eig^\lambda_{\mu, \nu}|^{2t_n}
\end{align*}
where we used Proposition~\ref{prop:sym-rep-theory-identities} in the third equality. To summarize, we have proven the following lemma.

\begin{Lemma}\label{lem:bound-on-error-Zone-Red-I}
    For all times $t$, we have 
    \begin{equation}\label{eqn:zoneRI-intermediate-bound}
        \sum_{\lambda \in \Zone_R^I} \Omega_\lambda(t) \leq \exp ( O(n) ) \sum_{\substack{\mu_1, \mu_1^* \leq 0.7n \\ \nu_1, \nu_1^* \leq 0.7n}} (f_\mu f_\nu)^2 \max_{\lambda \in \LR_{\mu, \nu}} |\Eig^\lambda_{\mu, \nu}|^{2t}.
    \end{equation}
\end{Lemma}

To bound the right hand side of Equation~\ref{eqn:zoneRI-intermediate-bound}, we split up the sum into different parts and use Lemma~\ref{lem:diaconis-inequalities} to bound each part. Define the following subsets of $\YY_n$: 
\begin{align*}
    W_1 & \eqdef \left \{\mu_1, \mu_1^* \leq \frac{n}{3} \right \}, \\
    W_2 & \eqdef \left \{ \frac{n}{3} < \mu_1 \leq \frac{n}{2} \text{ and } \mu_1^* \leq \frac{n}{2} \right \} \cup \left \{ \frac{n}{3} < \mu_1^* \leq \frac{n}{2} \text{ and } \mu_1 \leq \frac{n}{2} \right \}, \\
    W_3 & \eqdef \left \{ \frac{n}{2} < \mu_1 \leq 0.7n \right \} \cup \left \{ \frac{n}{2} < \mu_1^* \leq 0.7n \right \}.
\end{align*}
Together, these subsets partition $\{\mu \vdash n : \mu_1, \mu_1^* \leq 0.7n\}$. To complete the bound for $\Zone_R^I$ we will show that for $(i, j) \in [3] \times [3]$, the following sums are small:
\[
    \sum_{(\mu, \nu) \in W_i \times W_j} (f_\mu f_\nu)^2 \max_{\lambda \in \LR_{\mu, \nu}} |\Eig^\lambda_{\mu, \nu}|^{2t_N}. 
\]

To present the subsequent analysis in more uniform way, we define the following quantity for every choice $(i, j) \in [3] \times [3]$:
\begin{Definition}
    For $(i, j) \in [3] \times [3]$, we define the two quantities
    \begin{align*}
        \Lambda_{i, j} & \eqdef \frac{a^2}{4} \max_{\mu \in W_i} |D_\mu| + \frac{b^2}{4} \max_{\nu \in W_j} |D_\nu| + \frac{ab}{2} \max_{(\mu, \nu) \in W_i \times W_j} \frac{\langle \mu, \nu \rangle}{n^2}, \\
        \mathcal{E}_{i, j}(t) & \eqdef \max_{\mu \in W_i} (f_\mu)^2 \cdot \max_{\nu \in W_j} (f_\nu)^2 \cdot \Lambda_{ij}^{2t}.
    \end{align*}
    for all $t \geq 0$. 
\end{Definition} 

\begin{Lemma}\label{lem:bound-over-Wi-Wj}
    For $(i, j) \in [3] \times [3]$, we have 
    \[
        \sum_{(\mu, \nu) \in W_i \times W_j} (f_\mu f_\nu)^2 \max_{\lambda \in \LR_{\mu, \nu}} |\Eig^\lambda_{\mu, \nu}|^{2t} \leq \mathcal{E}_{i, j}(t) \cdot \exp(O(\sqrt{n})).
    \]
\end{Lemma}

\begin{proof}
    Let $\mu \in W_i$ and $\nu \in W_j$. From Corollary~\ref{cor:Eig-bound-in-terms-of-RT} and the fact that the $W_1, W_2$, and $W_3$ are closed under conjugation, we get the bound 
    \[
        |\Eig^\lambda_{\mu, \nu}| \leq \frac{a^2}{4} \max_{\mu \in W_i} |D_\mu| + \frac{b^2}{4} \max_{\nu \in W_j} |D_\nu| + \frac{ab}{2} \max_{\mu \in W_i, \nu \in W_j} \left \{ \frac{\langle \mu, \nu \rangle}{n^2} \right \} = \Lambda_{i,j}.
    \]
    Thus, we have 
    \begin{align*}
        \sum_{(\mu, \nu) \in W_i \times W_j} (f_\mu f_\nu)^2 \max_{\lambda \in \LR_{\mu, \nu}} |\Eig^\lambda_{\mu, \nu}|^{2t} & \leq \Lambda_{ij}^{2t} \sum_{(\mu, \nu) \in W_i \times W_j} (f_\mu f_\nu)^2 \\
        & \leq \max_{\mu \in W_i} (f_\mu)^2 \cdot \max_{\nu \in W_j} (f_\nu)^2 \cdot \Lambda_{ij}^{2t} \cdot p(n)^2 
    \end{align*}
    Since $p(n) = \exp(O(\sqrt{n}))$ from Theorem~\ref{thm:partition-function-bound}, we get the desired bound. 
\end{proof}

In the following Corollary, we rewrite Lemma~\ref{lem:diaconis-inequalities} in a form that appears more applicable to upper bound $\Lambda_{i,j}$. 

\begin{Corollary}\label{cor:restated-diaconis-inequalities}
    For sufficiently large $n$, we have the inequalities
    \begin{align*}
        \max_{\mu \in W_1} |D_\mu| & \leq \frac{1}{3}, \quad \quad \quad \quad \quad \quad \quad \max_{\mu \in W_1} (f_\mu)^2 \leq \exp (n \log n + O(n)) \\ 
        \max_{\mu \in W_2} |D_\mu| & \leq \frac{1}{2} \quad \quad \quad \quad \quad \quad \quad \max_{\mu \in W_2} (f_\mu)^2 \leq  \exp \left ( \frac{2}{3} n \log n + O(n) \right ) \\
        \max_{\mu \in W_3} |D_\mu| & \leq 0.58 + O \left ( \frac{1}{n} \right ), \quad \max_{\mu \in W_3} (f_\mu)^2 \leq \exp \left ( \frac{1}{2} n \log n + O(n) \right )
    \end{align*} 
\end{Corollary}
In the next Lemma, we bound the value of $\langle \mu, \nu \rangle$ depending on the sets $W_1, W_2$ and $W_3$ containing $\mu$ and $\nu$. 
\begin{Lemma}\label{lem:bounds-on-inner-product}
    \phantom{}
    \begin{enumerate}
        \item Let $\mu \in W_1$ and $\nu \vdash n$ be any partition. Then $\langle \mu, \nu \rangle \leq \frac{n^2}{3}$. 
        \item Let $\mu \in W_2$ and $\nu \vdash n$ be any partition. Then $\langle \mu, \nu \rangle \leq \frac{n^2}{2}$. 
        \item Let $\mu, \nu \in W_3$, then $\langle \mu, \nu \rangle \leq 0.58 n^2$. 
    \end{enumerate}
\end{Lemma}

\begin{proof}
    We can use Lemma~\ref{lem:oliver}(1) to conclude that for any $\mu, \nu \vdash n$, we have the inequality 
    \[
        \langle \mu, \nu \rangle \leq \langle \mu, (n) \rangle = n \mu_1.  
    \]
    Parts (1) and (2) immediately follow from this fact. In part (1), we have $\mu_1 \leq \frac{n}{3}$. In part (2), we have $\mu_1 \leq \frac{n}{2}$. For Part (3), we can apply Lemma~\ref{lem:oliver}(1) to get  
    \[
        \langle \mu, \nu \rangle \leq (0.7n, 0.3n), (0.7n, 0.3n) \rangle = 0.58n^2. 
    \]
    This suffices for the proof of the lemma. 
\end{proof}

The following result follows immediately from Corollary~\ref{cor:restated-diaconis-inequalities} and Lemma~\ref{lem:bounds-on-inner-product}. 

\begin{Corollary}\label{cor:bounds-on-Lambda-ij}
    \phantom{}
    \begin{enumerate}
        \item We have the inequalities
        \begin{align*}
            \Lambda_{1, 1} & \leq \frac{1}{3}, \quad \Lambda_{2, 2} \leq \frac{1}{2}, \quad \Lambda_{3, 3} \leq 0.58 + O \left ( \frac{1}{n} \right )
        \end{align*}
        \item We have the inequalities 
        \begin{align*}
            \max \{ \Lambda_{1, 2}, \Lambda_{2, 1} \} & \leq \frac{a^2}{8} + \frac{b^2}{12} + \frac{ab}{6}, \\
            \max \{ \Lambda_{1, 3}, \Lambda_{3, 1} \} & \leq \frac{0.58 a^2}{4} + \frac{b^2}{12} + \frac{ab}{6} \\ 
            \max \{ \Lambda_{2, 3}, \Lambda_{3, 2} \} & \leq \frac{0.58 a^2}{4} + \frac{b^2}{8} + \frac{ab}{4}.
        \end{align*}
    \end{enumerate}
\end{Corollary}

In Lemma~\ref{lem:the-parts-of-Zone-Red-I-are-small}, we show that for every choice of $i, j \in [3]$ there is some constant $K_{i, j}$ such that $\mathcal{E}_{i, j}(t_N)$ is at most on the order of $\exp (K_{ij} n \log n)$ for $K_{ij} < 0$. In the following definition, we describe exactly what these constants $K_{i, j}$ are. 

\begin{Definition}
    For $(i, j) \in [3] \times [3]$, we define the following constants: 
    \begin{align*}
        K_{11} & \eqdef 2 + \frac{2}{b} \log \left ( \frac{1}{3} \right ), \\ 
        K_{12} = K_{21} & \eqdef \frac{5}{3} + \frac{2}{b} \log \left ( \frac{a^2}{8} + \frac{b^2}{12} + \frac{ab}{6} \right ), \\ 
        K_{13} = K_{31} & \eqdef \frac{3}{2} + \frac{2}{b} \log \left ( \frac{0.58a^2}{4} + \frac{b^2}{12} + \frac{ab}{6} \right ), \\
        K_{22} & \eqdef \frac{4}{3} + \frac{2}{b} \log \left ( \frac{1}{2} \right ), \\
        K_{23} = K_{32} & \eqdef \frac{7}{6} + \frac{2}{b} \log \left ( \frac{0.58a^2}{4} + \frac{b^2}{8} + \frac{ab}{4} \right ), \\ 
        K_{33} & \eqdef 1 + \frac{2}{b} \log 0.58.
    \end{align*}
\end{Definition}

\begin{Lemma}\label{lem:all-negative}
    For $i, j \in [3]$ and $b \in (0, 1]$, $K_{ij} < 0$. 
\end{Lemma}

The proof of Lemma~\ref{lem:all-negative} is a simple exercise in calculus. Using all of the bounds we have collected, we can bound $\mathcal{E}_{i, j} (t_N)$ for all $i, j \in [3]$.

\begin{Lemma}\label{lem:the-parts-of-Zone-Red-I-are-small}
    For any choice of $i, j \in [3]$, we have the bound 
    \[
        \mathcal{E}_{i, j}(t_N) \leq \exp \left ( K_{ij} n \log n + O(n) \right ).
    \]
    In particular, there is a universal constant $C > 0$ such that $\mathcal{E}_{i, j}(t_N) \leq \exp (-C n \log n)$ for sufficiently large $n$ and all $i, j \in [3]$.
\end{Lemma}

\begin{proof}
    Recall that $t_N = \frac{1}{2b} N \log N$. Consider the expression 
    \[
        \mathcal{E}_{i, j}(t_N) = \max_{\mu \in W_i} (f_\mu)^2 \cdot \max_{\nu \in W_j} (f_\nu)^2 \cdot \Lambda_{i, j}^{2t_N}.
    \]
    Let $A_1 \eqdef 1$, $A_2 \eqdef 2/3$, and $A_3 \eqdef 1/2$. For all $i, j \in [3]$, let $B_{i,j}$ be the constant term of the upper bound on $\Lambda_{i, j}$ in Corollary~\ref{cor:bounds-on-Lambda-ij}. From Corollary~\ref{cor:bounds-on-Lambda-ij} and the inequalities on the right hand side of Corollary~\ref{cor:restated-diaconis-inequalities}, we have the following bounds: 
    \begin{align*}
        \max_{\mu \in W_i} (f_\mu)^2 & \leq \exp \left ( A_i n \log n + O(n) \right ), \\
        \max_{\nu \in W_j} (f_\nu)^2 & \leq \exp (A_j n \log n + O(n)), \\
        \Lambda_{i, j} & \leq B_{i, j} + O \left ( \frac{1}{n} \right ). 
    \end{align*}
    Then, we can bound $\mathcal{E}_{i, j}(t_N)$ by 
    \begin{align*}
        \mathcal{E}_{i, j}(t_N) & \leq \exp ((A_i + A_j) n \log n + 2 \log B_{i, j} t_N + O(n))\\
        & \leq  \exp \left \{ \left (A_i + A_j + \frac{2}{b} \log B_{i, j} \right ) n \log n + O(n) \right \}. 
    \end{align*}
    It is routine to check that for all $i, j \in [3]$ we have 
    \[
        K_{i, j} = A_i + A_j + \frac{2}{b} \log B_{i, j}. 
    \]
    This gives the first inequality in the statement of the Lemma. The second statement follows immediately from Lemma~\ref{lem:all-negative}. 
\end{proof}

Lemma~\ref{lem:the-parts-of-Zone-Red-I-are-small} and Lemma~\ref{lem:bound-on-error-Zone-Red-I} immediately give the following corollary. 

\begin{Corollary}\label{cor:red-zone-i-finished-bound}
    There exists $C > 0$ such that for sufficiently large $N$, we have 
    \[
        \sum_{\lambda \in \Zone_R^I} \Omega_\lambda (t_N     ) \leq \exp (-C n \log n)
    \] 
\end{Corollary}

\subsubsection{The Red Zone II}

In this section, we consider the sum of $\Omega_\lambda (t)$ over $\lambda \in \Zone_R^{II}$. This sub-zone can be split up into two conjugate subsets:
\begin{align*}
    \Zone_{R, +}^{II} & \eqdef \{0.7n \leq \lambda_1 \leq n, \lambda_1^* \leq 0.5n \} \\
    \Zone_{R,-}^{II} & \eqdef \{0.7n \leq \lambda_1^* \leq n, \lambda_1 \leq 0.5n\}. 
\end{align*}

Let us consider the subset $\Zone_{R, +}^{II}$ for now. In Definition~\ref{defn:auxiliary-functions-Zone-Red} we define several functions which will help us partition the zone $\Zone_{R,+}^{II}$ into smaller regions which we can bound more easily. 

\begin{Definition}\label{defn:auxiliary-functions-Zone-Red}
    \phantom{h} 
    \begin{enumerate}
    \item Let $\phi_1 : \RR \to \RR$ be the function defined by 
    \[
        \phi_1(x) \eqdef \frac{(ab - b^2 - abx)^2}{16ab} + \frac{a^2 - ab}{2} x^2 + \frac{2ab-a^2}{2} x + \frac{a^2 - ab}{4}. 
    \]
    \item Let $\phi_2 : (0, \infty) \to \RR$ be the function defined by 
    \[
        \phi_2 (x) \eqdef 2 + \frac{2}{b} \log \phi_1(x). 
    \]
    \item Let $\phi_3 : (\phi_2(0.5), \infty) \to (0.5, \infty)$ be the inverse function to $\phi_2 : (0.5, \infty) \to (\phi_2(0.5), \infty)$. 
    \item Let $\phi_4 : (0.5, \infty) \to (0.5, \infty)$ be the function defined by 
    \[
        \phi_4(x) \eqdef \frac{\phi_3(x) + x}{2}.
    \]
    \end{enumerate}
\end{Definition}

\begin{Remark}
    Some properties and the fact that these functions are well-defined are discussed in the appendix. The orbit of $0.7$ under the map $\phi_4$ will determine how we partition $\Zone_{R, +}^{II}$ into smaller parts.
\end{Remark}

\begin{Proposition}\label{prop:tanuki}
    Let $\lambda \vdash 2n$ satisfy $0.7n \leq \lambda \leq n$ and $\lambda_1^* \leq 0.5n$, i.e. let $\lambda \in \Zone_{R, +}^{II}$. Let $(\mu, \nu) \in \LR^\lambda$. Then 
    \[
        |\Eig^\lambda_{\mu, \nu}| \leq \max \left \{ \frac{1}{2}, \phi_1 \left ( \frac{\lambda_1}{n} \right ) \right \} + O \left ( \frac{1}{n} \right ). 
    \]
\end{Proposition}

\begin{proof}
    For any $\lambda$ in $\Zone_{R, +}^{II}$ and any $(\mu, \nu) \in \LR^\lambda$, Lemma~\ref{lem:oliver} and Lemma~\ref{lem:generic-diagonal-sum-bound} gives us the diagonal index bounds
    \begin{align*}
        \Diag (\lambda) & \leq n \lambda_1 - \frac{\lambda_1 \lambda_1^*}{2} - \frac{(\lambda_1^*)^2}{2} + O(n) \\
        \Diag(\mu) & \leq \Diag(\lambda_1, n - \lambda_1) = \frac{n^2}{2} - n \lambda_1 + \lambda_1^2 \\ 
        \Diag(\nu) & \geq - \frac{n \lambda_1^*}{2} + O(n). 
    \end{align*}

    To get diagonal index bounds in the other direction, we construct a Young diagram which is dominated by $\lambda$. Since $0.7n \leq \lambda_1$ and $\lambda_1^* \leq 0.5n$, we can always move down boxes from the Young diagram of $\lambda$ to get the Young diagram $\mathbf{T}$ with first row of length $0.7n$, first column of length $0.5n$, second column of length $0.5n$, and third column of length $0.3n$ (up to constant error). From Lemma~\ref{lem:oliver}, we get the bounds 
    \begin{align*}
        \Diag (\lambda) & \geq \Diag(\mathbf{T}) = -0.05n^2 + O(n), \\ 
        \Diag(\mu) & \geq \Diag ( (0.5n, 0.5n)^*) = -0.25n^2 + O(n), \\ 
        \Diag(\nu) & \leq \Diag((n)) = 0.5n^2 + O(n). 
    \end{align*}
    This gives us a lower bound 
    \[
        \Eig^\lambda_{\mu, \nu} \geq - (0.125a^2 + 0.15ab - 0.25b^2) + O\left (\frac{1}{n} \right ) \geq -\frac{1}{2}
    \]
for sufficiently large $n$. The first set of diagonal index inequalities gives the eigenvalue bound 
\begin{equation}\label{eqn:EEE}
    \Eig^\lambda_{\mu, \nu} \leq S_1 \left ( \frac{\lambda_1}{n}, \frac{\lambda_1^*}{n} \right ) + S_2 \left ( \frac{\lambda_1}{n} \right ) + O \left ( \frac{1}{n} \right )
\end{equation}
where we define 
\begin{align*}
    S_1(x, y) & \eqdef - \frac{ab}{4} y^2 + \left ( \frac{ab - b^2}{4} - \frac{ab}{4}x \right ) y, \\
    S_2(x) & \eqdef \frac{a^2 - ab}{2} x^2 + \frac{2ab - a^2}{2} x + \frac{a^2 - ab}{4}.
\end{align*}

Viewing $S_1(x, y)$ as a quadratic equation in $y$, it is maximized when $y = \frac{ab - b^2 - abx}{2ab}$. Thus, for any $x \in \RR$, we have the bound 
\[
    S_1(x, y) \leq S_1 \left (x, \frac{ab - b^2 - abx}{2ab} \right ). 
\]
Applying this bound to Equation~\ref{eqn:EEE} gives
\[
    \Eig^\lambda_{\mu, \nu} \leq \phi_1 \left ( \frac{\lambda}{n} \right ) + O \left ( \frac{1}{n} \right ). 
\]
This suffices for the proof of the proposition. 
\end{proof}

To prove that the sum over the region $0.7 \leq \lambda_1 \leq n$, $\lambda_1^* \leq 0.5n$ is small, we will split this region into sub-regions $c_1 n \leq \lambda_1 \leq c_2 n$, $\lambda_1^* \leq 0.5n$. The following lemma allows us to use the function $\phi_4$ to determine these sub-regions.

\begin{Lemma}\label{lem:Red-II-plus-each-partition-is-small}
    Let $x$ be a real number satisfying $0.7 \leq x \leq \phi_4(x)$.
    \begin{enumerate}
        \item If $\phi_4(x) \leq 1$, then there is some constant $C > 0$ such that 
        \[
            \sum_{\lambda \in R_x} \Omega_\lambda (t_N) \leq \exp (-C n \log n)
        \]
        where $R_x$ consists of the $\lambda \in \Zone_{R,+}^{II}$ satisfying $xn \leq \lambda_1 \leq \phi_4(x)n$.

        \item If $\phi_4(x) > 1$, then there is some constant $C > 0$ such that 
        \[
            \sum_{\lambda \in R^x} \Omega_\lambda (t_N) \leq \exp ( -C n \log n)
        \]
        where $R^x$ consists of the $\lambda \in \Zone_{R, +}^{II}$ satisfying $xn \leq \lambda_1 \leq n$. 
    \end{enumerate}
\end{Lemma}

\begin{proof}
    Suppose that we are in the situation of part (1). Let $\lambda \in R_x$ be any partition. From Proposition~\ref{prop:tanuki} and the fact that $\phi_1$ is increasing on $[0.7, \infty)$ (Lemma~\ref{lem:APPENDIX-prop-D}), we have the bound 
    \begin{align*}
        \max_{(\mu, \nu) \in \LR^\lambda} \left |\Eig^\lambda_{\mu, \nu} \right | & \leq \max \left \{ \frac{1}{2}, \phi_1 (\phi_4(x)) \right \} + O \left ( \frac{1}{n} \right ) \\
        & \leq \max \left \{ \frac{1}{2}, \phi_1(\phi_3(x))e^{-\varepsilon} \right \} + O \left ( \frac{1}{n} \right ) \\
        \implies \max_{\substack{\lambda \in R_x \\ (\mu, \nu) \in \LR^\lambda}} \left |\Eig^\lambda_{\mu, \nu} \right | & \leq \max \left \{ \frac{1}{2}, \phi_1(\phi_3(x)) e^{-\varepsilon}\right \} + O \left ( \frac{1}{n} \right )
    \end{align*}
    for some $\varepsilon > 0$ depending only on $x$, $a$, and $b$. We can pick such a $\varepsilon$ since $\phi_3(x) > \phi_4(x)$ in this domain and $\phi_1(\bullet)$ is strictly increasing in this interval. From Lemma~\ref{lem:f-lambda-bound}, we have the bound 
    \[
        \max_{\lambda \in R_x} (f_\lambda)^2 \leq \exp ((2-x) n \log n + O(n)). 
    \]
    From Lemma~\ref{lem:essential-tool}, we have that $\sum_{\lambda \in R_x} \Omega_\lambda (t_N)$ is bounded above by 
    \[
        \exp \left ((2-x)n  \log n + t_N \max \left \{ \log \frac{1}{2}, \log (\phi_1(\phi_3(x))) - \varepsilon \right \} + O(n) \right ). 
    \]
    We define the constants
    \begin{align*}
        U_1 & \eqdef 2 - x + \frac{2}{b} \log \frac{1}{2}, \\
        U_2 & \eqdef  2- x + \frac{2}{b}\log (\phi_1(\phi_3(x))) - \frac{2\varepsilon}{b}.
    \end{align*}
    Then, we have the following upper bound: 
    \begin{equation}\label{eqn:beaver}
        \sum_{\lambda \in R_x} \Omega_\lambda (t_N) \leq \max \{ \exp(U_1 n \log n + O(n)), \exp(U_2 n \log n + O(n)) \}. 
    \end{equation}
    We can check that
    \begin{align*}
        U_1 & = 2 - x + \frac{2}{b} \log \frac{1}{2} \leq 1.3 + \frac{2}{b} \log \frac{1}{2}  < 0, \\
        U_2 & = 2 - x + \frac{2}{b} \log (\phi_1(\phi_3(x))) - \frac{2\varepsilon}{b} = \phi_2(\phi_3(x)) - x - \frac{2\varepsilon}{b} = -\frac{2\varepsilon}{b} < 0.
    \end{align*} 
    Applying this to Equation~\ref{eqn:beaver}, we have 
    \[
        \sum_{\lambda \in R_x} \Omega_\lambda (t_N) \leq \exp (-C n \log n)
    \]
    for some $C > 0$ depending only on $x$, $a$, and $b$. This proves part (1). The proof to part (2) is exactly the same up to replacing instances of $\phi_3(x)$ and $\phi_4(x)$ with $\min \{\phi_3(x), 1\}$ and $\min \{\phi_4(x), 1\}$. 
\end{proof}

\begin{Corollary}\label{cor:full-of-zone-red-II}
    For sufficiently large $N$ the inequality 
    \begin{equation}\label{eqn:zone-red-II}
        \sum_{\lambda \in \Zone^{II}_R} \Omega_\lambda(t_N) \leq \exp (-Cn \log n)
    \end{equation}
    for some universal constant $C > 0$ depending on $a$ and $b$. 
\end{Corollary}

\begin{proof}
    We first prove that Equation~\ref{eqn:zone-red-II} holds when the sum is over $\Zone_{II, +}^R$. The result for the whole region then follows from Lemma~\ref{lem:OMEGA-is-INVOLUTIVE}. Proposition~\ref{prop:appendix-dynamical-subzoning} gives us a sequence 
    \[
        0.7 = x_0 < x_1 < \ldots < x_{M-1} \leq 1 < x_M
    \]
    such that $x_{i+1} = \phi_4(x_i)$ for all $i$. Then the result follows by applying Lemma~\ref{lem:Red-II-plus-each-partition-is-small} to every adjacent pair of terms in the sequence. 
\end{proof}

\subsubsection{The Red Zone III}

In this section, we prove that the sum of $\Omega_\lambda(t_N)$ over $\Zone_R^{III}$ is small. Recall that this zone contains all $\lambda \vdash 2n$ satisfying 
\[
    0.7n \leq \lambda_1 \leq n, 0.5n \leq \lambda_1^* \leq 0.7n \quad \text{ or } \quad 0.7n \leq \lambda_1^* \leq n, 0.5n \leq \lambda_1 \leq 0.7n. 
\]
Proposition~\ref{prop:BOUND-on-EIG-between-0.5-and-1} gives the following bound on the eigenvalue in terms of $\mathbf{Q}_R$. Recall that $\mathbf{Q}_R$ was defined in Definition~\ref{defn:q-function}. 

\begin{Lemma}\label{lem:france}
    For $\lambda \in \Zone_R^{III}$, we have 
    \[
        \max_{(\mu, \nu) \in \LR^\lambda} |\Eig^\lambda_{\mu, \nu}| \leq \max \left \{ \max_{\substack{x \in [0.7, 1] \\ y \in [0.5,0.7]}} \mathbf{Q}_R(x, y), \max_{\substack{x \in [0.7,1] \\ y \in [0.5, 0.7]}} \mathbf{Q}_R(y, x) \right \} + O \left ( \frac{1}{n} \right )
    \]
    where the $O(1/n)$ error is universal. 
\end{Lemma}

From Lemma~\ref{lem:france} and Lemma~\ref{lem:APPENDIX-BOUND-ON-R-FUNCTION-RED-III}, we get the following Corollary. 
\begin{Corollary}\label{cor:RedZone-III-max-eigenvalue-bound}
    \phantom{}
    \[
        \max_{\substack{\lambda \in \Zone_R^{III} \\ (\mu, \nu) \in \LR^\lambda}} |\Eig^\lambda_{\mu, \nu}| \leq \frac{a^2}{4} + \frac{3ab}{16} - \frac{b^2}{8} + O \left ( \frac{1}{n} \right ). 
    \]
\end{Corollary}

This allows us to prove that the sum of our error term over $\Zone_R^{III}$ is small. 

\begin{Corollary}\label{cor:red-zone-III-full-bound}
    There is some $C > 0$ such that for sufficiently large $N$ we have
    \[
        \sum_{\lambda \in \Zone_R^{III}} \Omega_\lambda(t_N) \leq \exp ( - C n \log n).
    \] 
\end{Corollary}

\begin{proof}
    Lemma~\ref{lem:f-lambda-bound} gives us the bound 
    \[
        \max_{\lambda \in \Zone_R^{III}} (f_\lambda)^2 \leq \exp (0.8n \log n + O(n)). 
    \]
    From Corollary~\ref{cor:RedZone-III-max-eigenvalue-bound} and Lemma~\ref{lem:essential-tool} we have the bound 
    \[
        \sum_{\lambda \in \Zone_R^{III}} \Omega_\lambda (t_N) \leq \exp (C_3 n \log n + O(n))
    \]
    where 
    \[
        C_3 \eqdef 0.8 + \frac{2}{b} \log \left ( \frac{a^2}{4} + \frac{3ab}{16} - \frac{b^2}{8} \right ) < 0. 
    \]
    This suffices for the proof. 
\end{proof}
\subsubsection{The Red Zone IV}

To complete the red zone, we will prove that the sum of $\Omega_\lambda (t_N)$ over $\Zone_R^{IV}$ is small. Recall that this zone consists of the partitions $\lambda \vdash 2n$ satisfying 
\[
    0.7n \leq \lambda_1, \lambda_1^* \leq n. 
\]
From Proposition~\ref{prop:BOUND-on-EIG-between-0.5-and-1}, we have the following bound on the eigenvalues in this regime. 

\begin{Lemma}
    For $\lambda \in \Zone_R^{IV}$, we have 
    \[
        \max_{(\mu, \nu) \in \LR^\lambda} |\Eig^\lambda_{\mu, \nu}| \leq \max_{x, y \in [0.7, 1]} \mathbf{Q}_R(x, y) + O \left ( \frac{1}{n} \right ). 
    \]
\end{Lemma}

\begin{Corollary}\label{cor:full-red-zone-4-bound}
    There is some $C > 0$ such that for sufficiently large $N$ we have 
    \[
        \sum_{\lambda \in \Zone_R^{IV}} \Omega_\lambda (t_N) \leq \exp (-C n \log n). 
    \]
\end{Corollary}

\begin{proof}
    From Lemma~\ref{lem:f-lambda-bound}, we have 
    \[
        \max_{\lambda \in \Zone_R^{IV}} (f_\lambda)^2 \leq \exp (0.6 n \log n + O(n)). 
    \]
    From Lemma~\ref{lem:generic-diagonal-sum-bound} and Lemma~\ref{lem:APPENDIX-ZONE-RED-IV-BOUND-ON-Q}, we have the bound 
    \[
        \sum_{\lambda \in \Zone_R^{IV}} \Omega_\lambda(t_N) \leq \exp (C_4 n \log n + O(n))
    \]
    where
    \[
        C_4 \eqdef 0.6 + \frac{2}{b} \log (0.25a^2 + 0.0975ab-0.145b^2) < 0.
    \]
    This completes the proof of the corollary. 
\end{proof}

\subsection{The (Modified) Blue Zone}\label{sec:blue-zone}

In this section, we will prove that the sum of $\Omega_\lambda(t)$ over a \emph{modified} version of $\Zone_B$ is small. We will want to sum over a slightly bigger zone to make our analysis in $\Zone_C$ easier. Recall that $a^* = 2-a^{-1}$. 

\begin{Definition}
    For $\varepsilon > 0$, we define the \textbf{$\varepsilon$-modified Blue Zone} to be the set 
    \[
        \Zone_{B, \varepsilon} \eqdef \{ n \leq \lambda_1 \leq (a^* + \varepsilon) n, \lambda_1^* \leq n\} \cup \{ n \leq \lambda_1^* \leq (a^* + \varepsilon)n, \lambda_1 \leq n \}. 
    \]
    We define the sub-region $\Zone_{B, \varepsilon}^{I}$ as the union of the sets 
    \begin{align*}
       \Zone_{B, \varepsilon}^{I, +} & \eqdef \{n \leq \lambda_1 \leq (a^* + \varepsilon) n, \quad 0.5n \leq \lambda_1^* \leq n\} \\
       \Zone_{B, \varepsilon}^{I, -} & \eqdef \{n \leq \lambda_1^* \leq (a^* + \varepsilon) n, \quad 0.5n \leq \lambda_1 \leq n\}. 
    \end{align*}
    We define the sub-region $\Zone_{B, \varepsilon}^{II}$ as the union of sets
    \begin{align*}
        \Zone_{B, \varepsilon}^{II, +} & \eqdef \{n \leq \lambda_1 \leq (a^* + \varepsilon) n, \quad \lambda_1^* \leq 0.5n\} \\
        \Zone_{B, \varepsilon}^{II, -} & \eqdef \{n \leq \lambda_1^* \leq (a^* + \varepsilon) n, \quad \lambda_1 \leq 0.5n\}.
    \end{align*}
\end{Definition}

See Figure~\ref{fig:subzone-zone-blue} for a picture of these sub-regions. The "$+$" regions are below the $\lambda_1 = \lambda_1^*$ line and the "$-$" regions are above the line. In the next section, we begin showing that for sufficiently small $\varepsilon$, the sum over the sub-regions of the modified blue zone is small. 

\begin{figure}
    \begin{tikzpicture}[scale = 1.2]
        \definecolor{lightblue}{rgb}{0.68, 0.85, 0.90}

        \draw[->, line width=0.25mm, dotted] (0,0) -- (3.8,0) node[right] {$\lambda_1$};
        \draw[->, line width=0.25mm, dotted] (0,0) -- (0,3.8) node[right] {$\lambda_1^*$};
        \fill[yellow, opacity=0.5] (2.8,0)--(3.1,0)--(3.1,0.9)--(2.8,1.2)--cycle;
        \fill[yellow, opacity=0.5] (0,2.8)--(0,3.1)--(0.9,3.1)--(1.2,2.8)--cycle;
        \fill[lightblue] (0, 2) -- (2, 2) -- (1.2, 2.8) -- (0, 2.8) -- cycle;

        \fill[lightblue] (2, 0) -- (2, 2) -- (2.8,1.2) -- (2.8,0) -- cycle;
        \fill[lightblue] (2, 0) -- (2, 2) -- (2.8,1.2) -- (2.8,0) -- cycle;

        \draw (3, 3) node[right] {$a^* \eqdef 2-a^{-1}$};

        \draw[line width=0.4mm] (2, 0) -- (2.8, 0);
        \draw[line width=0.4mm] (2,0) -- (2,2);
        \draw[line width=0.4mm] (0,2) -- (2,2);
        \draw[line width=0.4mm] (0,2)--(0,2.8);
        \draw[line width=0.4mm] (2,0) -- (3.1,0) -- (3.1,0.9) -- (0.9,3.1) -- (0,3.1) -- (0, 2);
        
        \draw[line width=0.4mm] (1,2)--(1,3);
        \draw[line width=0.4mm] (2,1)--(3,1);

        \draw (2.53, 0.3) node[above] {\Large $II$};
        \draw (0.25, 2.53) node[right] {\Large $II$};
        \draw (2.32,1.1) node[above] {$I$};
        \draw (1.1,2.32) node[right] {$I$};

        \draw (3.1, 0) node[below] {$(a^* + \varepsilon) n$};
        \draw (0, 3.1) node[left] {$(a^* + \varepsilon) n$};
        \draw (2, 0) node[below] {$n$};
        \draw (0,2) node[left] {$n$};
        \draw (2,1) node[left] {$0.5n$};
        \draw (1,2) node[below] {$0.5n$};
    \end{tikzpicture}
    \caption{This is $\Zone_{B, \varepsilon}$, a modified version of $\Zone_B$ where we extend the region $\varepsilon n$ to the right. We have split this into subzones $\Zone_{B, \varepsilon}^{I}$ and $\Zone_{B, \varepsilon}^{II}$.}\label{fig:subzone-zone-blue}
\end{figure}
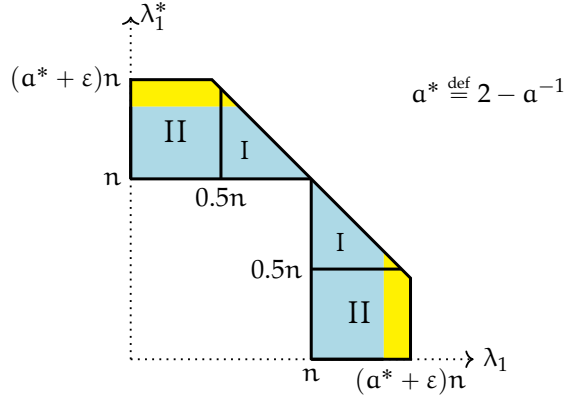

\subsubsection{Modified Blue Zone Sub-region I}

We prove that the sum of $\Omega_\lambda(t)$ over $\Zone_{B, \varepsilon}^{I, +}$ is small for sufficiently small $\varepsilon$. Recall that in this region we have 
\[
    n \leq \lambda_1 \leq (a^* + \varepsilon) n, \quad 0.5n \leq \lambda_1^* \leq n.
\]
We first define the auxiliary function $\mathbf{Q}_B$ that will be helpful for bounding our eigenvalues in this region. This function is similar to the function $\mathbf{Q}_R$ which was used in our analysis of the red zone. 

\begin{Definition}
    Define the function $\mathbf{Q}_B : \RR^2 \to \RR$ by 
    \[
       \mathbf{Q}_B(x, y) \eqdef \frac{a^2 - b^2}{4} + \frac{ab}{2} \left( x - \frac{xy}{2} - \frac{y^2}{2} \right) + \frac{ab - b^2}{2} (y^2 - y)
    \]
\end{Definition}

\begin{Proposition}\label{prop:two-sided-bounds-Zone-Blue-I}
    For $\lambda \in \Zone_{B, \varepsilon}^{I, +}$ and $(\mu, \nu) \in \LR^\lambda$, we have 
    \[
        - \frac{a^2 - b^2}{4} + O\left( \frac{1}{n} \right) \leq \Eig^\lambda_{\mu, \nu} \leq \mathbf{Q}_B \left( \frac{\lambda_1}{n}, \frac{\lambda_1^*}{n} \right) + O\left( \frac{1}{n} \right). 
    \]
\end{Proposition}

\begin{proof}
    Let $\lambda, \mu, \nu$ be as in the hypothesis. From Lemma~\ref{lem:oliver} and Lemma~\ref{lem:essential-tool}, we have diagonal sum bounds 
    \begin{align*}
        \Diag(\lambda) & \leq n \lambda_1 - \frac{\lambda_1 \lambda_1^*}{2} - \frac{(\lambda_1^*)^2}{2} + O(n) \\ 
        \Diag(\mu) & \leq \Diag((n)) \leq \frac{n^2}{2} \\
        \Diag(\nu) & \geq - \Diag(\lambda_1^*, n-\lambda_1^*) = - \left( \frac{n^2}{2} - n \lambda_1^* + (\lambda_1^*)^2 \right) + O(n). 
    \end{align*}
    Applying these bounds directly to $\Eig^\lambda_{\mu, \nu}$ gets us the upper bound. For the lower bound, we have from Lemma~\ref{lem:oliver}
    \begin{align*}
        \Diag(\lambda) & \geq \Diag((n, 1^n)) = O(n) \\
        \Diag(\mu) & \geq \Diag(1^n) \geq - \frac{n^2}{2} \\
        \Diag(\nu) & \leq \Diag((n)) = \frac{n^2}{2}. 
    \end{align*}
    Applying these bounds to $\Eig^\lambda_{\mu, \nu}$ gets us the lower bound.
\end{proof}

\begin{Proposition}\label{prop:Blue-Zone-I-Full-Bound}
    For sufficiently small $\varepsilon > 0$ there is some $C > 0$ such that for sufficiently large $N$
    \[
        \sum_{ \lambda \in \Zone_{B, \varepsilon}^{I}} \Omega_\lambda (t_N) \leq \exp (-C n \log n).
    \]
\end{Proposition}

\begin{proof}
    In this regime, Lemma~\ref{lem:f-lambda-bound} gives us 
    \[
        \max_{\lambda \in \Zone_{B, \varepsilon}^I} (f_\lambda)^2 \leq \exp (0.5 n \log n + O(n)) 
    \]
    Proposition~\ref{prop:two-sided-bounds-Zone-Blue-I} and Lemma~\ref{lem:APPENDIX-ZONE-BLUE-I-BOUNDS} gives the bound 
    \[
        \max_{\lambda \in \Zone_{B, \varepsilon}^{I, +}} (f_\lambda)^2 \max_{\substack{\lambda \in \Zone_{B, \varepsilon}^{I, +}}} |\Eig^\lambda_{\mu, \nu}|^{2t_N}\leq \exp \left\{ K_\text{blue} n \log n + O(n) \right\}
    \]
    where
    \begin{align*}
        K_\text{blue} & \eqdef 0.5 + \frac{2}{b} \log \left  ( \max \left\{ \frac{a^2 - b^2}{4}, 1 - \frac{b}{4} - \frac{7}{16} b^2 + \varepsilon_0 \right\} \right ) \\
        & = 0.5 + \frac{2}{b} \log \left( 1 - \frac{b}{4} - \frac{7}{16} b^2 + \varepsilon_0 \right)
    \end{align*}
    where we can make $\varepsilon_0$ as small as we like by shrinking $\varepsilon$. Take $\varepsilon$ sufficiently small so that $\varepsilon_0 < \frac{7}{16}b^2$. Then 
    \begin{align*}
        K_\text{blue} & < 0.5 + \frac{2}{b} \log \left( 1 - \frac{b}{4} \right) \leq 0.5 + \frac{2}{b} \log e^{-\frac{b}{4}} = 0. 
    \end{align*}
    Lemma~\ref{lem:essential-tool} implies that our desired sum over the sub-region $\Zone_{B, \varepsilon}^{I, +}$ is at most $\exp(-C n \log n)$. Lemma~\ref{lem:OMEGA-is-INVOLUTIVE} implies that the sum over $\Zone_{B, \varepsilon}^{I, -}$ is also at most $\exp(-Cn \log n)$ (possibly with different $C$). This suffices for the proof of the proposition. 
\end{proof}

\subsubsection{Modified Blue Zone Sub-region II}

In this section, we bound the sum of $\Omega_\lambda(t_N)$ over the remaining blue region $\Zone_{B, \varepsilon}^{II}$. We will focus on $\Zone_{B, \varepsilon}^{II, +}$ and the negative region will follow from Lemma~\ref{lem:OMEGA-is-INVOLUTIVE}. Our proof in this zone will mimic that of $\Zone_R^{II}$: we define a function whose dynamics will determine the way we partition of zone into smaller sub-zones.

\begin{Definition}
    Let $L_B \eqdef 1 + \frac{2}{b} \log \frac{a}{2}$. Define the function $\mathbf{T}_B : (L_B, \infty) \to \RR$ by the equation 
    \[
        \mathbf{T}_B \eqdef \sqrt{\frac{2e^{\frac{b}{2}(x-1)}-a}{ab}}.
    \]
\end{Definition}

\begin{Definition}
    We define $\mathscr{T}_B : (L_B, \infty) \to \RR$ by 
    \begin{align*}
        \mathscr{T}_B (x) & \eqdef \frac{\mathbf{T}_B(x) + x}{2} 
    \end{align*}
\end{Definition}

\begin{Definition}
    Let $\mathbf{P}_B : \RR^2 \to \RR$ be the function given by 
    \[
        \mathbf{P}_B (x, y) \eqdef \frac{a^2+3ab}{4} + \frac{ab}{2} (x^2 - 2(x+y) + xy) + \frac{ab-b^2}{4} y. 
    \]
\end{Definition}

\begin{Remark}
    For properties of these functions, see Section~\ref{sec:technical-Blue-Zone-II}. 
\end{Remark}

The next result Lemma~\ref{lem:australia} plays a similar role in the blue zone as Lemma~\ref{lem:Red-II-plus-each-partition-is-small} plays in the red zone. It proves that we can use the function $\cT_B$ to create sub-regions of $\Zone_{B, \varepsilon}^{II, +}$ of the form $c_1 n \leq \lambda_1 \leq c_2 n$ and $\lambda_1^* \leq 0.5n$ whose contribution to the sum can be proved to be small. 

\begin{Lemma}\label{lem:australia}
    Let $\alpha \in [0, a^*-1]$, and suppose that $0 \leq \alpha < \mathscr{T}_B(\alpha)$. Let $\varepsilon > 0$ be a sufficiently small number which works in Lemma~\ref{lem:APPENDIX-zoneB-subzone-2-sequence-exists} and Lemma~\ref{lem:APPENDIX-boldP-B-max-bound}. There exists a constant $C > 0$ depending only on $\alpha$, $a$, $b$, such that for sufficiently large $N$ we have 
    \[
        \sum_{\lambda \in Z_\alpha} \Omega_\lambda (t_N) \leq \exp (-C n \log n)
    \]
    where 
    \[
        Z_\alpha \eqdef \{\lambda \vdash 2n : (1+\alpha)n \leq \lambda_1 \leq (\min\{1+\cT_B(\alpha), a^*+\varepsilon \}) n, 1 \leq \lambda_1^* \leq 0.5n \}. 
    \]
\end{Lemma}

\begin{proof}
    For any $\lambda \in Z_\alpha$ and $(\mu, \nu) \in \LR^\lambda$, we provide a uniform upper bound on $|\Eig^\lambda_{\mu, \nu}|$. From Lemma~\ref{lem:f-lambda-bound}, we have the dimension bound 
    \[
        \max_{\lambda \in Z_\alpha} (f_\lambda)^2 \leq \exp ((1-\alpha)n \log n + O(n)). 
    \]
    From Lemma~\ref{lem:oliver}, we have the diagonal sum bounds 
    \begin{align*}
        \Diag(\lambda) & \leq \Diag(\lambda_1, 2n-\lambda_1-\lambda_1^*+2, 1^{\lambda_1^*-2}) \\
        & = \lambda_1^2 + 2n^2 - 2n(\lambda_1 + \lambda_1^*) + \lambda_1 \lambda_1^* + O(n) \\
        \Diag(\mu) & \leq \frac{n^2}{2} + O(n) \\
        \Diag(\nu) & \geq - \frac{n\lambda_1^*}{2}. 
    \end{align*}
    This gives an upper bound of 
    \[
        \Eig^\lambda_{\mu, \nu} \leq \mathbf{P}_B \left( \frac{\lambda_1}{n}, \frac{\lambda_1^*}{n} \right) + O\left( \frac{1}{n} \right). 
    \]
    From Lemma~\ref{lem:oliver}, we have the diagonal sum bounds in the opposite direction 
    \begin{align*}
        \Diag(\lambda) & \geq \Diag(n, 1^n) = O(n) \\
        \Diag(\mu) & \geq -\Diag(0.5n, 0.5n) = - \frac{n^2}{4} + O(n) \\
        \Diag(\nu) & \leq \frac{n^2}{2}. 
    \end{align*}
    For sufficiently large $N$, this gives a lower bound of 
    \[
        \Eig^\lambda_{\mu, \nu} \geq - \frac{a^2 +ab -2b^2}{8} + O\left( \frac{1}{n} \right) > - \frac{1}{2}.
    \]
    Let $\beta = \min \{\mathscr{T}_B(\alpha), \alpha^* + \varepsilon - 1\}$. From Lemma~\ref{lem:APPENDIX-boldP-B-max-bound}, we have
    \begin{align*}
        \max_{\substack{\lambda \in Z_\alpha \\ (\mu, \nu) \in \LR^\lambda}} |\Eig^\lambda_{\mu, \nu}| & \leq \max \left\{ \max_{\substack{x \in [1+\alpha, 1+\beta] \\ y \in [0, 0.5]}}\mathbf{P}_B (x, y) , \frac{1}{2} \right\} + O\left (\frac{1}{n} \right ) \\
        & \leq \max\left\{ \frac{ab}{2} \beta^2 + \frac{a}{2}, \frac{1}{2}  \right\} + O\left( \frac{1}{n} \right) \\ 
        & \leq \max\left\{ \frac{ab}{2} \mathscr{T}_B(\alpha)^2 + \frac{a}{2}, \frac{1}{2}  \right\} + O\left( \frac{1}{n} \right)
    \end{align*}
    where we used the lemma on the penultimate line. From Lemma~\ref{lem:essential-tool}, we get 
    \[
        \sum_{\lambda \in Z_\alpha} \Omega_\lambda(t_N) \leq \exp \left ( J_B n \log n + O(n) \right ) 
    \]
    where
    \[
        J_B \eqdef \max \left \{ 1 - \alpha + \frac{2}{b} \log \left( \frac{ab}{2} \mathscr{T}_B(\alpha)^2 + \frac{a}{2} \right), 1 - \alpha + \frac{2}{b} \log \frac{1}{2} \right \}.
    \]
    The second term in the maximum is always negative. We can also check that  
    \begin{align*}
        1 - \alpha + \frac{2}{b} \log \left( \frac{ab}{2} \mathscr{T}_B(\alpha)^2 + \frac{a}{2} \right) & < 1 - \alpha + \frac{2}{b} \log \left( \frac{ab}{2} \mathbf{T}_B(\alpha)^2 + \frac{a}{2} \right) = 0.
    \end{align*}
    Thus $J_B < 0$ which completes the proof of the lemma. 
\end{proof}

\begin{Proposition}\label{prop:Blue-Zone-II-Full-Bound}
   For sufficiently small $\varepsilon > 0$, there is a universal constant $C > 0$ such that for sufficiently large $N$ we have 
    \[
        \sum_{\lambda \in \Zone_{B, \varepsilon}^{II}} \Omega_\lambda(t_N) \leq \exp (-C n \log n). 
    \]
\end{Proposition}

\begin{proof}
    Lemma~\ref{lem:APPENDIX-zoneB-subzone-2-sequence-exists} gives us an increasing sequence which eventually gets bigger than $a^* - 1$. Applying Lemma~\ref{lem:australia} to every pair of adjacent terms in the sequence, we get the result over $\Zone_{B, \varepsilon}^{II, +}$. Lemma~\ref{lem:OMEGA-is-INVOLUTIVE} extends the bound to all of $\Zone_{B, \varepsilon}^{II}$. 
\end{proof}

\subsection{The (Modified) Yellow Zone}\label{sec:yellow-zone}

In this section, we will bound the sum of $\Omega_\lambda (t_N+cN)$ over a \emph{modified} version of $\Zone_Y$. The largest terms are in this zone, so our \emph{window} will finally play a role. The modified version of $\Zone_Y$ will be denoted by $\Zone_Y^{\varepsilon}$ and we think of it as retracting $\Zone_Y$ by $\varepsilon n$. Explicitly, these regions are
\begin{align*}
    \Zone_{Y, \varepsilon}^{+} & = \{ (a^* + \varepsilon)n \leq \lambda_1 \leq 2n\} \\
    \Zone_{Y, \varepsilon}^{-} & = \{ (a^* + \varepsilon)n \leq \lambda_1^* \leq 2n \}. 
\end{align*}
See Figure~\ref{fig:yellow-zone} for an illustration of these regions. 

\begin{figure}
    \begin{tikzpicture}[scale = 1.2]
        \definecolor{lightblue}{rgb}{0.68, 0.85, 0.90}

        \draw[->, dotted] (0,0) -- (4.4,0) node[right] {$\lambda_1$};
        \draw[->, dotted] (0,0) -- (0,4.4) node[right] {$\lambda_1^*$};

        \fill[yellow, opacity=0.5] (2.8, 0) -- (2.8, 1.2) -- (4, 0) -- cycle;

        \fill[yellow, opacity=0.5] (0, 2.8) -- (1.2, 2.8) -- (0, 4) -- cycle;

        \draw (2, 3) node[right] {$a^* \eqdef 2-a^{-1}$};

        \draw[line width=0.4mm] (3.1,0)--(3.1,0.9)--(4,0)--cycle;
        \draw[line width=0.4mm] (0,3.1)--(0.9,3.1)--(0,4)--cycle;
        \draw (3.4, 0.1) node[above] {$Y^\varepsilon$};
        \draw (0,3.4) node[right] {$Y^\varepsilon$};

        \draw (3.1,0) node[below] {\tiny $(a^* + \varepsilon)n$};
        \draw (4,0) node[below] {\tiny $2n$};

        \draw (0,3.1) node[left] {\tiny $(a^* + \varepsilon)n$};
        \draw (0,4) node[left] {\tiny $2n$};
    \end{tikzpicture}
    \caption{This is $\Zone_{Y, \varepsilon}$, a modified version of $\Zone_Y$ where we shrink the region by $\varepsilon n$. The positive part $\Zone_{Y, \varepsilon}^+$ is at the bottom right and the negative part $\Zone_{Y, \varepsilon}^-$ is at the top left. }\label{fig:yellow-zone}
\end{figure}
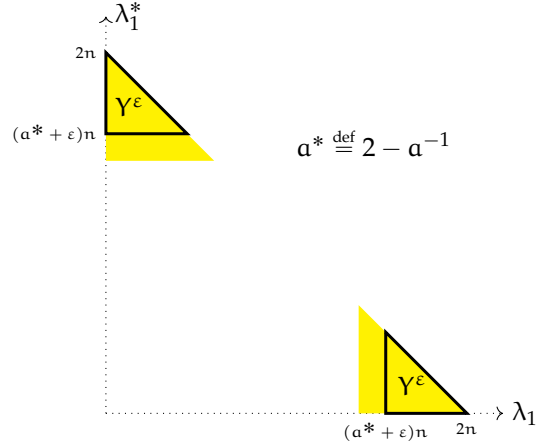

\begin{Lemma}\label{lem:bounds-on-eigenvalues-for-main-terms}
    Let $\lambda \vdash 2n$ and let $1 \leq j < a^{-1}n$. 
    \begin{enumerate}[label = (\alph*)]
        \item Suppose that $\lambda_1 = N-j$. The eigenvalue corresponding to any $(\lambda, \mu, \nu) \in \LR$ satisfies 
        \begin{equation}\label{eqn:turtle}
            1 - \frac{2aj}{N} \leq \Eig^\lambda_{\mu, \nu} \leq 1 - \frac{2bj}{N} + \frac{2ab}{N^2} j(j-1).
        \end{equation}

        \item Suppose that $\lambda_1^* = N-j$. The eigenvalue corresponding to any $(\lambda, \mu, \nu) \in \LR$ satisfies 
        \begin{equation}\label{eqn:turtle2}
            -1 + \frac{2bj}{N} - \frac{2ab}{N^2} j(j-1) + \frac{a^2 + b^2}{N} \leq \Eig^\lambda_{\mu, \nu} \leq -1 + \frac{2aj}{N} + \frac{a^2 + b^2}{N}. 
        \end{equation}
    \end{enumerate} 
\end{Lemma}

\begin{proof}
    Let $(\lambda,  \mu, \nu) \in \text{LR}$ be a Littlewood-Richardson triple. From Proposition~\ref{prop:LR-lattice-dominance-global-maximum} we have $\lambda \lhd \mu + \nu$. This implies that 
    \[
        n + \nu_1 \geq \mu_1 + \nu_1 \geq \lambda_1 \implies \nu_1 \geq n-j. 
    \]
    Thus, $\nu$ dominates $(n-j, 1^j)$. From Lemma~\ref{lem:oliver}, we get the bounds 
    \begin{align*}
        \Diag(\lambda) & \leq \Diag(2n-j, j) = 2n^2 - n(2j+1) + j(j-1) \\
        \Diag(\mu) & \leq \Diag(n) = \frac{n^2}{2} - \frac{n}{2} \\
        \Diag(\nu) & \geq \Diag(n-j, 1^j) = \frac{n^2}{2} - \frac{n}{2} (2j+1).
    \end{align*}
    Combining these diagonal sum bounds together, we get a bound on the eigenvalue 
    \[
        \mathsf{Eig}^\lambda_{\mu \nu} \leq 1 - \frac{bj}{n} + \frac{ab}{2n^2} j(j-1). 
    \]
    For the lower bound, we use Proposition~\ref{prop:LR-lattice-dominance-global-maximum} to get the inequalities 
    \[
        n + \mu_1 \geq \mu_1 + \nu_1 \geq \lambda_1 \implies \mu_1 \geq n-j. 
    \]
    This allows us to bound the diagonal statistics by 
    \begin{align*}
        \Diag(\lambda) & \geq \Diag(2n-j, 1^j) = 2n^2 - n(2j+1)\\
        \Diag(\mu) & \geq \Diag(n-j, 1^j) = \frac{n^2}{2} - \frac{n}{2}(2j+1) \\
        \Diag(\nu) & \leq \Diag(n) = \frac{n^2}{2} - \frac{n}{2}. 
    \end{align*}
    This gives us the lower bound of 
    \[
        \Eig^\lambda_{\mu, \nu} \geq 1 - \frac{aj}{n}
    \]
    This immediately yields Equation~\ref{eqn:turtle}. Equation~\ref{eqn:turtle2} follows from Lemma~\ref{lem:new-zealand}. 
\end{proof}

The following result follows immediately from Lemma~\ref{lem:bounds-on-eigenvalues-for-main-terms}. 

\begin{Lemma}\label{lem:main-term-bound-master-version}
    For sufficiently large $N$, we have the bound
    \begin{equation}
        \max |\Eig^\lambda_{\mu, \nu}| \leq 1 - \frac{bj}{n} + \frac{ab}{2n^2}j(j-1)
    \end{equation}
    where the maximum is taken over all $\lambda \vdash 2n$ satisfying $\lambda_1 = 2n-j$ or $\lambda_1^* = 2n-j$ and $(\mu, \nu) \in \LR^\lambda$. 
\end{Lemma}

\begin{Proposition}\label{prop:Zone-Yellow-Full-Bound}
    Let $c > 0$. For any $\varepsilon > 0$, there is a universal constant $C > 0$ such that for sufficiently large $N$, we have 
    \begin{equation}\label{eqn:rhode-island}
        \sum_{\lambda \in \Zone_{Y, \varepsilon}} \Omega_\lambda(t_N + cN) \leq C e^{-4bc}
    \end{equation}
\end{Proposition}

\begin{proof}
    From Lemma~\ref{lem:f-lambda-bound}(2), Lemma~\ref{lem:essential-tool}, and Lemma~\ref{lem:main-term-bound-master-version}, the left hand side of Equation~\ref{eqn:rhode-island} is \emph{bounded above} by 
    \begin{equation}\label{eqn:new-york-city}
        \Omega_{(1^{2n})}(t_N + cN) + 2 \sum_{j = 1}^{(a^{-1}-\varepsilon)n} \frac{p(j)}{j!} e^{2j \log (2n)} \left( 1 - \frac{bj}{n} + \frac{ab}{2n^2}j(j-1) \right)^{2(t_N + cN)}. 
    \end{equation}
    The left term in Equation~\ref{eqn:new-york-city} is equal to 
    \begin{align*}
        \Omega_{(1^{2n})}(t_N + cN) & = \left |\Eig^{1^{2n}}_{1^n, 1^n} \right |^{2(t_N + cN)} = \left ( 1 - \frac{a^2 + b^2}{2n} \right )^{t_N + cN}.
    \end{align*}
    Using the standard bound $1-x \leq e^{-x}$, we get 
    \begin{align*}
        \Omega_{(1^{2n})}(t_N + cN) & \leq \exp \left ( - (t_N + cN) \frac{a^2 + b^2}{2n}\right ) = O \left ( \frac{1}{N^{2b}} \right ).
    \end{align*}
    We now bound the sum 
    \[
        \sum_{j = 1}^{ \left (a^{-1}-\varepsilon \right ) n} \frac{p(j)}{j!} e^{2j \log (2n)} \left( 1 - \frac{bj}{n} + \frac{ab}{2n^2}j(j-1) \right)^{2(t_N + cN)}.
    \]
    We again apply the standard bound $1-x \leq e^{-x}$ to get an upper bound of 
    \begin{equation}\label{eqn:paris}
        \sum_{j = 1}^{\left( a^{-1}-\varepsilon \right)n} \frac{p(j)}{j!} \exp \left\{ 2j \log(2n) - 2(t_N + cN) \left( \frac{bj}{n} - \frac{ab}{2n^2}j(j-1) \right) \right\}.
    \end{equation}
    The expression in the curly braces in Equation~\ref{eqn:paris} can be simplified to
    \begin{equation}\label{eqn:canberra}
        \frac{a}{n} j(j-1) \log N - 4c \left \{ bj - \frac{ab}{2n} j(j-1) \right \}.
    \end{equation}
    The expression in the curly braces of Equation~\ref{eqn:canberra} is a quadratic equation in $j$ with negative leading term and vertex at $a^{-1}n + 0.5$. This implies that among all $1 \leq j \leq (a^{-1} - \varepsilon)n$, it is minimized when $j = 1$. This gives the bound 
    \begin{equation}\label{eqn:capybara}
        \leq e^{-4b\cdot c} \left \{ \sum_{j = 1}^{(a^{-1} - \varepsilon)n} \frac{p(j)}{j!} \exp \left\{ \frac{a}{n} j(j-1) \log N \right\} \right \}. 
    \end{equation}
    We now show that the sum in the curly braces of Equation~\ref{eqn:capybara} is \emph{bounded}. To do this, we split the sum into three different parts: 
    \begin{enumerate}
        \item[($\mathbf{P}_1$)] $1 \leq j \leq \sqrt{N/ \log N}$, 
        \item[($\mathbf{P}_2$)] $\sqrt{N / \log N} \leq j \leq N/\log N$, 
        \item[($\mathbf{P}_3$)] $N/\log N \leq j \leq (a^{-1}-\varepsilon)n$. 
    \end{enumerate}
    When summing over the first part $\mathbf{P}_1$, we have 
    \[
        \sum_{j = 1}^{\sqrt{N/\log N}} \frac{p(j)}{j!} \exp \left\{ \frac{a}{n} j(j-1) \log N \right\} \leq C \sum_{j = 1}^\infty \frac{p(j)}{j!} < \infty. 
    \] 
    When summing over parts $\mathbf{P}_2$ and $\mathbf{P}_3$, we will use the following estimate from \cite{Diaconis1981}*{Equation 3.14}: 
    \[
        \frac{p(j)}{j!} \leq \exp(-j \log j + O(j)). 
    \]
    This provides an upper bound for our summand 
    \begin{equation}\label{eqn:charlottesville}
        \exp \left \{ -j \log j + a j^2 \frac{\log N}{n} + O(j) \right \} \leq \exp \left \{ j \log N \left( \frac{aj}{n} - \frac{\log j}{\log N} + O \left( \frac{1}{\log N} \right) \right)\right \}. 
    \end{equation}
    When summing over $\mathbf{P}_2$, we can further bound the summand by 
    \begin{equation}
        \exp \left \{  j \log N \left( -\frac{1}{2} + O\left( \frac{\log \log N}{\log N} \right) \right) \right \} \leq \exp \left \{ - \frac{j}{3} \log N \right \}.
    \end{equation}
    Thus, for sufficiently large $N$, the sum over $\mathbf{P}_2$ is bounded above by a geometric series of rate less than $1$, which is finite. 
    
    When summing over $\mathbf{P}_3$, we have the bound 
    \[
        \frac{aj}{n} - \frac{\log j}{\log N} + O \left( \frac{1}{\log N} \right) \leq -a \varepsilon + O \left( \frac{\log \log N}{\log N} \right) \leq -\frac{a}{2} \varepsilon
    \]
    for sufficiently large $N$. Thus, we can bound the right hand side of Equation~\ref{eqn:charlottesville} by 
    \[
        \exp \left \{ j \log N \left( \frac{aj}{n} - \frac{\log j}{\log N} + O \left( \frac{1}{\log N} \right) \right)\right \} \leq \exp \left\{ -\frac{a \varepsilon}{2} j \log N  \right\}.
    \]
    The sum is then also bounded above by a geometric series of rate less than $1$. Collecting all of our bounds, there is some $C > 0$ such that 
    \[
        \sum_{j = 1}^{(a^{-1}-\varepsilon)n} \frac{p(j)}{j!} e^{2j \log (2n)} \left( 1 - \frac{bj}{n} + \frac{ab}{2n^2}j(j-1) \right)^{2(t_N + cN)} \leq C e^{-4b \cdot c}. 
    \]
\end{proof}

\subsubsection{}

We can now complete the proof of Theorem~\ref{thm:main-theorem-1}. 
\begin{proof}[Proof of Theorem~\ref{thm:main-theorem-1}]
    The theorem follows by combining the bounds of Corollary~\ref{cor:red-zone-i-finished-bound}, Corollary~\ref{cor:full-of-zone-red-II}, Corollary~\ref{cor:red-zone-III-full-bound}, Corollary~\ref{cor:full-red-zone-4-bound}, Proposition~\ref{prop:Blue-Zone-I-Full-Bound}, Proposition~\ref{prop:Blue-Zone-II-Full-Bound}, and Proposition~\ref{prop:Zone-Yellow-Full-Bound}. 
\end{proof}

\section{Lower Bound}\label{sec:lower-bound}

In this section, we prove the lower bound of Theorem~\ref{thm:main-theorem-1}. In Section~\ref{sec:asymptotics-of-fixed-points}, we will prove that the number of fixed points of our deck after $\frac{1}{2b}N (\log N - c)$ shuffles is Poisson in the limit. This provides a lower bound for the total variation distance in terms of the total variation distance between two Poisson random variables. In Section~\ref{sec:proof-of-lower-bound-2}, we complete the proof of the lower bound of Theorem~\ref{thm:main-theorem-1} by bounding the total variation distance by the Hellinger distance. Recall that this bound was discussed in Section~\ref{sec:distances-p-measures}.

\subsection{Asymptotics of fixed points}\label{sec:asymptotics-of-fixed-points}

For the rest of this section, let $V$ be the \textbf{permutation representation} of $\fS_N$. This is the $N$-dimensional representation with basis $e_1, \ldots, e_N$ with $\CC[\fS_N]$-module structure given by $\pi. e_i = e_{\pi(i)}$ and extended linearly. In other words, as elements of $\End(\CC^n)$, each group element $\pi$ is given by the corresponding permutation matrix. For any permutation $\pi \in \fS_N$, let $\Fix(\pi)$ be the \textbf{number of fixed points} of $\pi$. That is, it is the number of $i \in [N]$ satisfying $\pi(i) = i$. The following lemma gives a representation-theoretic interpretation of the moments of $\Fix(\pi)$ for a random permutation $\pi$. 

\begin{Lemma}\label{lem:fixed-points}
    Let $\theta : \fS_N \to [0, 1]$ be a probability distribution on $\fS_N$ and let $\vartheta \in \CC[\fS_N]$ the corresponding element in the group algebra. 
    \begin{enumerate}
        \item The expected number of fixed points of a permutation picked according to $\theta$ is the trace of $\vartheta$ on $V$.
        \item The $p\textsuperscript{th}$ moment of the number of fixed points of a permutation picked according to $\theta$ is the trace of $\vartheta$ on $V^{\otimes p}$. 
    \end{enumerate}
\end{Lemma}

\begin{proof}
    The group algebra element $\vartheta$ is given by $\sum \theta(\pi) \pi \in \CC[\fS_N]$ where the sum ranges over all permutations. Viewing $\pi$ as an element of $\End(V)$, the trace $\tr(\pi)$ is equal to the number of fixed points of $\pi$. Thus, we have 
    \[
        \tr (\vartheta) = \sum_{\pi \in \fS_N} \theta(\pi) \tr (\pi) = \sum_{\pi \in \fS_N} \theta(\pi) \Fix(\pi) = \EE_\theta [\Fix(\pi)]. 
    \]
    This proves (1). To prove (2), we also need the additional observation that $\Fix(\pi) = \tr(\pi) = \chi_V(\pi)$ where $\chi_V$ is the character of $V$. From standard character theory, the character of the tensor product $V^{\otimes p}$ is equal to $\chi_V^p$. Thus, viewing each $\pi$ as an element of $\End(V^{\otimes p})$, we have  
    \[
        \tr (\vartheta) = \sum_{\pi \in \fS_N} \theta(\pi) \tr (\pi) = \sum_{\pi \in \fS_N} \theta(\pi) \Fix(\pi)^p = \EE_\theta [\Fix(\pi)^p]. 
    \]
    This completes the proof of the Lemma. 
\end{proof}

In the next result, let $S(n, k)$ be the \textbf{Stirling number of the second kind}. This is the number of ways to partition $n$ labeled objects into $k$ unlabeled sets. The following lemma follows from~\cite{Benkart2016}*{Equation 5.5}. 

\begin{Lemma}\label{lem:permutation-module-decomposition}
    As $\fS_N$-representations, we have the decomposition
    \[
        V^{\otimes p} = \bigoplus_{\lambda} m^\lambda_p S^\lambda, \quad \text{where} \quad m^\lambda_p \eqdef \sum_{t = 0}^N S(p, t) K_{\lambda, (N-t, 1^t)}.
    \]
\end{Lemma}

We also recall the following lemma from~\cite{Teyssier2019}*{Proposition 3.2}.  
\begin{Lemma}
    Let $j \geq 1$ be a positive integer. Let $\lambda = (N-j, T)$ where $T \vdash j$ is a fixed partition. Then for sufficiently large $N$ 
    \[
        f_\lambda = \frac{N^j}{j!} f_T \left (1 + O \left ( \frac{1}{N} \right ) \right ). 
    \]
\end{Lemma}

In the next theorem, we let $\Fix_c$ be the random variable counting the number of fixed points after $K \eqdef  \frac{N}{2b}(\log N - c)$ shuffles in the biased random transposition shuffle. Our next result states that the limiting distribution of this random variable is Poisson. 

\begin{Theorem}\label{thm:main-theorem-4}
    Suppose that $b < 1$. Then, 
    \[
        \Fix_c \xrightarrow{dist} \Poiss \left (1 + \frac{1}{2}e^c \right ).
    \]
\end{Theorem}

\begin{proof}
    Since the Poisson distribution is determined by its moments \cite{DasGupta2008}*{Example 6.3}, it is enough to show that for any $p \geq 1$, the $p$th moment of $\Fix_c$ converges to the $p$th moment of a $\Poiss(1 + e^c/2)$ random variable. From Lemma~\ref{lem:fixed-points}, the $p$th moment of $\Fix_c$ is exactly the trace of $\mathscr{A}^K$ on $V^{\otimes p}$. 
    
    To compute the trace of $\cA^K$ on $V^{\otimes p}$, recall that we can decompose $V^{\otimes p}$ using Lemma~\ref{lem:permutation-module-decomposition} and Theorem~\ref{thm:LR-rule} to get
    \begin{equation}\label{eqn:decomp-decomp-decomp}
        V^{\otimes p} = \bigoplus_{\lambda \vdash N} m^\lambda_p S^\lambda = \bigoplus_{\lambda \vdash N} \bigoplus_{\mu, \nu \vdash n} m^\lambda_p c^\lambda_{\mu, \nu} \cdot \left ( S^\mu \boxtimes S^\nu \right )
    \end{equation}
    where the first equality is as $\fS_N$-modules and the second equality is as $\fS_A \times \fS_B$-modules. From the proof of Theorem~\ref{thm:main-theorem-3}, we know that $\cA$ acts by scalar multiplication by $\Eig^\lambda_{\mu, \nu}$ on each copy of $S^\mu \boxtimes S^\nu$ in Equation~\ref{eqn:decomp-decomp-decomp}. Thus, the trace of $\mathscr{A}^K$ on $V^{\otimes p}$ is exactly 
    \begin{equation}\label{eqn:value of the trace}
        \tr \left ( \mathscr{A}^K \right ) = \sum_{\lambda \vdash N} \sum_{\mu, \nu \vdash n} m^\lambda_p c^\lambda_{\mu, \nu} f^\mu f^\nu (\Eig^\lambda_{\mu, \nu})^K.
    \end{equation}
    It remains to figure out the asymptotics of Equation~\ref{eqn:value of the trace} as $N \to \infty$. 
    
    For sufficiently large $N$, Lemma~\ref{lem:kostka number example computation} implies that $m^\lambda_p = 0$ for all $\lambda$ with $\lambda_1 < N-p$. This allows us to simplify Equation~\ref{eqn:value of the trace} to 
    \begin{equation}\label{eqn:multiplicity-formula}
        \tr \left ( \cA^K \right ) = \sum_{j = 0}^p \sum_{\substack{\lambda \vdash N \\ \lambda_1 = N-j}} \left [ m^\lambda_p \left \{ \sum_{\mu, \nu \vdash n} c^\lambda_{\mu, \nu} f^\mu f^\nu (\Eig^\lambda_{\mu, \nu})^K \right \} \right ] .
    \end{equation}
    
    Since we are fixing $p$ and letting $N$ tend to infinity, there are only finitely many partitions (depending on $p$) which appear in the decomposition. These partitions are of the form $\lambda = (N-j, T)$ where $0 \leq j \leq p$, and $T$ is a partition of $j$. 
    
    Take any partition $\lambda = (N-j, T)$ of this form. We first compute the asymptotics of  
    \begin{equation}\label{eqn:focus-asymptotics-here}
        \sum_{\mu, \nu \vdash n} c^\lambda_{\mu, \nu} f^\mu f^\nu (\Eig^\lambda_{\mu, \nu})^K. 
    \end{equation}
    Suppose that $\mu, \nu$ satisfy $c^\lambda_{\mu, \nu} > 0$. Then, Lemma~\ref{lem:asymptotic-bound-on-LR-coefficient} implies that they must be of the form 
    \[
        \mu = (n-i_1, T_1), \quad \nu = (n-i_2, T_2), \quad i_1 + i_2 \leq j, \quad T_1 \vdash i_1, T_2 \vdash i_2.
    \]
    In particular, the number of non-zero summands in Equation~\ref{eqn:focus-asymptotics-here} is bounded above by a constant depending only on $p$. From Lemma~\ref{lem:diagonal-index-asymptotics}, we have the asymptotics
    \[
        \Eig^\lambda_{\mu, \nu} = 1 - \frac{B_{i_1, i_2, j}}{N} + O \left ( \frac{1}{N^2} \right )
    \]
    where  $B_{i_1, i_2, j} \eqdef i_1 a^2 + i_2 b^2 + (2j - i_1 - i_2)ab$. Using the approximation $1 - x = e^{-x} + O(x^2)$, the contribution of all copies of $S^\mu \boxtimes S^\nu$ in $S^\lambda$ to the trace is 
    \begin{align}
        c^\lambda_{\mu, \nu} f_\mu f_\nu (\Eig^\lambda_{\mu, \nu})^K & = \frac{N^{i_1 + i_2}}{2^{i_1 + i_2} (i_1)! (i_2)!} f_{T_1} f_{T_2} c^\lambda_{\mu, \nu} \cdot \frac{e^{\frac{B_{i_1, i_2, j}}{2b} c}}{N^{\frac{B_{i_1, i_2, j}}{2b}}} + o(1)\label{eqn:contribution} \\
        & = C \cdot N^{\frac{2b (i_1 + i_2) - B_{i_1, i_2, j}}{2b}} + o(1)\label{eqn:order-of-growth} \\
        & = C \cdot N^{E_{i_1, i_2, j}} + o(1),
    \end{align}
    where we define $E_{i_1, i_2, j} \eqdef \frac{2b(i_1 + i_2) - B_{i_1, i_2, j}}{2b}$. In Equation~\ref{eqn:order-of-growth} we consolidated all constants which are of constant order with respect to $N$ in the constant $C$. Note that (emma~\ref{lem:asymptotic-bound-on-LR-coefficient} tells us that $c^\lambda_{\mu, \nu}$ is of constant order depending only on $p$. 
    
    We will study the maximum value of $E_{i_1, i_2, j}$ depending on the values of $i_1, i_2, j$. It turns out that $E_{i_1, i_2, j}$ is will less than or equal to $0$, and it will be equal to zero for a single choice of the pair $(i_1, i_2)$.
    
    \begin{Lemma}\label{lem:computing-exponent}\phantom{h}
        \begin{enumerate}
            \item We have $E_{i_1, i_2, j} \leq 0$. If $E_{i_1, i_2, j} = 0$, then $i_1 = 0$, $i_2 = j$, and $B_{i_1, i_2, j} = 2jb$. 
            \item Let $i_1 = 0$ and $i_2 = j$. Then $c^\lambda_{\mu, \nu} > 0$ if and only if $\mu = (n)$ and $\nu = (n-j, T)$. If this happens, then $c^{\lambda}_{\mu, \nu} = 1$.
        \end{enumerate}
    \end{Lemma}

    \begin{proof}
        From $j \geq i_1 + i_2$, we have 
        \begin{align*}
            i_1 a^2 + i_2b^2 + (2j - i_1 - i_2) ab & \geq 2i_1 a + 2 i_2 b \geq 2b(i_1 + i_2).
        \end{align*}
        This implies that $E_{i_1, i_2, j} \leq 0$. For equality to hold, we need $j = i_1 + i_2$ and $i_1 a = b i_1$. Since $a > b$, we must have $i_1 = 0$ and $i_2 = j$. 

        When $i_1 = 0$ and $i_2 = j$, we have $\mu = (n)$ and $\nu = (n-j, T_2)$ for some $T_2 \vdash j$. From the Littlewood-Richardson rule~\cite{Macdonald2015}*{\S I.9} and Pieri's formula~\cite{Macdonald2015}*{Equation 5.16}, $c^\lambda_{\mu, \nu} > 0$ if and only if $T_2 = T_1$. Moreover, when $c^\lambda_{\mu, \nu} > 0$ it is equal to $1$. This proves the lemma. 
    \end{proof}
    From Lemma~\ref{lem:computing-exponent}, the asymptotics of Equation~\ref{eqn:focus-asymptotics-here} for any partition $\lambda = (N-j, T)$ are given by  
    \[
        \sum_{\mu, \nu \vdash n} c^\lambda_{\mu, \nu} f_\mu f_\nu \left (\Eig^\lambda_{\mu, \nu} \right )^K = \frac{f_T}{j!} \left ( \frac{e^c}{2} \right )^j + o(1).
    \]
    Substituting this formula into Equation~\ref{eqn:multiplicity-formula} and substituting the formula for $m^\lambda_p$ from Lemma~\ref{lem:permutation-module-decomposition}, we find that the trace is equal to
    \begin{align}
        \tr(\cA^K) & = o(1) + \sum_{j = 0}^p \sum_{T \vdash j} \frac{f_T}{j!} \left ( \frac{e^c}{2} \right )^j \sum_{t = j}^p S(p, t) \binom{t}{j} f_T\label{eqn:hi} \\
        & = o(1) + \sum_{t = 0}^p S(p, t) \sum_{j = 0}^t \binom{t}{j} \left ( \frac{e^c}{2} \right )^j \left [ \frac{\sum_{T \vdash j} f_T^2}{j!} \right ] \\
        & = o(1) + \sum_{t = 0}^p S(p, t) \left (1  + \frac{e^c}{2} \right )^t.\label{eqn:finally!}
    \end{align}
    In the equality in Equation~\ref{eqn:finally!}, we used Proposition~\ref{prop:sym-rep-theory-identities}. From~\cite{Haight1967}*{Equation 1.3-14}, the right hand side of Equation~\ref{eqn:finally!} is exactly the $p\textsuperscript{th}$ moment of a Poisson random variable of rate $1 + \frac{e^c}{2}$. This suffices for the proof. 
\end{proof} 

\begin{Remark}
In Theorem~\ref{thm:main-theorem-4}, it is important that $b < a$. Indeed, the condition that $b < a$ plays an important role in limiting the sources which provide the main contribution to the trace. In the case $a = b = 1$, which corresponds to random transpositions, it is shown in~\cite{Matthews1988}*{Equation 1.6} that
\begin{equation}\label{eqn:matthews-fix-limit}
    \Fix_c \xrightarrow{dist} \Poiss \left (1 + e^c \right ).
\end{equation}
This illustrates the clear difference between the $a < b$ case and the $a = b = 1$ case. In~\cite{Teyssier2019}*{Theorem 1.1}, L. Teyssier explicitly computes the limit profile for random transpositions. Explicitly, they prove that
\begin{equation}\label{eqn:limit-profile-eqn}
    d_{\TV} \left ( P^{\frac{1}{2}N (\log N - c)}, U \right ) \to d_{\TV}(\Poiss(1 + e^c), \Poiss(1))\quad \text{as } N \to \infty.  
\end{equation}
Note that the distributions in the right hand side of Equation~\ref{eqn:limit-profile-eqn} are exactly the distributions of the number of fixed points of the distributions in the left hand side of Equation~\ref{eqn:limit-profile-eqn}. From Equation~\ref{eqn:matthews-fix-limit}, the number of fixed points of $P^{\frac{1}{2}N (\log N - c)}$ is $\Poiss(1 + e^c)$ in the limit. From the classical \emph{problème des rencontres}, the number of fixed points of $\Unif_{\fS_N}$ is $\Poiss(1)$ in the limit.

Since the limiting distributions of the number of fixed points when $a = b$ versus when $a > b$ is different, we expect a different limit profile for the \emph{strictly} biased random transposition shuffle.  Explicitly, we make the following conjecture based on Theorem~\ref{thm:main-theorem-4}. 
\end{Remark}
\begin{Conjecture}[Limit profile for biased random transpositions] \label{conjecture:limit-profile2}
    Let $c \in \RR$. Then we have 
    \[
        d_{\TV} \left ( P^{\frac{1}{2b}N(\log N - c)}, U \right ) \to d_{\TV} \left ( \Poiss \left (1 + \frac{e^c}{2} \right ), \Poiss(1) \right ), \quad \text{as } N \to \infty.
    \]
\end{Conjecture}

Using Theorem~\ref{thm:main-theorem-4}, we prove the lower bound of Conjecture~\ref{conjecture:limit-profile}. 

\begin{Theorem}\label{thm:main-theorem-5}
    Let $t = \frac{N}{2b} (\log N - c)$ for any $c \in \RR$. Then 
    \[
        d_{\TV} \left ( U, P^t(\id, \bullet) \right ) \geq d_{\TV} \left (\Poiss \left (1 + \frac{1}{2} e^c \right ), \Poiss(1) \right ) + o(1)
    \]
\end{Theorem}

\begin{proof}
   For $0 \leq k \leq N$, let $\fS_N^{(k)}$ be the subset of permutations with exactly $k$ fixed points. Then we have 
    \begin{equation}\label{eqn:lower-bound-finish}
        2 \cdot d_{\TV}(U, P^t(\id, \bullet)) = \sum_{k = 0}^N \sum_{\pi \in \fS_N^{(k)}} |U(\pi) - P^t(\id, \pi)| \geq \sum_{k = 0}^N \left |U(\fS_N^{(k)}) - P^t(\id, \fS_N^{(k)}) \right |.
    \end{equation}
    The right hand side of Equation~\ref{eqn:lower-bound-finish} is (two times) the total variation distance between $\Fix_c$ and $\Fix$, where $\Fix$ is the number of fixed points of a uniformly chosen permutation. From Theorem~\ref{thm:main-theorem-4}, we have that $\Fix_c$ converges in distribution to $\Poiss(1 + e^c/2)$. The classic \emph{problème des rencontres} tells us that $\Fix$ converges in distribution to $\Poiss(1)$. Thus, we have 
    \[
        \text{RHS of Equation~\ref{eqn:lower-bound-finish}} = 2 \cdot d_{\TV} \left (\Poiss(1), \Poiss \left ( 1 + \frac{e^c}{2} \right ) \right ) + o(1).
    \]
    This suffices for the proof of the Theorem. 
\end{proof}

\subsection{Proof of the lower bound}\label{sec:proof-of-lower-bound-2}

We can finally complete the proof of the lower bound in Theorem~\ref{thm:main-theorem-1}. Indeed, the lower bound of Theorem~\ref{thm:main-theorem-5} is bounded below by Equation~\ref{eqn:poissTV>=LOWERBOUND} with $x = e^c$. This completes the proof of Theorem~\ref{thm:main-theorem-1}.

\section{Appendix}

\subsection{Calculations for the Red Zone}

\subsubsection{Technical calculations for \texorpdfstring{$\Zone_R^{II}$}{Red Zone II}}

In this section, we prove some properties about the auxiliary functions $\phi_1, \phi_2, \phi_3$, and $\phi_4$. 
\begin{Definition}
    Let $\phi_1 : \RR \to \RR$ be the function defined by 
    \begin{equation}\label{eqn:APPENDIX-D-DEFINITION}
        \phi_1(x) \eqdef \frac{(ab - b^2 - abx)^2}{16ab} + \frac{a^2 - ab}{2} x^2 + \frac{2ab-a^2}{2} x + \frac{a^2 - ab}{4}. 
    \end{equation}
\end{Definition}

\begin{Lemma}[Properties of $\phi_1$]\label{lem:APPENDIX-prop-D}
    The function $D$ satisfies the following properties:
    \begin{enumerate}
        \item For all $x$, we have 
        \[
            \phi_1'(x) = \left ( a^2 - \frac{7ab}{8} \right ) x + \frac{7ab - 4a^2 + b^2}{8}.
        \]
        \item There is $\varepsilon > 0$ depending only on $a$ and $b$ such that $\phi_1'(x) > 0$ for all $x \in (0.5-\varepsilon, \infty)$. 
        \item $\phi_1(x) > 0$ for all $x \in (0, \infty)$.
    \end{enumerate}
\end{Lemma}

\begin{proof}
    Part (1) follows after differentiating Equation~\ref{eqn:APPENDIX-D-DEFINITION}. For part (2), for $x > 0.5-\varepsilon$ we have 
    \[
        \phi_1'(x) \geq \frac{7ab + 2b^2}{16} - \left ( a^2 - \frac{7ab}{8} \right ) \varepsilon > 0
    \]
    for sufficiently small $\varepsilon$ depending only on $a$ and $b$. For part (3), for $x > 0$ we can bound 
    \begin{align*}
        \phi_1(x) & \geq \frac{a^2 - ab}{2} x^2 + \frac{2ab - a^2}{2} x + \frac{a^2 - ab}{4} \\
        & \geq \frac{a^2 - ab}{2} x^2 + \frac{ab - a^2}{2}x + \frac{a^2 - ab}{4} \\
        & = \frac{a^2 - ab}{2} \left ( x^2 - x + \frac{1}{2} \right ) > 0. 
    \end{align*}
\end{proof}

\begin{Definition}
    We define the function $\phi_2 : (0, \infty) \to \RR$ given by 
    \[
        \phi_2(x) \eqdef 2 + \frac{2}{b} \log \phi_1(x). 
    \]
    From Lemma~\ref{lem:APPENDIX-prop-D}, the function $\phi_2$ is well-defined. 
\end{Definition} 

\begin{Lemma}[Properties of $\phi_2$]\label{lem:appendix-frakD}
    The function $\phi_2$ satisfies the following properties.
    \begin{enumerate}
        \item $\phi_2'(x) > 0$ for $x \in (0.5-\varepsilon, \infty)$ for some $\varepsilon > 0$ depending only on $a$ and $b$. 
        \item $\phi_2(x) < x$ for all $x \in [0.7, 1]$
        \item $\phi_2(0.5) < 0$
    \end{enumerate}
\end{Lemma}
\begin{proof}
    From Lemma~\ref{lem:APPENDIX-prop-D}(2), there is a real number $\varepsilon > 0$ depending only on $a$ and $b$ for which $\phi_1'(x) > 0$ for all $x \in (0.5-\varepsilon, \infty)$. Then for all $x \in (0.5-\varepsilon, \infty)$, we have 
    \[
        \phi_2'(x) = \frac{2}{b} \frac{\phi_1'(x)}{\phi_1(x)} > 0. 
    \]
    Part (2) is equivalent to the inequality
    \[
        \phi_1(x) < e^{\frac{b}{2}(x-2)} \text{ for all } x \in [0.7, 1].
    \]
    To prove this inequality, we study the difference $\Delta(t) \eqdef e^{\frac{b}{2}(x-2)} - \phi_1(x)$. It is enough to prove that $\Delta'(x) < 0$ in $[0.7, 1]$ and $\Delta(1) > 0$. For $x \in [0.7, 1]$, we can bound the derivative by 
    \begin{align*}
        \Delta'(x) & = \frac{b}{2} e^{\frac{b}{2}(x-2)} - \left ( a^2 - \frac{7ab}{8} \right ) x - \frac{7ab - 4a^2 + b^2}{8} \\
        & \leq \frac{b}{2} - \left (a^2 - \frac{7ab}{8} \right ) x - \frac{7ab - 4a^2 + b^2}{8} \\
        & = - \left ( a^2 - \frac{7ab}{8} \right ) x + \frac{4b - 7ab + 4a^2 - b^2}{8}.
    \end{align*}

    It is readily verified that the right hand side is strictly negative when $x = 0.7$ and $x = 1$. Since it is linear, it is strictly negative on $[0.7, 1]$. This proves that $\Delta'(x) < 0$ for $x \in [0.7, 1]$. Now, we want to verify that 
    \[
        \Delta(1) = e^{-\frac{b}{2}} - \left ( \frac{b^3}{16a} + \frac{a}{2} \right ) > 0 \text{ for all } b \in [0, 1]. 
    \]
    After substituting $a = 2-b$, this becomes a simple exercise in single-variable calculus. We leave the details to the reader. Similarly, for part (3), it is enough to prove that 
    \[
        \frac{b(a-2b)^2}{64a} + \frac{a}{4} < e^{-b}, \quad \text{ for } b \in [0, 1]. 
    \]
    This is also a simple exercise in single variable calculus which we leave to the reader.
\end{proof}

We now want to define an inverse function for $\phi_2$. To construct this function, we will need the following easy exercise from real analysis (see ~\cite{Rudin1976}*{Exercise 5.2} and~\cite{Bredon}*{Corollary 19.9}).  

\begin{Lemma}\label{lem:rudin}
    Let $L \in \RR$ and $\varepsilon > 0$ be real numbers. Let $f : (L-\varepsilon, \infty) \to \RR$ be a differentiable function satisfying $f'(x) > 0$ for all $x \in (L-\varepsilon, \infty)$ and $\lim_{x \to \infty} f(x) = \infty$.  Then, restricting $f$ to $(L, \infty)$ induces a differentiable bijection $f|_{(L, \infty)} : (L, \infty) \to (f(L), \infty)$ with a differentiable strictly increasing inverse function $g : (f(L), \infty) \to (L, \infty)$. 
\end{Lemma}

Lemma~\ref{lem:appendix-frakD}(1) and Lemma~\ref{lem:rudin} implies that there exists some differentiable inverse function of $\phi_2 : (0.5, \infty) \to (\phi_2(0.5), \infty)$. 

\begin{Definition}
    Let $\phi_3: (\phi_2(0.5), \infty) \to (0.5, \infty)$ be the inverse function of $\phi_2 : (0.5, \infty) \to (\phi_2(0.5), \infty)$. 
\end{Definition}

\begin{Lemma}[Properties of $\phi_3$]\label{lem:appendix-frakC}
    The function $\phi_3 : (\phi_2(0.5), \infty) \to (0.5, \infty)$ satisfies the following properties. 
    \begin{enumerate}
        \item For $x \in [0.7, 1]$, we have $\phi_3(x) > x$. 
        \item $\phi_3'(x) > 0$ for all $x \in (\phi_2(0.5), \infty)$.
    \end{enumerate}
\end{Lemma}

\begin{proof}
    For part (1), suppose for the sake of contradiction that for some $x \in [0.7, 1]$ we have $\phi_3(x) \leq x$. From the monotonicity of $\phi_2$ on $(0.5, \infty)$, we know that 
    \[
        x = \phi_2(\phi_3(x)) \leq \phi_2(x).
    \]
    But this contradicts Lemma~\ref{lem:appendix-frakD}(2). Part (2) follows from Lemma~\ref{lem:rudin}. 
\end{proof}

\begin{Definition}
    We define $\phi_4: (0.5, \infty) \to (0.5, \infty)$ to be the function defined by 
    \[
        \phi_4(x) \eqdef \frac{\phi_3(x) + x}{2}. 
    \]
    From Lemma~\ref{lem:appendix-frakD}(2), this is well-defined. 
\end{Definition}

\begin{Lemma}[Properties of $\phi_4$]\label{lem:appendix-scrC}
    The function $\phi_4: [0.7, \infty) \to (0.5, \infty)$ satisfies the following properties:
    \begin{enumerate}
        \item For $x \in [0.7,1]$, we have $\phi_4(x) > x$. 
        \item For $x \in [0.7,1]$, we have $x < \phi_4(x) < \phi_3(x)$. 
        \item $\phi_4'(x) > 0$ for all $x \in (0.7, \infty)$. 
    \end{enumerate}
\end{Lemma}

\begin{proof}
    From Lemma~\ref{lem:appendix-frakC}, we have 
    \[
        \phi_4(x) = \frac{\phi_3(x) + x}{2} > \frac{x + x}{2} = x \text{ for all } x \in [0.7, 1].
    \]
    Part (2) follows from Part (1). Part (3) follows from Lemma~\ref{lem:appendix-frakC}(2). 
\end{proof}

The main result of this section is the following dynamical result. It dictates the partitions we pick in the analysis of $\Zone_R^{II}$. 
\begin{Proposition}\label{prop:appendix-dynamical-subzoning}
    Consider the sequence defined by $x_0 = 0.7$ and $x_{t+1} = \phi_4(x_t)$ for all $t \geq 0$. Then there is some finite integer $M$ depending on only $a$ and $b$ for which 
    \[
        0.7 = x_0 < x_1 < \ldots < x_{M-1} \leq 1 < x_M. 
    \]
\end{Proposition}

\begin{proof}
    For the sake of contradiction, suppose that no such finite integer $M$ exists. From Lemma~\ref{lem:appendix-scrC}(1) it must be the case that our sequence is infinitely strictly increasing in the interval $[0.7, 1]$. But an increasing sequence bounded sequence always converges to some point $L$. Since $[0.7, 1]$ is closed we must have $L \in [0.7, 1]$. Thus, we have 
    \[
        L = \lim_{t \to \infty} x_{t+1} = \lim_{t \to \infty} \phi_4(x_t) = \phi_4(L). 
    \]
    But this contradicts Lemma~\ref{lem:appendix-scrC}(1). 
\end{proof}

\subsubsection{Technical calculations for \texorpdfstring{$\Zone_R^{III}$}{Red Zone III}}

In this section, we prove bounds on $\mathbf{Q}_R(x, y)$ relevant to the analysis of the subzone $\Zone_R^{III}$. Recall that
\[
    \mathbf{Q}_R (x, y) = \frac{a^2 - b^2}{4} + \frac{a^2-ab}{2}(x^2-x) + \frac{ab-b^2}{2} (y^2-y) + \frac{ab}{2} \left ( x- \frac{xy}{2} - \frac{y^2}{2} \right ). 
\]

\begin{Lemma}\label{lem:APPENDIX-BOUND-ON-R-FUNCTION-RED-III}
    The function $\mathbf{Q}_R(x, y)$ satisfies the following inequality. 
    \begin{align*}
        \max \left\{  \max_{\substack{x \in [0.7, 1] \\ y \in [0.5, 0.7]}} \mathbf{Q}_R(x, y), \max_{\substack{x \in [0.5, 0.7] \\ y \in [0.7, 1]}} \mathbf{Q}_R(x, y) \right\} \leq \frac{a^2}{4} + \frac{3ab}{16} - \frac{b^2}{8}. 
    \end{align*}
\end{Lemma}

\begin{proof}
    For any $y \in [0.5, 0.7]$, the functions $x - \frac{xy}{2}$ and $x^2 - x$ are monotonically increasing when $x \in [0.7, 1]$. Thus, we can bound 
    \[
        \max_{\substack{x \in [0.7, 1] \\ y \in [0.5, 0.7]}} \mathbf{Q}_R(x, y) \leq \max_{y \in [0.5, 0.7]} \mathbf{Q}_R(1, y). 
    \]
    The result of substituting $x = 1 $ in $R_3$ is 
    \begin{equation}\label{eqn:spain}
    \mathbf{Q}_R(1, y) = \left ( \frac{ab - 2b^2}{4} \right ) y^2 - \left ( \frac{3ab - 2b^2}{4} \right ) y + \left ( \frac{a^2 + 2ab - b^2}{4} \right ). 
    \end{equation}
    We want to maximize Equation~\ref{eqn:spain} over $y \in [0.5, 0.7]$. When $b = \frac{2}{3}$, the equation is a linear equation in $y$ with negative slope. In this case, the maximum value is achieved when $y = 0.5$. When $b < \frac{2}{3}$, Equation~\ref{eqn:spain} is a quadratic equation in $y$ with positive leading coefficient. The $y$ value over $\RR$ which achieves the absolute minimum is 
    \[
        y_{\min} = \frac{3a - 2b}{2(a - 2b)} > 1 \quad \text{ for } b < \frac{2}{3}. 
    \]
    In this case $y = 0.5$ still achieves the maximum when we restrict $y \in [0.5, 0.7]$. When $b > 2/3$ the quadratic has negative leading coefficient. The value of $y$ which achieves the maximum is 
    \[
        y_{\max} = \frac{3a - 2b}{2(a - 2b)} < 0 \quad \text{ for } b > \frac{2}{3}. 
    \]
    Thus the maximum value is still achieved when $y = 0.5$.
    
    In all cases, $\mathbf{Q}_R(1, y)$ is maximized by $y =0.5$ over $[0.5, 0.7]$. This gives us the bound 
    \[
        \max_{\substack{x \in [0.7, 1] \\ y \in [0.5, 0.7]}} \mathbf{Q}_R(x, y)\leq \mathbf{Q}_R(1, 0.5) = \frac{a^2}{4} + \frac{3ab}{16} - \frac{b^2}{8}.
    \]
    Now we want to maximize $\mathbf{Q}_R(x, y)$ in the domain $x \in [0.5, 0.7]$ and $y \in [0.7, 1]$. For any $y \in [0.7, 1]$, the functions $x - \frac{xy}{2}$ and $x^2 - x$ are monotonically increasing functions in $x \in [0.5, 0.7]$. Thus, we can bound 
    \[
        \max_{\substack{x \in [0.5, 0.7] \\ y \in [0.7, 1]}} \mathbf{Q}_R (x, y) \leq \max_{y \in [0.7, 1]} \mathbf{Q}_R(0.7, y).
    \]
    Using $y^2-y \leq 0$ for $y \in [0.7, 1]$ we can further bound $\max_{y \in [0.7, 1]} \mathbf{Q}_R(0.7, y)$ by
    \begin{align*}
        & \leq (0.145a^2 + 0.105ab - 0.25b^2) + \frac{ab}{2} \max_{y \in [0.7, 1]} \left \{ 0.7 - 0.35 y - 0.5 y^2 \right \} \\
        & < \frac{a^2}{4} + \frac{3ab}{16}- \frac{b^2}{8}. 
    \end{align*}
    This suffices for the proof of the lemma. 
\end{proof}

\subsubsection{Technical calculations for \texorpdfstring{$\Zone_R^{IV}$}{Red Zone IV}}

In this section, we prove bounds on $\mathbf{Q}_R(x, y)$ relevant to the analysis of the subzone $\Zone_R^{IV}$. 

\begin{Lemma}\label{lem:APPENDIX-ZONE-RED-IV-BOUND-ON-Q}
    The function $\mathbf{Q}_R(x, y)$ satisfies the following inequality. 
    \[
        \max_{\substack{x \in [0.7, 1] \\ y \in [0.7, 1]}} \mathbf{Q}_R(x, y) \leq 0.25a^2 + 0.0975ab-0.145b^2. 
    \]
\end{Lemma}

\begin{proof}
    For any $y \in [0.7, 1]$, the functions $x - \frac{xy}{2}$ and $x^2 - x$ are monotonically increasing when $x \in [0.7, 1]$. Thus the maximum will be achieved when $x = 1$. We now want to maximize 
    \[
        \mathbf{Q}_R(1, y) = \frac{ab - 2b^2}{4} y^2 - \frac{3ab - 2b^2}{4} y + \frac{a^2 + 2ab - b^2}{4}
    \]
    over $y \in [0.7, 1]$. From the same analysis as in the proof of Lemma~\ref{lem:APPENDIX-BOUND-ON-R-FUNCTION-RED-III}, the maximum is achieved when $y = 0.7$. This completes the proof of the lemma. 
\end{proof}

\subsection{Calculations for the Blue Zone}
\subsubsection{Technical calculations for \texorpdfstring{$\Zone_B^I$}{Blue Zone I}}

In this section, we prove bounds on $\mathbf{Q}_B$ relevant to the analysis of $\Zone_B^I$. Recall that $a^* = 2-a^{-1}$ and 
\[
    \mathbf{Q}_B(x, y) = \frac{a^2 - b^2}{4} + \frac{ab}{2} \left( x - \frac{xy}{2} - \frac{y^2}{2} \right) + \frac{ab - b^2}{2} (y^2 - y).
\]
\begin{Lemma}\label{lem:APPENDIX-ZONE-BLUE-I-BOUNDS}
    For any $\varepsilon > 0$, there is a $\delta > 0$ depending on $\varepsilon, a, b$ such that 
    \[
        \max_{\substack{x \in [1, (a^* + \delta)] \\ y \in [0.5, 1]}} \mathbf{Q}_B(x, y) \leq 1 - \frac{b}{4} - \frac{7}{16}b^2 + \varepsilon. 
    \]
\end{Lemma}

\begin{proof}
    Let us pick $\delta = \varepsilon / ab$. For $y \in [0.5, 1]$, we have that $x  - \frac{xy}{2}$ is monotonically increasing as a function in $x$. In particular, it is maximized when $x = a^* + \delta$. Thus, the maximum is bounded above by 
    \begin{align*}
        \max_{\substack{x \in [1, (a^* + \delta)] \\ y \in [0.5, 1]}} \mathbf{Q}_B(x, y) & \leq ab \delta + \max_{y \in [0.5, 1]} \mathbf{Q}_B(a^*, y) = \varepsilon + \max_{y \in [0.5, 1]} \mathbf{Q}_B(a^*, y).
    \end{align*}
    Since $a^* = 2 - a^{-1}$, we have $aa^* = 2a - 1$. This gives us 
    \[
        \mathbf{Q}_B(a^*, y) = \left( \frac{ab - 2b^2}{4} \right) y^2 - \left( \frac{4ab - 2b^2 - b}{4} \right) y + \left( \frac{a^2 - b^2 + 4ab - 2b}{4} \right).
    \]
    When $b = 2/3$, this is a linear equation in $y$ with negative slope. Thus it is maximized when $y = 0.5$. When $b < 2/3$, we have a quadratic equation in $y$ with positive leading coefficient which is minimized at 
    \[
        y = \frac{4a-2b-1}{2(a-2b)} > 1, \quad \text{ for } b < \frac{2}{3}. 
    \]
    In this case the function is maximized at $y = 0.5$. When $b > 2/3$, we have a quadratic equation in $y$ with negative leading coefficient which is maximized at 
    \[
        y = \frac{4a-2b-1}{2(a-2b)} < 0, \quad \text{ for } b > \frac{2}{3}. 
    \]
    Thus, the function is still maximized at $y = 0.5$. This gives the bound 
    \[
        \max_{y \in [0.5, 1]} \mathbf{Q}_B(a^*, y) \leq \mathbf{Q}_B(a^*, 0.5) = 1 - \frac{b}{4} - \frac{7}{16} b^2. 
    \]
    This suffices for the proof. 
\end{proof}

\subsubsection{Technical calculations for \texorpdfstring{$\Zone_B^{II}$}{Blue Zone II}}\label{sec:technical-Blue-Zone-II}

In this section, we explore some properties of the function $\mathbf{T}_B$ and $\mathbf{P}_B$ for the analysis of $\Zone_B^{II}$. Recall the definition of $\mathbf{T}_B$.
\begin{Definition}
    Let $L_B \eqdef 1 + \frac{2}{b} \log \frac{a}{2}$. The function $\mathbf{T}_B : (L_B, \infty) \to \RR$ is defined by the equation 
    \[
        \mathbf{T}_B(x) \eqdef \sqrt{\frac{2e^{\frac{b}{2}(x-1)}-a}{ab}}.
    \]
\end{Definition}

\begin{Definition}
    We define $\mathscr{T}_B : (L_B, \infty) \to \RR$ by 
    \[
        \mathscr{T}_B (x) \eqdef \frac{\mathbf{T}_B(x) + x}{2}. 
    \]
    Note that $L_B \leq 0$ for $b \in (0, 1]$. 
\end{Definition}
\begin{Lemma}\label{lem:APPENDIX-T_B-properties}
    $\mathbf{T}_B$ satisfies the following properties:
    \begin{enumerate}
        \item For all $x \in (L_B, \infty)$, $\mathbf{T}_B'(x) > 0$ and $\mathscr{T}_B'(x) > 0$. 
        \item $\mathbf{T}_B(x) > x$ and $\mathscr{T}_B(x) > x$ for $x \in [0, a^*-1]$.
    \end{enumerate}
\end{Lemma}

\begin{proof}
    For part (1), we take the derivative of $\mathbf{T}_B^2$. This gives  
    \[
        2 \mathbf{T}_B \mathbf{T}_B' = \frac{d}{dx} \mathbf{T}_B^2 = \frac{1}{a} e^{\frac{b}{2}(x-1)} > 0. 
    \]
    Since $\mathbf{T}_B > 0$, we must have $\mathbf{T}'_B > 0$. This also implies that $\mathscr{T}_B' > 0$. 

    For part (2), the inequality for $\mathscr{T}_B$ follows from the inequality for $\mathbf{T}_B$. The inequality $\mathbf{T}_B$ is equivalent to
    \[
        2e^{\frac{b}{2}(x-1)} > ab x^2 + a, \quad \text{ for $x \in [0, a^*-1]$}.
    \]
    But this follows from 
    \[
        2e^{\frac{b}{2}(x-1)} > 2 \left( 1 + \frac{b}{2}(x-1) \right) = a + bx \geq ab x^2 + a
    \]
    whenever $x \in [0, a^*-1]$. 
\end{proof}

\begin{Lemma}\label{lem:APPENDIX-zoneB-subzone-2-sequence-exists}
    Let $\{x_t\}_{t \geq 0}$ be the sequence defined by $x_0 = 0$ and $x_t = \mathscr{T}_B(x_{t-1})$ for $t \geq 1$. Then there is some positive integer $M$ depending only on $a, b$ and a sufficiently small $\varepsilon > 0$ depending only on $a, b$ such that 
    \[
        0 = x_0 < x_1 < \ldots < x_{M-1} \leq a^*-1 < a^*-1 + \varepsilon \leq x_M. 
    \]
\end{Lemma}

\begin{proof}
    Lemma~\ref{lem:APPENDIX-T_B-properties}(2) implies that $\{x_t\}_{t \geq 0}$ is a strictly increasing sequence until it passes $a^* - 1$. 

    If the sequence $\{x_t\}_{t \geq 0}$ never becomes greater than $a^* - 1$, then it is an infinite strictly increasing sequence in the compact interval $[0, a^* - 1]$. But then it has some limit point $L \in [0, a^*-1]$ which would satisfy the equation
    \[
        L = \lim_{t \to \infty} x_{t+1} = \lim_{t \to \infty} \mathscr{T}(x_t) = \mathscr{T}(L). 
    \]
    This contradicts Lemma~\ref{lem:APPENDIX-T_B-properties}(2). This completes the proof of the lemma. 
\end{proof}

Recall the definition of $\mathbf{P}_B$. 
\begin{Definition}
    Let $\mathbf{P}_B : \RR^2 \to \RR$ be the function given by 
    \[
        \mathbf{P}_B (x, y) \eqdef \frac{a^2+3ab}{4} + \frac{ab}{2} (x^2 - 2(x+y) + xy) + \frac{ab-b^2}{4} y. 
    \]
\end{Definition}

\begin{Lemma}\label{lem:APPENDIX-boldP-B-max-bound}
    For sufficiently small $\varepsilon > 0$ depending only on $a, b$, we have the bound
    \[
        \max_{\substack{1+\alpha \leq x \leq 1 + \beta \\ y \leq 0.5}} \mathbf{P}_B(x, y) \leq \mathbf{P}_B(1+\beta, 0) = \frac{ab}{2} \beta^2 + \frac{a}{2}.
    \]
    for all $0 \leq \alpha \leq \beta \leq a^*-1+\varepsilon$. 
\end{Lemma}

\begin{proof}
    Fix some $x \in (1+\alpha, 1 + \beta)$. As a function in $y$, $\mathbf{P}_B(x, \cdot)$ is a linear function with slope 
    \[
        \frac{ab}{2}x - \frac{3ab + b^2}{4} \leq \frac{ab}{2}(a^* - 1 + \varepsilon) - \frac{3ab + b^2}{4} < 0
    \]
    for sufficiently small $\varepsilon$. Thus, the function will be maximized when $y = 0$. This simplifies our problem to maximizing
    \[
        \mathbf{P}_B(x, 0) = \frac{ab}{2}(x-1)^2 + \frac{a}{2}.
    \]
    In the interval $1 + \alpha \leq x \leq 1 + \beta$, this is clearly maximized at $x = 1 + \beta$. This suffices for the proof. 
\end{proof}

\subsection{Bounds on conjugate regions}

In this section, we prove results that imply that if the sum over a region is small, then the sum over the conjugate region is small as well. It turns out that this is applicable in all zones except for the yellow zone (the main zone). This is because we show that the conjugate region can only have an additional \emph{multiplicative} $\exp(cn)$ contribution. In the regions other than the yellow zone, this is negligible. 

\begin{Lemma}\label{lem:new-zealand}
    Let $(\lambda, \mu, \nu)$ be a Littlewood-Richardson triple with $\lambda \vdash 2n$ and $\mu, \nu \vdash n$. Then 
    \[
        \Eig^{\lambda}_{\mu, \nu} + \Eig^{\lambda^*}_{\mu^*, \nu^*} = \frac{a^2 + b^2}{2n}. 
    \]
\end{Lemma}

\begin{proof}
    This is immediate from Lemma~\ref{lem:oliver}(3).
\end{proof}

Recall that $t_N = \frac{1}{2b} N \log N$ is the predicted cutoff time for the biased random transposition. 

\begin{Lemma}\label{lem:OMEGA-is-INVOLUTIVE}
    There are universal constants $C_1, C_2 > 0$ such that for all sufficiently large $n$, we have 
    \[
        \Omega_{\lambda^*}(t_N) \leq \Omega_{\lambda}(t_N) \exp(C_1 n) + \exp(-C_2 n (\log n)^2). 
    \]
    For any $\lambda \vdash 2n$, we have $\Omega_{\lambda}(t) = \Omega_{\lambda^*}(t) \exp(O(n))$. 
\end{Lemma}
\begin{proof}
    Since $c^\lambda_{\mu, \nu} = c^{\lambda^*}_{\mu^*, \nu^*}$, we have a natural bijection between $\LR^\lambda$ and $\LR^{\lambda^*}$ simply by sending $(\mu, \nu)$ to the conjugates $(\mu^*, \nu^*)$. Since $f_\lambda = f_{\lambda^*}$ as well, we must have 
    \[
        \Omega_{\lambda^*}(t_N) = \sum_{(\mu, \nu) \in \LR^\lambda} c^\lambda_{\mu, \nu} f_\lambda f_\mu f_\nu |\Eig^{\lambda^*}_{\mu^*, \nu^*}|^{2t_N}. 
    \]
    From Lemma~\ref{lem:new-zealand}, we have 
    \[
        \left |\Eig^{\lambda^*}_{\mu^*, \nu^*} \right |^{2t_N} = \left ( \Eig^{\lambda}_{\mu, \nu} - \frac{a^2 + b^2}{2n} \right )^{2t_N}. 
    \]
    If $|\Eig^\lambda_{\mu, \nu}| > \frac{\log n}{n}$, then we have 
    \begin{align*}
        \left ( \Eig^{\lambda}_{\mu, \nu} - \frac{a^2 + b^2}{2n} \right )^{2t_N} & \leq |\Eig_{\mu, \nu}^\lambda |^{2t_N} \left ( 1 + \frac{a^2 + b^2}{2\log n} \right )^{2t_N} \\
        & \leq |\Eig^\lambda_{\mu, \nu}|^{2t_N} \exp \left ( \frac{a^2 + b^2}{\log n} t_N \right ) \\ 
        & \leq |\Eig^\lambda_{\mu, \nu}|^{2t_N} \cdot \exp (C_1 n), 
    \end{align*}
    where $C_1$ is a universal constant. If $|\Eig^\lambda_{\mu, \nu}| \leq \frac{\log n}{n}$, then we have 
    \[
        |\Eig^{\lambda^*}_{\mu^*, \nu^*}|^{2t_N} \leq \exp(-C_2' n (\log n)^2)
    \]
    where $C_2' > 0$ is a universal constant. Thus, we have 
    \[
        |\Eig^{\lambda^*}_{\mu^*, \nu^*}|^{2t_N} \leq |\Eig^\lambda_{\mu, \nu}|^{2t_N} \exp(C_1n) + \exp(-C_2' n (\log n)^2). 
    \]  
    Using this bound in $\Omega_{\lambda^*}(t_N)$ and Proposition~\ref{prop:rep-theory-facts-sym}, we get the desired inequality. 
\end{proof}

\section*{Acknowledgements}

E. N. was supported in part by NSF grant DMS-2346986. The authors thank Ezra Edelman for introducing us to the Hellinger distance. The authors also thank Johnny Gao and Laura Colmenarejo for enlightening conversations about Littlewood-Richardson coefficients and the hive model. 

%
%
%
%
%
%
%


\end{document}